\documentclass[default,12pt]{sn-jnl-modified-for-arxiv}% Math and Physical Sciences Reference Style
\usepackage{microtype}
\usepackage[round-mode=uncertainty]{siunitx}

\jyear{2021}%

\usepackage{enumitem}

\theoremstyle{thmstyleone}%
\newtheorem{theorem}{Theorem}[section]% meant for sectionwise numbers
\newtheorem{proposition}[theorem]{Proposition}%
\newtheorem{corollary}[theorem]{Corollary}
\newtheorem{lemma}[theorem]{Lemma}

\numberwithin{equation}{section}

\theoremstyle{thmstyletwo}%
\newtheorem{remark}{Remark}%

\theoremstyle{thmstylethree}%
\newtheorem{definition}{Definition}%

\newtheorem{assumption}{Assumption}
\newtheorem*{assumption*}{Assumption}

\def\calC{\mathcal{C}} % function class
\def\calX{\mathcal{X}} % domain symbol
\def\pbsI{\mathbf{I}}
\def\pbsJ{\mathbf{J}}
\def\bN{\mathbb{N}}
\def\bR{\mathbb{R}}
\def\idI{\mathbb{I}}
\def\bfA{\mathbf{A}}
\def\bfB{\mathbf{B}}

\def\bfS{\mathbf{S}}
\def\bfp{\mathbf{p}}
\def\bfu{\mathbf{u}} % shouldn't appear anywhere, perhaps once
\def\bfx{\mathbf{x}}
\def\bfy{\mathbf{y}}
\def\bfz{\mathbf{z}}

\newcommand{\ObjRelErr}{\ensuremath{\bar\varkappa}}
\newcommand{\err}{\ensuremath{\mathbf{E}}}         % absolute error matrix
\newcommand{\rerr}{\ensuremath{E}}      % relative error, a scalar
\newcommand{\trap}{\ensuremath{\wedge}}            % error is due to trapezoidal approximation
\newcommand{\tmax}{\ensuremath{\mathrm{max}}}      % the maximuam of some value, annotation
\newcommand{\tmin}{\ensuremath{\mathrm{min}}}      % the minimum of some value, annotation
\newcommand{\tint}{\ensuremath{\mathrm{int}}}      % int (annotation) like n_\tint, for "interval"
\newcommand{\tinv}{\ensuremath{\mathrm{inv}}}      % inv (annotation) for "inverse"

\newcommand{\tODE}{\ensuremath{\mathrm{ODE}}}
\newcommand{\teq}{\ensuremath{\mathrm{eq}}}
\newcommand{\tT}{\text{T}}                          % transpose
\begin{document}

\title[Sensitivity Approximation by the Peano-Baker Series]{Sensitivity Approximation by the Peano-Baker Series}

%%=============================================================%%
%% Prefix	-> \pfx{Dr}
%% GivenName	-> \fnm{Joergen W.}
%% Particle	-> \spfx{van der} -> surname prefix
%% FamilyName	-> \sur{Ploeg}
%% Suffix	-> \sfx{IV}
%% NatureName	-> \tanm{Poet Laureate} -> Title after name
%% Degrees	-> \dgr{MSc, PhD}
%% \author*[1,2]{\pfx{Dr} \fnm{Joergen W.} \spfx{van der} \sur{Ploeg} \sfx{IV} \tanm{Poet Laureate}
%%                 \dgr{MSc, PhD}}\email{iauthor@gmail.com}
%%=============================================================%%

\author[1,2]{\fnm{Olivia} \sur{Eriksson}}%\email{olivia@kth.se}

\author[2,3]{\fnm{Andrei} \sur{Kramer}}%\email{andrei.kramer@scilifelab.se}
%\equalcont{The authors are listed in alphabetical order.}

\author*[4]{\fnm{Federica} \sur{Milinanni}} \email{fedmil@kth.se}
\equalcont{Authors listed in alphabetical order.}

\author[4,5]{\fnm{Pierre} \sur{Nyquist}}%\email{pierren@kth.se}
%\equalcont{These authors contributed equally to this work.}

\affil[1]{\orgdiv{Department of Electrical Engineering and Computer Science}, \orgname{KTH Royal Institute of Technology}, \orgaddress{\city{Stockholm}, \country{Sweden}}}

\affil[2]{\orgname{Science for Life Laboratory}, \orgaddress{\city{Solna}, \country{Sweden}}}

\affil[3]{\orgdiv{Department of Neuroscience}, \orgname{Karolinska Institute}, \orgaddress{\city{Solna}, \country{Sweden}}}

\affil[4]{\orgdiv{Department of Mathematics}, \orgname{KTH Royal Institute of Technology}, \orgaddress{\city{Stockholm}, \country{Sweden}}}

\affil[5]{\orgdiv{Department of Mathematical Sciences}, \orgname{Chalmers University of Technology and University of Gothenburg}, \orgaddress{\city{Gothenburg}, \country{Sweden}}}

%%==================================%%
%% sample for unstructured abstract %%
%%==================================%%

\abstract{In this paper we develop a new method for numerically approximating sensitivities in parameter-dependent ordinary differential equations (ODEs). Our approach, intended for situations where the standard forward and adjoint sensitivity analyses become too computationally costly for practical purposes, is based on the Peano-Baker series from control theory. Using this series, we construct a representation of the sensitivity matrix $\bfS$ and, from this representation, a numerical method for approximating $\bfS$. We prove that, under standard regularity assumptions, the error of our method scales as $\mathcal{O}(\Delta t^2_{\tmax})$, where $\Delta t_{\tmax}$ is the largest time step used when numerically solving the ODE. We illustrate the performance of the method in several numerical experiments, taken from both the systems biology setting and more classical dynamical systems. The experiments show the sought-after improvement in running time of our method compared to the forward sensitivity approach. In experiments involving a random linear system, the forward approach requires roughly $\sqrt{n}$ longer computational time, where $n$ is the dimension of the parameter space, than our proposed method.}

\keywords{Sensitivity analysis, Peano-Baker series, ordinary differential equations, error analysis}

%%\pacs[JEL Classification]{D8, H51}

\pacs[MSC Classification]{65L05, 65L20}

\maketitle
\textbf{ORCID}: \href{0000-0003-0740-4318}{https://orcid.org/0000-0003-0740-4318}~(OE), \href{0000-0002-3828-6978}{https://orcid.org/0000-0002-3828-6978}~(AK), \href{0000-0003-3635-8760}{https://orcid.org/0000-0003-3635-8760}~(FM), \href{0000-0001-8702-2293}{https://orcid.org/0000-0001-8702-2293}~(PN)
\section{Introduction}\label{sec:intro}

Mathematical models are used in all areas of science and engineering to model more and more complex real-world phenomena. An important aspect of such modeling is to understand how changes and uncertainty in model parameters translate to the output of a model. The first question is the topic of \textit{sensitivity analysis}, an active research area at the intersection of several branches of mathematics and its applications; see e.g.\ \cite{Frank78, CS85, Cacuci03, SaltelliA2008GSAT, CIBN05} and references therein. Motivated by problems arising in systems biology, specifically concerning statistical inference and uncertainty quantification of models based on non-linear dynamical systems, in this paper we consider the rather classical question of local sensitivity analysis in ODE models. More precisely, we are interested in the design of efficient numerical methods for approximating the sensitivity matrix for such models. Although our motivation originally comes from wanting sensitivity approximation methods that are suitable for Markov chain Monte Carlo (MCMC) methods in the systems biology setting—see Section \ref{subsec:motivation} for more details and references—we note that the general problem of computing sensitivities in an ODE model is of great interest in a number of fields.

The general starting point is a system of ordinary differential
equations (ODEs), which we take to be of the form
\begin{equation}
\label{eq:ODEsystem}
\begin{cases}
\dot{\bfx}(t) = f(t,\bfx,\bfp)\,,\\
\bfx(t_0) = \bfx_0\,,
\end{cases}
\end{equation}
where $\bfx(t)$ represents the state of the system at time $t$, and
$\bfp$ denotes a set of model parameters\footnote{In a typical systems biology
setting there is also often an input function $\bfu(t)$ explicit in
the arguments of $f$. Because such inputs $\bfu (t)$ do not play any role in developing the numerical methods that are the focus of this paper, and for simplicity of notation, we here consider $\bfu(t)$ as embedded into $f$.}; a more precise definition,
including the state spaces of the different quantities involved, is
given in Section \ref{sec:problem}.

For such a model, it is important to understand how changes in the parameter-vector $\bfp$ affect the output $\bfx$: we are interested in computing the derivatives $S_i^{~j}(t)=\partial x_i (t) / \partial p_j$, the \textit{sensitivities} of the model, for all component-combinations $(i, j)$\footnote{In some cases we are interested in computing the sensitivities of an output function $h$ applied to the state $\bfx$, i.e. $\bar S_i^{~j}(t)=\partial h(\bfx)_i (t) / \partial p_j$. By the chain rule, we can reduce this problem to the computation of the sensitivities $S_i^{~j}$ of the state $\bfx$; in fact,  $\bar S_i^{~j}(t)=\sum_{k}\frac{\partial h(\bfx)_i}{\partial x_k}S_k^{~j}(t)$.}. The sensitivities provide local information about the parameter space that is important for a variety of tasks within modeling where this space is explored, e.g., quantifying uncertainty, finding optimal parameters, and experimental design \cite{CS85, Cacuci03, CIBN05, CE17}.

For all but very simple systems the sensitivities are not available in an (explicit) analytical form. Instead, we have to turn to different types of approximations \citep{Cacuci03, CIBN05}. The two standard approaches for obtaining numerical approximations to the sensitivities of an ODE or PDE model are (i) the \textit{forward}, or \textit{variational}, approach \cite{CS85, Cacuci03}, and (ii) the \textit{adjoint} approach \citep{Marchuk95, MAS96, CLPR03}. There is a vast literature on both methods and here we only give a brief review; for a comparison of the two approaches, see \cite{CE17}.

Between the forward and adjoint methods, forward sensitivity analysis is the more straightforward of the two. It is based on finding an ODE system satisfied by the sensitivities $S_i^{~j}(t)$, along with appropriate initial values. The solution to this system can then be approximated together with that of the original ODE for $\bfx$ \citep{Cacuci03, CE17}. It is well-known that for high-dimensional problems—$n_x\times n_p\gg 1$, where $n_x$ and $n_p$ are the dimension of the state space and parameter space, respectively—computation of the sensitivities using forward sensitivity analysis becomes slow. As a consequence, this method can be too computationally costly for certain applications, and more efficient methods are needed. Indeed, the manuals of solvers such as \textsc{cvodes} \citep{gardner2022sundials} warn against this approach for large systems~\cite{CVODES_5_7_Chap_adj_sens}, favouring adjoint methods where possible.

Adjoint sensitivity analysis is designed to ease the computation of sensitivities with respect to $\bfp$ of an objective function
\begin{align*}
    G(\bfx, \bfp) = \int _0 ^T g(\bfx, t, \bfp) dt,
\end{align*}
for some function $g$ (satisfying certain smoothness assumptions). Alternatively, the method can be used to compute sensitivities of $g(\bfx (T), T, \bfp)$, that is a function only defined at the final time $T$. The method amounts to introducing an augmented objective function from $G$ by introducing Lagrange multipliers associated with the underlying ODE, and computing the derivatives $dG/dp_j$ of this objective function; see e.g.\ \cite{CLPR03, Marchuk95, MAS96} for the details.

In contrast to the forward sensitivity analysis, the adjoint method does not rely on the actual state sensitivities, $\partial x_i / \partial p_j$, but rather on the adjoint state process. It is possible, albeit impractical, to formulate the problem so that the adjoint sensitivity analysis computes $\partial x_i / \partial p_j$, the quantities of interest in this paper: define $g(\bfx,t,\bfp)=x_i(t)$ and use the adjoint sensitivity method to compute the sensitivities of $g$ at the specific times $T=t_k$. However, this must be repeated for each component $x_i$ of the state $\bfx$ and for each time $t_k$ at which an approximation of the sensitivity of the ODE model is sought. As a result, the use of adjoint sensitivity is not recommended for this purpose, especially when (i) the dimension of the state space is large, and (ii) there are many time points $t_k$ at which the sensitivities are to be computed.

%\pn{[Needed here? Comes a bit out of sync with the rest.] If the state is not measurable directly, we model the observable quantities by means of an algebraic output function $h(\bfx(t))$ and the available experimental data points $y_i$ shall be comparable to this function. The sensitivity $\Sh_i^{~j}(t)$ of $h(\bfx(t))$ can be determined in a straightforward manner using the chain rule of differentiation.} \fm{I am ok with removing the comment on the output function.}

The forward sensitivity and the adjoint methods are both too computationally expensive for the type of applications we have in mind, in particular when there is a need to approximate sensitivities in larger models (i.e., high-dimensional state and/or parameter space).  As a first step towards resolving this issue, we introduce a new method, based on the Peano-Baker series
\cite{PeanoBaker}, referred to as the \textit{Peano-Baker Series (PBS) method}. %, that can resolve the issue of computational burden. 
It turns out that this method can be unstable when applied to stiff problems\footnote{\emph{Stiffness} has no
  rigorous definition, here we treat the term as meaning that we need
  an algorithm of \emph{high order} and small \emph{step sizes}; see, e.g., \cite{HT93}.}, which are common in systems biology applications. We therefore modify the PBS method by adding a stiffness detection mechanism and refining the time steps accordingly: this leads to the \textit{PBS algorithm with refinement (PBSR)}. Additionally, we introduce an algorithm that we call the \textit{Exponential (Exp) algorithm}, which is cheaper, though in general less accurate, and has the advantage of being unconditionally stable. Because we will mainly work with stiff problems, the PBSR and Exp algorithm are the methods that we will implement in practice.
  
The methods we propose utilize two different approximation techniques, that differ in their stability properties, accuracy and
computational time. By the use of these two techniques, we can handle various dynamical properties that arise in different parts of the state and parameter space associated with the ODE system, e.g. stiffness and transient versus equilibrium behaviour. We show that
by combining these techniques we construct numerical methods to
approximate the sensitivity matrix with a computational time and an
accuracy that are suitable for use within MCMC methods, which is our main motivation for the sensitivity approximation (see Section~\ref{subsec:motivation}). 
Using the proposed methods
in the context of MCMC sampling is, however, ongoing work and in this
paper we consider the methods solely from a numerical analysis
perspective.

The remainder of the paper is organised as follows. In Section \ref{subsec:motivation} we expand briefly on our particular motivation for studying the problem of computing sensitivities in an ODE system. This is followed by an outline of the notation used throughout the paper (Section \ref{subsec:notation}). In Section~\ref{sec:problem} we give a precise formulation of the ODE
model and the associated sensitivities of interest. Next, in
Section~\ref{sec:method} we derive analytical representations of the
sensitivities in two different cases: when the solution is at
equilibrium (Theorem~\ref{thm:constantMat}), and in the general case
(Theorem~\ref{thm:generalODESol}). In Section~\ref{sec:genSolution}
the analytical representations are used to propose a new method, the PBS method, outlined in
Algorithm~\ref{alg:approx}, for approximating sensitivities in the ODE
setting considered in this paper. Before analysing the method, we discuss
related stability issues in Section~\ref{sec:stability}, where we introduce the PBSR and the Exp algorithms (Algorithms~\ref{alg:PBSR} and \ref{alg:ExpAlg}, respectively), and analyze their properties. The main
theoretical result of the paper is the error analysis of the PBS and PBSR methods, presented in
Section~\ref{sec:errEst} in Theorem~\ref{thm:error} and in Corollary~\ref{cor:errorPBSR}, respectively; the corresponding proofs are given in
Section~\ref{sec:proofErrEst}. In Section~\ref{sec:numerical} we
present some numerical experiments that showcase the performance of
the proposed methods. In particular, we compare our methods with the forward
sensitivity approach and show superior performance. We end the paper
with a discussion in Section~\ref{sec:conclusion}.

\subsection{Motivation}
\label{subsec:motivation}
Our interest in the problem originally stems from uncertainty quantification for models of
intracellular pathways, such as those studied
in~\cite{Eriksson19}. Forward sensitivity analysis via \textsc{cvodes} is
prohibitively slow for such systems; for example, when sensitivity analysis is added for the model used in~\cite{Eriksson19} for modelling the CaMKII system (see also Section \ref{sec:numerical} for additional details on this model), the simulation time increases roughly 100-fold, compared to simulations without sensitivity analysis.

In the systems biology setting, the state process $\bfx(t)$
corresponds to the concentration of different compounds internal to
the cell (e.g. proteins, protein complexes, or ions), and the
parameters $\bfp = (p_1, \dots , p_{n_p})$ encode e.g. reaction
  rate coefficients of the chemical reactions involved in the system. In this setting, we are interested in uncertainty quantification for the unknown parameters within a Bayesian framework. One of the standard methods for
uncertainty quantification involves sampling from posterior
distributions, which can be carried out via MCMC
methods \cite{GelmanAndrew2013BDAT}. We implement some methods from the MCMC class in the R package UQSA\footnote{\textbf{U}ncertainty \textbf{Q}uantification and \textbf{S}ensitivity \textbf{A}nalysis} \cite{kramer2023uqsa}. % available at \href{github.com/icpm-kth/uqsa}{https://github.com/icpm-kth/uqsa}.
%The UQSA repository also contains examples of the systems biology models (PKA\footnote{called AKAR4 and AKAP79 (bigger model)} and CaMKII) on which we perform uncertainty quantification.

In general, MCMC sampling often requires one to compute different quantities that will involve the sensitivity matrix $\bfS$, such as the parameter derivatives of the log-likelihood function. 
%the log-likelihood function and its parameter derivatives. 
As a motivating example, 
recent developments in the MCMC community propose to take a
differential-geometric perspective on sampling, exploiting potential
underlying concepts related to statistical inference \cite{GiroCal,
  TSA20, BFXKG18}. This in turn relies on an appropriate choice of the metric
tensor on the parameter space. A good choice of metric tensor \cite{Rao45,GiroCal} is the expected Fisher information. For a generic random variable $Y$
(e.g. a model for the observation of data) and associated parameter
$\theta$, with conditional probability density function $p_{Y \mid \theta }$, the expected
Fisher information is defined as:
\begin{equation}
  F (\theta) = - E_{Y \mid \theta} \left[ \frac{\partial ^2}{ \partial \theta ^2} \log \left( p_{Y \mid\theta} (Y \mid \theta) \right)  \right]\,.   \label{eq:FisherInformation}
\end{equation}
In the case of a Gaussian measurement error model, %i.e., 
%$Y\sim \mathcal{N}(\mu,\sigma)$, 
the distribution of each
(independent) observable $Y_{ij}$ has two scalar parameters:
$\mu_{ij}$ and $\sigma_{ij}$, where $i$ is the state index and $j$
the measurement time-point
index.

When combined with an ODE model, the entries of the Fisher information $F$
then contain the sensitivities of the model alongside the parameters of the
noise model~\citep{Kramer2016,GiroCal}:
\begin{equation}
  \label{eq:F}
  %F(\theta)_i^{~j} = \sum_{m,l} \Sh_m^{~i}(t_l) \sigma_{ml}^{-2}  \Sh_m^{~j}(t_l)\,,
   F(\theta)_i^{~j} = \sum_{m,l} S_m^{~i}(t_l) \sigma_{ml}^{-2}  S_m^{~j}(t_l)\,,
\end{equation}
where $l$ is the index of measurement times within a \emph{time
  series}.

The quality of the sensitivity estimates in \eqref{eq:F}, and possibly the log-likelihood and its derivatives, affects
the efficiency of MCMC in terms of the convergence speed. Even an
efficiently implemented MCMC method, such as the simplified manifold MALA (\textsc{smmala}, see \cite{GiroCal}), will
require repeatedly approximating the sensitivities for all time
points $l$, possibly millions of times for different values of
$\theta$. Because the sensitivities $S_i^{~j}(t)$ are required in \eqref{eq:F}, and also
to compute the gradient of the log-likelihood function, the adjoint
sensitivity analysis method is not appropriate for this type of setting.

Our main goal is to construct a method that can approximate the
sensitivities\footnote{The sensitivities must be approximated for all time points $\{t_l\}$ that appear in the
  likelihood function, i.e., any phase of the time span under consideration: initial,
  transient and equilibrium (possibly a limit cycle).} \emph{after} the initial value
problem for $\bfx(t)$ has been solved independently (not
simultaneously)\footnote{To evaluate the likelihood function $p(Y\mid\theta)$, we need to solve
the initial value problem \eqref{eq:ODEsystem} for $\bfx$. This
solution has to be as accurate as possible, as it directly affects the
inferred shape of the posterior distribution. For this reason, we can
safely assume that in all the error terms in this manuscript, the
relative error of $\bfx$ will be comparatively small. In practice, we
will use the best available solvers and set tight tolerances when
solving for $\bfx(t)$.}. This decoupling of the two problems allows us to use
any numerical solver, in almost any programming language, rather than
being tied to a solver that is built to include forward
  sensitivity analysis (such as \textsc{cvodes}). Furthermore, the problems we encounter in
systems biology are often stiff, and
so are the associated sensitivity equations. To avoid slowing down numerical solvers,
we would like the stiffness of the sensitivity equations (i.e., the ODE
system for $\bfS$) to not affect the solver's step size for the ODE system for $\bfx$.

In MCMC sampling, when there is an underlying ODE system, the sensitivities are used to inform the MCMC update step. A more precise approximation of the sensitivity matrix will lead to better—as in more aligned with the true transition kernel, based on the true sensitivity matrix $\bfS$—proposals for the next state of the underlying Markov chain. This in turn should lead to faster convergence to equilibrium of the chain in the case of MCMC algorithms such as such as \textsc{smmala} and \textsc{rmhmc}, meaning
fewer MCMC iterations are needed to reach convergence of the
Markov chain. However, precise approximations of the sensitivity matrix lead to a higher cost
per iteration, increasing the cost of the update step, and thus the entire MCMC scheme.  In this paper we propose new methods for approximating the
sensitivity matrix that are
more efficient for our
purposes.
The trade-off between accuracy
and efficiency of the sensitivity approximation should lead to an
increased overall performance of the MCMC method\footnote{Such a comparison is outside the scope of this manuscript.}. We are planning to implement the sensitivity approximation methods proposed in this paper into the MCMC methods available in the R package UQSA, and investigate the impact on performance.

\subsection{Notation}
\label{subsec:notation}
The following notation is used. Elements in $\bR ^n$ are denoted by bold font, e.g. $\bfx, \bfy, \bfz$, whereas $x,y,z$ denote elements of $\bR$; note that the relevant dimension $n$ will change between different vectors. With some abuse of notation, $0$ denotes the zero element in $\bR ^n$ for any $n\geq 1$. The identity matrix is denoted by $\idI$, or $\idI_n$ whenever we want to emphasise the size $n \times n$. For a matrix $\bfA$, the matrix exponential $e^\bfA$ is defined in the usual way:
\begin{equation*}
    e^\bfA = \sum _{h \geq 0} \frac{1}{h!} \bfA ^h.
\end{equation*}

Objects with a hat, such as $\hat f$ or $\hat \bfx$, refer to approximations.
For $k \in \bN$ and an open set $A$ in $\bR ^n$, $\calC^k (A{;}\bR^m)$ is the space of continuous functions on $A$, taking values in $\bR ^m$, with continuous partial derivatives up to the $k$th order; for $m \geq 2$, the latter is defined in the usual way through natural projections. For a compact set $K$, $\calC^k(K{;}\bR ^m)$ refers to functions that are $\calC^k(A{;}\bR ^m)$ on some open neighbourhood $A$ of $K$.

For a function $f: \bR ^n \to \bR ^m$, $f_l$ denotes the $l$-th entry of the value of $f$, with $l=1,\dots,m$, and $\nabla f$ denotes the Jacobian matrix of $f$.

For a normed space $(\calX, \|\cdot\|)$, $\calC^0([0,T];\calX)$ denotes the space of continuous functions from $[0,T]$ to $\calX$, and $\|\cdot \| _{[0,T]}$ denotes the supremum norm over this space: for $g \in \calC^0 ([0,T]{;}\calX)$,
\begin{align*}
    \|g\|_{[0,T]} = \sup _{t \in [0,T]} \|g (t) \|.
\end{align*}
Unless otherwise stated, we will let $\|\cdot\|$ denote an arbitrary matrix norm; because we are working on a finite-dimensional space, the specific choice of matrix norm is not important for the results of this paper (discussed more in Section~\ref{sec:errEst}).

When listing the arguments of a function, we
omit nested repetitions of arguments: when $t$, $\bfx$, and $\bfp$
appear together in an argument list, they always belong to the same
time-slice and parametrization. For example, in \eqref{eq:ODEsystem} $f(t,\bfx,\bfp)$ means the
same as $f(t,\bfx(t),\bfp)$, or indeed $f(t,\bfx(t,\bfp),\bfp)$.

Additional notation that is particular to this work is defined as needed, particularly in the error analysis in Sections~\ref{sec:errEst}-\ref{sec:proofErrEst}.

\section{Problem Formulation}
\label{sec:problem}

The general problem of interest is to compute sensitivities for the solution of the parametrised initial value problem (IVP)
\begin{equation}
\label{eq:ODE}
\begin{cases}
\dot{\bfx}(t) = f(t,\bfx,\bfp)\,,\\
\bfx(t_0) = \bfx_0\,,\\
\end{cases}
\end{equation}
with $\bfx (t)\in\bR ^{n_x}$, and $\bfp \in \bR ^{n_p}$. As noted in Section~\ref{sec:intro}, although our interest in this problem comes from a desire to conduct uncertainty quantification for models in systems biology, the problem is much more general and arises in a wide variety of areas within applied mathematics. We therefore work in a general framework for the remainder of this paper.

We assume the existence and uniqueness of solutions:
\begin{assumption}
The function $f: \bR \times \bR ^{n_x} \times \bR ^{n_p} \to \bR ^{n_x}$ is globally Lipschitz continuous in the state variable coordinate, for all values of $t$ and $\bfp$, and uniformly $C^1( \bR ^{n_p}; \bR ^{n_x})$ in the parameter coordinate for all values of $t$ and $\bfx$.
\end{assumption}
This is in accordance with the Picard-Lindelöf theorem for existence and uniqueness of the solution to the initial value problem with $f$ and $x_0$. 

\begin{definition}
Given a solution $\bfx$ of the ODE system \eqref{eq:ODE}, the \textit{sensitivity} at time $t$ of the $l$-th component of $\bfx$ with respect to parameter $p_i$ is defined as
\begin{equation}
  S_{l}^{~i}(t)=\frac{\partial x_l(t)}{\partial p_i},\quad l=1,\dots,n_x,\;i=1,\dots,n_p.\label{eq:Sli}
\end{equation}
\end{definition}
By differentiating \eqref{eq:ODE} with respect to $\bfp$, we obtain that the time derivative of the sensitivities can be expressed in matrix form as
\begin{equation}
\label{ODE_S}
    \dot{\bfS}(t) = \nabla_x f(t,\bfx,\bfp)\cdot \bfS(t) + \nabla_pf(t,\bfx,\bfp),
\end{equation}
where $\dot{\bfS}(t), \bfS(t), \nabla_p f(t,\bfx,\bfp)\in\mathbb{R}^{n_x\times n_p}$ are matrices with elements $(l,i)\in\{1,\dots,n_x\}\times\{1,\dots,n_p\}$ given, respectively, by

  \begin{equation*}
\dot{S}_{l}^{~i}(t)\,,\quad\quad S_{l}^{~i}(t)\quad\quad \text{and}\quad\quad\frac{\partial f_{l}(t,\bfx,\bfp)}{\partial p_i}\,,
\end{equation*}
and $\nabla_xf(t,\bfx,\bfp)\in\mathbb{R}^{n_x\times n_x}$ with entries
\begin{equation*}
\nabla_xf(t,\bfx,\bfp)_{l}^{~j}=\frac{\partial f_{l}(t,\bfx,\bfp)}{\partial x_j}, \quad (l,j)\in\{1,\dots,n_x\}\times\{1,\dots,n_x\}\,.
\end{equation*}

Because the initial condition $\bfx_0$ does not depend on $\bfp$, we have $S_{l}^{~i}(t_0)=\frac{\partial x_{0,l}}{\partial p_i}=0$ and equation \eqref{ODE_S} is a linear ODE system with initial condition $\bfS_{\mid t=t_0}=\bfS_0=0\in\mathbb{R}^{n_x\times n_p}$.

Let $\{t_k\}_{k=1}^K$, with $0=t_0 < t_1 < \dots < t_K=T$, be the time instants at which we want to compute the sensitivity matrix $\bfS$, and let $\bfS_k\in\bR^{n_x\times n_p}$ denote the exact sensitivity matrices at times $t_k$, with $k=0,\dots,K$. We can then formulate $K$ ODE problems in an iterative fashion,
\begin{equation}
\label{sensitivityODE}
    \begin{cases}
    \dot{\bfS}(t) = \nabla_xf(t,\bfx,\bfp)\cdot \bfS(t) + \nabla_pf(t,\bfx,\bfp)\\
    \bfS(t_{k})=\bfS_{k}
    \end{cases},\quad k=0,\dots,K-1,
\end{equation}
where the $(k+1)$th problem in the sequence is used to determine the solution at time step $t_{k+1}$, given the sensitivity matrix at the previous time step $t_k$ as initial condition. We will adopt a slight abuse of notation and refer to $\bfS$ as the sensitivity matrix, and similar for the corresponding approximations, although to be precise it is a matrix-valued function.

The matrix of coefficients $\nabla_xf$ and the forcing term $\nabla_pf$ in \eqref{sensitivityODE} both depend on $t, \bfx(t)$, and $\bfp$, and computing the solution $\bfS(t)$ therefore becomes a difficult task in general. We address this problem by developing a new, efficient algorithm for approximating the sensitivity matrix $\bfS$ at times $\{t_k\}_{k=1}^K$.

A comment on the sequence $\{ t_k\} _{k=1} ^K$ of times at which $\bfS$ is to be evaluated is in place. In Section~\ref{subsec:motivation} we briefly discussed the specific application of approximating the sensitivity matrix of an ODE in the context of uncertainty quantification in biological models. There, as in many other application areas, it is natural to consider a sequence of time instants $t_i, i=1,\dots, n_t$, that represent times of measurements of experimental data. However, in experimental settings it is rather common to have measurements only at a small number $n_t$ of times. In order to obtain a good numerical approximation of \eqref{sensitivityODE}, we therefore need a finer collection of time steps. For this reason, even when there are physical time instants to consider, we introduce a finer collection of time steps $\tilde{t}_0,\dots,\tilde{t}_K$, such that it contains all the time steps at which measurements 
are available, i.e.\ $\{t_1,\dots,t_{n_t}\}\subset \{\tilde{t}_0,\dots,\tilde{t}_K\}$; the inclusion is enforced because in this experimental setting the goal is to approximate the sensitivity matrix $\bfS$ at the $n_t$ measurement times. Moreover, we let $\tilde{t}_0$ be equal to the time of the initial condition: $\tilde{t}_0=t_0$. To simplify and to conform to the notation used in \eqref{sensitivityODE}, we drop the tilde from the new time steps $\{\tilde{t}_k\}_{k=0}^K$, which in the following are denoted by $t_0,\dots,t_K$.

\section{Analytical Solutions for the Sensitivity Matrix $\bfS$}
\label{sec:method}
In this section we derive analytical representations for the sensitivity matrix $\bfS(t)\in\mathbb{R}^{n_x\times n_p}$ along the trajectory of the solution $\bfx(t)$ of the original ODE system \eqref{eq:ODE}. To simplify notation, we will here denote $\nabla_x f (t,\bfx, \bfp)$ and $\nabla_p f(t,\bfx, \bfp)$ as time dependent matrices $\bfA(t)\in\mathbb{R}^{n_x\times n_x}$ and $\bfB(t)\in\mathbb{R}^{n_x\times n_p}$, respectively, defined on a closed interval $I\subset\bR$ such that $[t_0;t_K]\subset I$. The system \eqref{sensitivityODE} then takes the form
\begin{equation}
\label{linODE}
    \begin{cases}
    \dot{\bfS}(t) = \bfA(t)\cdot \bfS(t) + \bfB(t),\\
    \bfS(t_{k})=\bfS_{k},
    \end{cases}
\end{equation}
which in general represents a first order inhomogeneous linear differential equation with non-constant coefficient matrix $\bfA(t)$ and forcing term $\bfB(t)$. Such systems are well-understood from a theoretical point of view \cite{BrockettRogerW.2015Fdls, RughWilsonJ1993Lst} and there exist a variety of numerical methods for solving them \cite{HNW93, HW10, GH10, SauerTim2006Na}.
Although efficient numerical methods are readily available for the $n_x-$dimensional system in $\bfx(t)$ \eqref{eq:ODE}, the potentially high-dimensional nature of the associated sensitivity problem \eqref{linODE} renders such methods inefficient in that setting. Indeed, since $\bfS\in\mathbb{R}^{n_x\times n_p}$, the total dimension of the ODE system for the sensitivity is in general much higher than the dimension $n_x$ of \eqref{eq:ODE}, in particular for large values of $n_x$. To remedy this, we approach the problem of approximating $\bfS(t)$ by exploiting some results from the theory for linear ODEs; some are versions of well-known results and we include them here to make the paper self-contained.

In order to derive solutions of \eqref{linODE}, we start by
considering a solution $\bfx (t)$ that is in an equilibrium point,
$\bfx (t) \equiv \bfx_{\teq}$. This leads to a constant matrix of
coefficients $\bfA(t)\equiv \bfA$ and constant forcing term
$\bfB(t)\equiv \bfB$, thus the system \eqref{linODE} takes the form
\begin{equation}
\label{constantCoefficients}
    \begin{cases}
    \dot{\bfS}(t) = \bfA\cdot \bfS(t) + \bfB\,,\\
    \bfS(t_{k})=\bfS_{k}\,.
    \end{cases}
\end{equation}
The solution of this system is obtained using well-known results from ODE theory; we include a proof for completeness.

\begin{theorem}
\label{thm:constantMat}
    Let $\bfA\in\mathbb{R}^{n_x\times n_x}$, $\bfB\in\mathbb{R}^{n_x\times n_p}$ and $\bfS_k\in\mathbb{R}^{n_x\times n_p}$. The solution $\bfS\in \calC^1(I{;}\mathbb{R}^{n_x\times n_p})$ at a time $t\in\bR$ of the first order linear ODE system \eqref{constantCoefficients} is given by
    \begin{equation}
    \label{ssSolution}
    \bfS(t) = e^{(t-t_k)\cdot \bfA}\cdot \Biggl(\bfS_{k} +\int_{t_k}^te^{(t_k-s)\bfA}\bfB ds\Biggr).
\end{equation}
In particular, if $\bfA$ is invertible, \eqref{ssSolution} is equivalent to
 \begin{equation}
      \label{ssSolutionAinv}
      \bfS(t) = e^{(t-t_k)\cdot \bfA}\cdot \Biggl(\bfS_{k} +\biggl(\idI-e^{-(t-t_k)\cdot \bfA}\biggr)\bfA^{-1}\bfB\Biggr).
 \end{equation}
\end{theorem}

\begin{proof}
    We show that the matrix-valued function $\bfS$ defined in \eqref{ssSolution} solves the ODE system \eqref{constantCoefficients}.

    The time derivative $\dot\bfS$ of the function $\bfS$ defined in \eqref{ssSolution} is given by
    \begin{align*}
        \dot \bfS(t) &= \bfA \cdot e^{(t-t_k)\cdot \bfA}\cdot \Biggl(\bfS_{k} +\int_{t_k}^te^{(t_k-s)\bfA}\bfB ds\Biggr) + e^{(t-t_k)\cdot \bfA}\cdot e^{(t_k-t)\bfA}\bfB\\
        &=\bfA\cdot \bfS(t)+\bfB.
    \end{align*}
    Moreover, the value of $\bfS$ at time $t=t_k$ is $\bfS(t_k)=\bfS_k$. This proves that $\bfS$ solves the ODE system \eqref{constantCoefficients}.

    Finally, note that when $\bfA$ is invertible,
    \begin{align*}
     \int_{t_k}^te^{(t_k-s)\bfA}\bfB ds&=\int_{t_k}^te^{(t_k-s)\bfA}\bfA\bfA^{-1}\bfB ds=\int_{t_k}^t\frac{d}{ds}\left(-e^{(t_k-s)\bfA}\right)ds\bfA^{-1}\bfB\\
     &=\left(\idI-e^{-(t-t_k)\bfA}\right)\bfA^{-1}\bfB.
 \end{align*}
    It immediately follows that, when $\bfA$ is invertible, \eqref{ssSolution} is equivalent to \eqref{ssSolutionAinv}.
\end{proof}

From a control-theoretic perspective, the matrix exponential $e^{(t-t_k)\cdot \bfA}$ corresponds to the so-called \textit{state-transition matrix} $\Phi(t{;}t_k)$ in the specific case of a constant matrix of coefficients $\bfA$. In general, we consider the first order homogeneous linear ODE system
\begin{equation}
    \label{homo}
\dot{\bfS}(t)=\bfA(t)\cdot \bfS(t),
\end{equation}
with a time dependent matrix of coefficients $\bfA(t)$, assumed to be continuous on the closed time interval $I$. In this more general setting, we define the state-transition matrix $\Phi(t{;}t_k)$ associated with $\bfA(t)$ as the matrix-valued function that, given any initial condition $\bfS(t_k)=\bfS_k$, allows us to represent the solution of \eqref{homo} at time $t\in I$ as
\begin{equation*}
\bfS(t) = \Phi(t{;}t_k)\cdot \bfS_k.
\end{equation*}
The state-transition matrix $\Phi(t{;}s)$ is a well-studied object; see \cite{RughWilsonJ1993Lst},  \cite[Chap.~1, Sect.~3]{BrockettRogerW.2015Fdls} for a detailed description and properties. In particular, in this paper we make repeated use of the the following proposition, which is a combination of well-known results (see e.g.\ \cite{BrockettRogerW.2015Fdls, RughWilsonJ1993Lst}); we omit the proof.
\begin{proposition}
\label{prop:Phi}
The matrix-valued function $\Phi:I\times I\to\bR^{n_x\times n_x}$ satisfies the following properties for $s,t\in I$:
\begin{enumerate}
    \item\label{itm:first} $\Phi(s{;}s) = \idI,$
    \item\label{itm:second} $\Phi(t{;}s)^{-1} = \Phi(s;t),$
    \item\label{itm:third} for a fixed $s\in I$, the matrix function $\Phi(\cdot{;}s):I\to\bR^{n_x\times n_x}$ satisfies the IVP
    \begin{equation*}
        \begin{cases}
        \frac{d}{dt}\Phi(t{;}s)=\bfA(t)\cdot\Phi(t{;}s),\\
        \Phi(s{;}s)=\idI.
        \end{cases}
    \end{equation*}
\end{enumerate}
\end{proposition}
In the case of a constant matrix of coefficients $\bfA$, the exponential matrix $e^{(t-s)\cdot \bfA}$ is the state-transition matrix associated with $\bfS$ \cite[Chap.~1, Sect.~5]{BrockettRogerW.2015Fdls}, and as such it satisfies properties~\ref{itm:first}-\ref{itm:third}.
For time-dependent coefficient and forcing-term matrices $\bfA (t)$ and $\bfB (t)$, we can formulate the solution for the first order linear ODE system \eqref{linODE} in terms of the matrix-valued function $\Phi(\cdot{;}t_k): I \to \bR ^{n_x \times n_p}$ associated with $\bfA(t)$.
\begin{theorem}
\label{thm:generalODESol}
     Let $\bfA(\cdot)\in \calC^0(I{;}\mathbb{R}^{n_x\times n_x})$, $\bfB(\cdot)\in \calC^0(I{;}\mathbb{R}^{n_x\times n_p})$ and $\Phi(\cdot{;}\cdot)\in \calC^1(I\times I{;}\bR^{n_x\times n_x})$ the state-transition matrix associated with $\bfA(t)$. Let $t_k\in I$ and $\bfS(t_k)=\bfS_k\in\mathbb{R}^{n_x\times n_p}$. Then, the solution at time $t\in I$ of the first order inhomogeneous linear ODE system \eqref{linODE} is given by
    \begin{equation}
    \label{solutionODE}
     \bfS(t) = \Phi(t{;}t_{k})\cdot\Biggl(\bfS_{k} + \int_{t_{k}}^{t}\Phi(t_k{;}\tau)\bfB(\tau)d\tau\Biggr).
\end{equation}
\end{theorem}

\begin{proof}
We show that the matrix-valued function $\bfS$ defined as \eqref{solutionODE} satisfies the ODE system \eqref{linODE}.

By differentiating $\bfS$ in time and using properties \ref{itm:second} and \ref{itm:third} of Proposition~\ref{prop:Phi} we obtain
\begin{align*}
    \dot\bfS(t)&=\left(\frac{d}{dt}\Phi(t{;}t_{k})\right)\cdot\Biggl(\bfS_{k} + \int_{t_{k}}^{t}\Phi(t_k{;}\tau)\bfB(\tau)d\tau\Biggr)+\Phi(t{;}t_{k})\cdot\Biggl(\Phi(t_k{;}t)\bfB(t)\Biggr)\\
    &=\bfA(t)\cdot \Phi(t{;}t_{k})\cdot\Biggl(\bfS_{k} + \int_{t_{k}}^{t}\Phi(t_k{;}\tau)\bfB(\tau)d\tau\Biggr) + \idI\cdot \bfB(t)\\
    &=\bfA(t)\cdot\bfS+\bfB(t).
\end{align*}
Therefore, the function $\bfS$ in \eqref{solutionODE} satisfies $\dot\bfS(t)=\bfA(t)\cdot\bfS+\bfB(t)$.

Moreover, using Property~\ref{itm:first} of Proposition~\ref{prop:Phi},
\begin{align*}
     \bfS(t_k) = \Phi(t_k{;}t_{k})\cdot\Biggl(\bfS_{k} + \int_{t_{k}}^{t_k}\Phi(t_k{;}\tau)\bfB(\tau)d\tau\Biggr)=\idI\cdot\bfS_k,
\end{align*}
thus $\bfS(t_k)=\bfS_k$. This completes the proof.

\end{proof}

Theorem~\ref{thm:generalODESol} shows that, if the state-transition matrix $\Phi(t{;}t_k)$ can be computed (for all relevant values $t_k, t$), then the solution for the general ODE system \eqref{linODE} can be obtained, which in turn solves the problem \eqref{sensitivityODE} for the sensitivity matrix. An explicit expression for $\Phi(t{;}s)$ is given by a result from control theory: as proved in \cite{PeanoBaker}, the state-transition matrix $\Phi(t{;}s)$ associated with the continuous coefficient matrix $\bfA(t)$ can be expressed in terms of the \textit{Peano-Baker series}:
\begin{equation}
\label{PBS}
    \Phi(t{;}s) = \sum_{n=0}^{\infty}\pbsI_n(t{;}s),
\end{equation}
where the summands are defined recursively as
\begin{align}
    & \pbsI_0(t{;}s)=\idI, \nonumber\\
    & \pbsI_{n+1}(t{;}s) = \int_{s}^t\bfA(\tau)\pbsI_n(\tau{;}s)d\tau. \label{recursiveTerm}
\end{align}

In \cite{PeanoBaker} it is also proved that, assuming $\|\bfA(\cdot)\|$ locally integrable on the closed time interval $I$ containing $s$ and $t$, the series \eqref{PBS} converges compactly on $I$. The proof is carried out by showing that the sequence of partial sums $\left\{\sum_{n=0}^N \pbsI_n(t{;}s)\right\}_{N\ge1}$ is Cauchy on every compact subset $J\subset I$.

\begin{remark}
\label{remarkI_id}
We note that if $t=s$, then the integration interval in \eqref{recursiveTerm} has Lebesgue measure zero, from which it follows that $\pbsI_{n}(s{;}s)=0$ for $n \ge 1$ and $\Phi(s{;}s)=\pbsI_{0}(s{;}s)=\idI$.
\end{remark}

\begin{remark}
As previously mentioned, for a constant matrix of coefficients $\bfA$ the state-transition matrix $\Phi(t{;}s)$ is given by $e^{(t-s)\cdot \bfA}$. In this case, it is possible to move the constant matrix $\bfA$ outside the integral in \eqref{recursiveTerm}. This leads to the recursive definition
\begin{equation*}
  \pbsI_n(t{;}s)=\frac{(t-s)^n}{n!}\cdot \bfA^{n},
\end{equation*}
and thus
\begin{equation*}
\Phi(t{;}s)=\sum_{n=0}^{\infty}\frac{(t-s)^n}{n!}\cdot \bfA^{n} = e^{(t-s)\cdot \bfA}\,.
\end{equation*}

\end{remark}
Consider the special case where the solution $\bfx$ of the underlying ODE \eqref{eq:ODE} has reached an equilibrium point $\bfx _{\teq}$. At equilibrium we have
\begin{equation}
  f(t,\bfx_{\teq},\bfp)=0\,.\label{eq:equilibrium}
\end{equation}
In this case, the gradients $\nabla_xf(t,\bfx,\bfp)$ and $\nabla_p f(t,\bfx,\bfp)$ have constant arguments $(t,\bfx_{\teq},\bfp)$, and thus become constant matrices. To simplify notation, we will omit their arguments and denote these matrices as $\nabla_xf:=\nabla_xf(t,\bfx_{\teq},\bfp)$ and $\nabla_pf:=\nabla_pf(t,\bfx_{\teq},\bfp)$.

Under this assumption on $\bfx$, the ODE system \eqref{sensitivityODE} for $\bfS$ is a first order linear ODE system with constant matrix of coefficients and constant forcing term. The solution is therefore provided by Theorem~\ref{thm:constantMat}, with $\bfA = \nabla _x f$ and $\bfB = \nabla _p f$.
\begin{corollary}
\label{cor:Sequi}
Suppose that the solution $\{ \bfx (t) \} _{t \in I}$ of \eqref{eq:ODE} at time step $t_k$ has reached an equilibrium point $\bfx _{\teq}$. Assume that $\nabla_xf(t,\bfx_{\teq},\bfp)$ is invertible.
The solution of \eqref{sensitivityODE}, at time step $t_{k+1}$ is then given by
\begin{equation}
    \label{solutionAtEquilibrium}
    \bfS(t_{k+1})=e^{\Delta t_k \cdot \nabla_xf}\cdot \biggl(\bfS_k+(\idI-e^{-\Delta t _k\cdot \nabla_xf})\cdot\left(\nabla_xf\right)^{-1}\cdot\nabla_pf\biggr).
\end{equation}
where we denote $\Delta t_k = t_{k+1} - t_k$.
\end{corollary}
In the following we will refer to \eqref{solutionAtEquilibrium} as the \textit{exponential formula}, as it makes use of the exponential matrix for the state-transition matrix.

This exact solution can be used as an approximation for $\bfS(t_{k+1})$ also when $\bfx (t)$ is not at an equilibrium point, with the approximation improving as $\bfx (t)$ moves closer to an equilibrium point.  %Understanding the performance of this approximation is left for future work; 
%In our implementations we will use \eqref{solutionAtEquilibrium} as approximation to speed up computations further when some criteria (explained in Section~\ref{sec:numerical}) are fulfilled.

In the following section we will show how we can use Theorem~\ref{thm:generalODESol} to design a formula to approximate the sensitivity matrix $\bfS(t)$ in the general case (when the solution $\bfx$ is not at equilibrium).

\section{Numerical Approximations of $\bfS$ based on the Peano-Baker Series}
\label{sec:genSolution}
For the problem of finding the sensitivities $\bfS$ for the solution of the parametrised IVP \eqref{eq:ODE}, Theorem~\ref{thm:generalODESol} gives a general representation of $\bfS$.
Equipped with this representation, and Corollary~\ref{cor:Sequi} for the special case when $\bfx$ has reached an equilibrium point, we are now ready to construct a numerical method for approximating the sensitivity matrix $\bfS$.

In general we cannot use the assumption of the solution of \eqref{eq:ODE} being at equilibrium, as in Section~\ref{sec:method}. That is, in general the matrix of coefficients $\nabla _x f (t,\bfx, \bfp)$ and the forcing term $\nabla _p f(t,\bfx, \bfp)$ are not constant matrices. In fact, in the transient of the dynamical system \eqref{eq:ODE}, there could be drastic variations of these objects. Furthermore, not all systems converge to an equilibrium point (where we can eventually use \eqref{solutionAtEquilibrium} to compute $\bfS(t_{k+1})$). For example, it is not unusual in systems biology to have the solution $\bfx(t)$ of \eqref{eq:ODE} converge, as $t \to \infty$, to a limit cycle. In such cases, the matrices $\nabla _x f (t,\bfx, \bfp)$ and $\nabla _p f (t,\bfx, \bfp)$ are in general never constant. In principle, the trajectory $\bfx(t)$ could also show a chaotic behaviour, although this is rare in biological systems.

As previously mentioned, the $n_x-$dimensional ODE problem \eqref{eq:ODE} can be solved rather efficiently with standard off-the-shelf ODE solvers. We therefore assume to have available approximations $\hat{\bfx} _k$ of $ \bfx(t_k)$ at time steps $t_k,\;k=1,\dots,K$.
For approximating the sensitivity matrix $\bfS$, we want to take advantage of the available $\hat{\bfx}_k$'s and avoid to exploit again an ODE solver for approximating $\bfx(t)$ at intermediate times $t \notin \{t_0, \dots, t_K\}$.
In addition to a solution $\bfx$ of \eqref{eq:ODE}, we assume to have access to the exact expression for the gradients $\nabla_x f(t,\bfx,\bfp)$ and $\nabla_p f(t,\bfx,\bfp)$.

For approximating $\bfS$, we propose a method that is based on approximating the state-transition matrix $\Phi(t{;}s)$ starting from the Peano-Baker series \eqref{PBS}. More specifically, the method can be divided into approximating the integrals in \eqref{solutionODE} and \eqref{recursiveTerm}, for which we can apply numerical integration, and approximating the series \eqref{PBS} itself, for which we use a truncation that is in a sense optimal (explained in more detail below and in Section~\ref{sec:proofErrEst}).

For approximating the integrals, because we know approximations of the values $\bfx_k$ for $k=1, \dots, K$, the endpoints of the integration intervals, we apply the trapezoidal rule.
In each of the recursive integrals $\pbsI_{n}$ ($n\ge1$) in \eqref{recursiveTerm}, the integrands are of the form
\begin{equation}
\label{integrand}
    \nabla_x f(\tau,\bfx,\bfp)\cdot\pbsI_{n-1}(\tau{;}s),
\end{equation}
with $s=t_k$ when we want to compute $\bfS(t_{k+1})$ given $\bfS(t_k)$ (based on \eqref{solutionODE}). The trapezoidal rule requires us to evaluate \eqref{integrand} for $\tau = t_k$ and $\tau=t_{k+1}$. To this end, consider the approximated solutions $\hat \bfx _k$ and $\hat \bfx_{k+1}$ at time steps $t_k$ and $t_{k+1}$, respectively, given by the ODE solver. Inserting these into the gradients results in the approximations
\begin{equation}
\label{Jac_x}
\begin{split}
  \nabla_x\hat{f}_k &:=\nabla_xf(t_k,\hat \bfx _k,\bfp) \approx \nabla_xf(t_k,\bfx,\bfp)\,, \\
  \nabla_x\hat{f}_{k+1} &:=\nabla_xf(t_{k+1},\hat \bfx_{k+1},\bfp)\approx \nabla_xf(t_{k+1},\bfx,\bfp)\,.
\end{split}
\end{equation}

Moreover, we recall that $\pbsI_{0}(t_k{;}t_k)=\idI$ (from Property~\ref{itm:first} of Proposition~\ref{prop:Phi}) and $\pbsI_{n\ge1}(t_k{;}t_k)=0$ (from Remark~\ref{remarkI_id}). Combining these with the trapezoidal rule gives the approximations,
\begin{equation}
  \label{approxIntegrals}
  \begin{split}
    \hat{\pbsI}_{0,k} &:= \idI\,,\\
    \hat{\pbsI}_{1,k} &:= \frac{\Delta t_k}{2}\cdot\bigl(  \nabla_x\hat{f}_k + \nabla_x\hat{f}_{k+1}\bigr)\,,\\
    \hat{\pbsI}_{n+1,k} &:= \frac{\Delta t_k}{2}\cdot\bigl( 0 + \nabla_x\hat{f}_{k+1}\cdot \hat{\pbsI}_{n,k} \bigr),\quad n \ge 2\,,
  \end{split}
\end{equation}
for $\pbsI_0(t_{k+1}{;}t_k)$, $\pbsI_1(t_{k+1}{;}t_k)$ and $\pbsI_{n+1}(t_{k+1}{;}t_k), n \ge 2$, respectively.

A study of the error introduced by the various approximations, presented in detail in Sections~\ref{sec:errEst}-\ref{sec:proofErrEst}, shows that when the trapezoidal rule is used for the integrals in the Peano-Baker series, it is reasonable to truncate the series at $n=2$. That is, for $k=0,1,\dots,K-1$, we approximate $\Phi(t_{k+1}{;}t_k)$ by
\begin{equation}
    \label{PhiApproximation}
\hat{\Phi}(t_{k+1}{;}t_k) := \hat{\pbsI}_{0,k} + \hat{\pbsI}_{1,k} + \hat{\pbsI}_{2,k}.
\end{equation}
In order to treat the integral in \eqref{solutionODE}, where the integrand is
  \begin{equation*}
\Phi(t_k{;}\tau)\cdot\nabla_pf(\tau,\bfx,\bfp)\,,
\end{equation*}
a similar strategy can be used. Similar to \eqref{Jac_x}, we define the approximations of the gradients with respect to $\bfp$ as
\begin{gather}
\label{Jac_p}
    \nabla_p\hat{f}_k:=\nabla_pf(t_k,\hat \bfx_k,\bfp)\approx \nabla_pf(t_k,\bfx,\bfp),\\
    \nabla_p\hat{f}_{k+1}:=\nabla_pf(t_{k+1},\hat \bfx_{k+1},\bfp)\approx \nabla_pf(t_{k+1},\bfx,\bfp).\nonumber
\end{gather}
By the same reasoning as for \eqref{PhiApproximation}, we can truncate the Peano-Baker series at $n=2$ and approximate the terms of the truncated series by
\begin{align*}
    &\pbsI_0(t_{k}{;}t_{k+1}) = \idI = \hat{\pbsI}_{0,k},\\
    &\pbsI_{1}(t_{k}{;}t_{k+1}) \approx \frac{-\Delta t_k}{2}\bigl(  \nabla_x\hat{f}_k + \nabla_x\hat{f}_{k+1}\bigr)=-\hat{\pbsI}_{1,k},\\
    &\pbsI_{2}(t_{k}{;}t_{k+1})\approx \frac{-\Delta t_k}{2}\bigl( 0 + \nabla_x\hat{f}_{k+1}\cdot (-\hat{\pbsI}_{1,k}) \bigr)=\hat{\pbsI}_{2,k}.
\end{align*}
The resulting approximation $\hat \Phi (t_k{;} t_{k+1}) $ of $\Phi (t_k{;} t_{k+1})$ is then defined as
\begin{equation*}\hat{\Phi}(t_k{;}t_{k+1}):=\hat{\pbsI}_{0,k} - \hat{\pbsI}_{1,k} + \hat{\pbsI}_{2,k}.
\end{equation*}
Note that the minus sign in front of the second term is due to the integration interval being reversed compared to \eqref{PhiApproximation}.

Combining the approximation of $\Phi (t_k{;} t_{k+1})$ with the trapezoidal rule, and the property $\Phi(t_k{;}t_k)=\idI$ (Property~\ref{itm:first} of Proposition~\ref{prop:Phi}), the integral in \eqref{solutionODE} can be approximated as
\begin{equation*}
  \int_{t_{k}}^{t_{k+1}}\Phi(t_{k}{;}\tau)\nabla_pf(\tau,\bfx,\bfp)d\tau \approx \frac{\Delta t_k}{2}\cdot\bigl(\nabla_p\hat{f}_k + \hat{\Phi}(t_k{;}t_{k+1})\cdot \nabla_p\hat{f}_{k+1}\bigr).
\end{equation*}
Next, we combine this approximation with $\hat \Phi(t_{k+1}{;}t_k)$ to obtain an approximation of the sensitivity matrix $\bfS$ at time $t_{k+1}$, given $\bfS(t_k)$:
\begin{align}
\label{eq:Approx}
    \bfS(t_{k+1})\approx  \hat{\Phi}(t_{k+1}{;}t_k)\cdot\biggl(\bfS(t_k)+\frac{\Delta t_k}{2}\cdot\bigl(\nabla_p\hat{f}_k + \hat{\Phi}(t_k{;}t_{k+1})\cdot \nabla_p\hat{f}_{k+1}\bigr)\biggr).
\end{align}
The right-hand side of \eqref{eq:Approx} is the approximation, based on the Peano-Baker series, that we use for the solutions $\bfS$ of the ODE problems \eqref{sensitivityODE}. In the remainder of the paper, we will refer to \eqref{eq:Approx} as the \textit{Peano-Baker series (PBS) formula}, as it uses an approximation of the state-transition matrix based on the Peano-Baker series.

We are now ready to introduce the Peano-Baker Series (PBS) algorithm to approximate the sensitivity matrix $\bfS$ of the solution $\bfx$ of the ODE problem \eqref{eq:ODE}. The PBS algorithm, which is described in Algorithm~\ref{alg:approx}, is based on the two formulas to approximate $\bfS$ that we have presented earlier: the \textit{PBS formula} \eqref{eq:Approx} is used in the general case, and the \textit{exponential formula} \eqref{solutionAtEquilibrium} at equilibrium.

\begin{algorithm}
 \caption{Peano-Baker Series (PBS) algorithm for the sensitivity matrix $\bfS$} \label{alg:approx}
\begin{algorithmic}[1]
\Require $\hat{\bfx}_k\approx \bfx(t_k),\quad k=1,\dots,K$
\Ensure $\hat{\bfS}_1,\dots,\hat{\bfS}_K$
\State $\nabla_x\hat{f}_k = \nabla_xf(t_k,\hat{\bfx}_k,\bfp)$
\State $\nabla_p\hat{f}_k = \nabla_pf(t_k,\hat{\bfx}_k,\bfp)$
\State $\hat{\bfS}_0=0\in\mathbb{R}^{n_x\times n_p}$
\For{$k\in\{0,\dots,K-1\}$}
 	\State $\Delta t_k = t_{k+1}-t_k$
  \If{$\nabla_xf\approx $ const \normalfont{\textbf{and}} $\nabla_pf\approx $ const}
  	\State   $\hat{\bfS}_{k+1}=e^{\Delta t_k\cdot\nabla_x\hat{f}_k}\cdot\bigl(\hat{\bfS}_k+\bigl(\mathbb{I}-e^{-\Delta t_k\cdot\nabla_x\hat{f}_k}\bigr)\cdot(\nabla_x\hat{f}_k)^{-1}\cdot\nabla_p\hat{f}_k\bigr)$
\Else
   \State $\hat{\pbsI}_{0,k} = \idI_{n_x}$
   \State $\hat{\pbsI}_{1,k} = \frac{\Delta t_k}{2}\cdot\bigl(  \nabla_x\hat{f}_k + \nabla_x\hat{f}_{k+1}\bigr)$
   \State $\hat{\pbsI}_{2,k} = \frac{\Delta t_k^2}{4}\cdot\bigl(\nabla_x\hat{f}_{k+1}\cdot(  \nabla_x\hat{f}_k + \nabla_x\hat{f}_{k+1}\bigr)\bigr)$
   \State $\hat{\Phi}(t_{k+1}{;}t_k)=\hat{\pbsI}_{0,k}+ \hat{\pbsI}_{1,k}+\hat{\pbsI}_{2,k}$
   \State $\hat{\Phi}(t_{k}{;}t_{k+1})=\hat{\pbsI}_{0,k}- \hat{\pbsI}_{1,k}+\hat{\pbsI}_{2,k}$
   \State $\hat{\bfS}_{k+1} = \hat{\Phi}(t_{k+1}{;}t_k)\cdot\bigl(\hat{\bfS}_k+\frac{\Delta t_k}{2}\cdot\bigl(\nabla_p\hat{f}_k + \hat{\Phi}(t_k{;}t_{k+1})\cdot \nabla_p\hat{f}_{k+1}\bigr)\bigr)$
\EndIf
\EndFor

\end{algorithmic}
\end{algorithm}
In Section~\ref{subsec:implementation} we discuss some criteria for the if-statement in Algorithm~\ref{alg:approx}, i.e. when it is proper to use the \textit{exponential formula} \eqref{solutionAtEquilibrium} 
to approximate $\bfS(t_{k+1})$.

\begin{remark}
    Note that the exact solution \eqref{solutionODE} to the ODE system \eqref{linODE} can be rewritten as
    \begin{align*}
        \bfS(t) = \Phi(t{;}t_{k})\cdot\bfS_{k} + \int_{t_{k}}^{t}\Phi(t{;}\tau)\bfB(\tau)d\tau.
    \end{align*}
    Based on this formula we can define the following alternative approximation for the sensitivity matrix $\bfS$:
    \begin{align}
    \label{alternativeApproximation}
         \bfS(t_{k+1})\approx\hat{\Phi}(t_{k+1}{;}t_k)\cdot{\bfS}(t_k)+\frac{\Delta t_k}{2}\cdot\bigl(\hat{\Phi}(t_{k+1}{;}t_{k})\nabla_p\hat{f}_k +\nabla_p\hat{f}_{k+1}\bigr).
    \end{align}
    A method based on this alternative approximation leads to a global error of the same order as the PBS method (Algorithm~\ref{alg:approx}),  which is based on the PBS formula \eqref{eq:Approx}. We do not investigate further the use of \eqref{alternativeApproximation} for constructing numerical methods in this paper. Rather, we will only consider the PBS formula \eqref{eq:Approx} and, without confusion, ``PBS algorithm'' refers to Algorithm~\ref{alg:approx}.
\end{remark}

\section{Algorithms to Approximate $\bfS$ for Stiff Problems}
\label{sec:stability}

The PBS algorithm presented in Section \ref{sec:genSolution} has the disadvantage that the PBS formula~\eqref{eq:Approx}, which the method relies on, can be unstable in certain regions when applied to stiff problems. Moreover, stiffness of the relevant ODEs is common in the applications that first motivated our interest in methods for computing sensitivities. To overcome this potential stability issue, the PBS approximation should be implemented on small time steps; in general, these time steps should be (much) smaller than the quantities $\Delta t_k=t_{k+1}-t_k$ returned by the ODE solver applied to the system~\eqref{eq:ODE}. With this in mind, we design a stiffness detection mechanism and refine the time steps where the PBS formula~\eqref{eq:Approx} would otherwise be unstable. To distinguish this new method, with the stiffness detection, from the PBS algorithm (Algorithm~\ref{alg:approx}), we refer to it as the \textit{PBS algorithm with refinement (PBSR)}. An alternative way to overcome the issue of stiffness is to never use the PBS formula~\eqref{eq:Approx} and, instead, use the exponential formula~\eqref{solutionAtEquilibrium}, which is unconditionally stable. We refer to the  corresponding algorithm as the \textit{Exponential (Exp) algorithm}.

Because of the stability concerns, the PBSR and the Exp algorithms are the methods that are implemented in practice, rather than the PBS algorithm (Algorithm~\ref{alg:approx}). In the following sections we will give precise definitions of the PBSR and Exp algorithms, along with their properties, and discuss the computational cost associated with them.

\subsection{The PBSR and Exp Algorithms}

The PBS algorithm relies on the the PBS formula~\eqref{eq:Approx}, which is based on a truncation of the Peano-Baker series~\eqref{PBS}. Because of the truncation, the PBS formula~\eqref{eq:Approx} can be unstable. Therefore, in the case of stiff problems, it can only be applied on ``small enough'' time steps.

The approach to approximate the sensitivity matrix $\bfS$ that we propose in this paper consists in first solving the ODE system~\eqref{eq:ODE}, and then approximate the sensitivities given the numerical ODE solution $\{\hat\bfx_k\}$ and the time steps $\{t_k\}$ returned by the ODE solver. However, the time intervals $[t_k,t_{k+1}]$ provided by the ODE solver can be too large for the PBS formula~\eqref{eq:Approx}. A natural idea is then to detect what time intervals $[t_k,t_{k+1}]$ are too large and refine them into $n_{\tint}$ uniformly spaced sub-intervals $\{[t_{k,i},t_{k,i+1}]\}_{i=1}^{n_{\tint}}$ of size $\Delta t_k/n_{\tint}$, with $t_{k,i}=t_k + (i-1)\Delta t_k/n_{\tint}, i=1,\dots,n_\tint+1$, where we can apply the PBS formula~\eqref{eq:Approx}. To implement \eqref{eq:Approx}, we also need an approximation of the state $\bfx$ on the finer time grid. For this purpose we linearly interpolate the  ODE solution $\hat\bfx$ on the refinement; i.e. for $s\in(t_k{;}t_{k+1})$ we approximate $\bfx(s)\approx \hat\bfx_k+(s-t_k)\frac{\hat{\bfx}_{k+1}-\hat{\bfx}_k}{\Delta t_k}$.

The number  $n_{\tint}$ of sub-intervals should be large enough to guarantee that the PBS formula~\eqref{eq:Approx} is properly used, i.e., that the size $\Delta t_k/n_{\tint}$ of the sub-intervals is sufficiently small. In Section \ref{sec:numerical} we describe how we determine a good choice of $n_\tint$ in our numerical examples in order to overcome stiffness.

 Introducing a finer time grid comes with an increase in computational cost. To reduce this additional computational burden, on time intervals $[t_k,t_{k+1}]$ where $n_\tint$ becomes prohibitively large, we avoid the refinement. For such intervals, rather than relying on the PBS formula \eqref{eq:Approx}, which would come with stability issues, we use the exponential formula \eqref{solutionAtEquilibrium}: although in general less accurate than the PBS formula, an advantage of the exponential formula is that it is stable independently of the value of $\Delta t_k$. The use of the exponential formula, as opposed of refining and using \eqref{eq:Approx}, is triggered in two situations, where the second one was already used in Algorithm~\ref{alg:approx}:
\begin{enumerate}
    \item[(i)] if $n_{\tint}$ is ``too large'', and therefore the refinement would lead to a prohibitive computational cost;
    \item[(ii)] if $\nabla_x f$ and $\nabla_p f$ are approximately constant on $[t_k,t_{k+1}]$, and therefore the use of the exponential formula \eqref{solutionAtEquilibrium} is motivated by Corollary~\ref{cor:Sequi}.
\end{enumerate}
The method just described is the PBS algorithm with refinement (PBSR) and its pseudo-code is presented in Algorithm~\ref{alg:PBSR}. 
In Section~\ref{sec:numerical} we describe the definitions of the conditions (i), that is the stiffness detection mechanism, and (ii) that are implemented in our numerical experiments; see specifically Section \ref{subsec:implementation}. This mechanism allows us to overcome the problem of stiffness, making the method robust. However, the exact definitions and implementation of the two conditions can vary, e.g., depending on what applications one has in mind or how one values speed (i.e., computational cost) vs.\ accuracy.

\begin{algorithm}
 \caption{Peano-Baker Series algorithm with Refinement (PBSR) for the sensitivity matrix $\bfS$ \label{alg:PBSR}}
\begin{algorithmic}[1]
\Require $\hat{\bfx}_k\approx \bfx(t_k),\quad k=1,\dots,K$
\Ensure $\hat{\bfS}_1,\dots,\hat{\bfS}_K$
\State $\nabla_x\hat{f}_k = \nabla_xf(t_k,\hat{\bfx}_k,\bfp)$
\State $\nabla_p\hat{f}_k = \nabla_pf(t_k,\hat{\bfx}_k,\bfp)$
\State $\hat{\bfS}_0=0\in\mathbb{R}^{n_x\times n_p}$
 \For{$k\in\{ 0,\dots,K-1\}$}
 	\State $\Delta t_k = t_{k+1}-t_k$
 \State estimate $n_\tint$
  \If{ $n_{\tint}$ ``too large'' \quad \normalfont{\textbf{or}} \quad $\nabla_xf\approx $ const \normalfont{\textbf{and}} $\nabla_pf\approx $ const }
		\State $\hat{\bfS}_{k+1}=e^{\Delta t_k\cdot\nabla_x\hat{f}_k}\cdot\bigl(\hat{\bfS}_k+\bigl(I-e^{-\Delta t_k\cdot\nabla_x\hat{f}_k}\bigr)\cdot\nabla_x\hat{f}_k^{-1}\cdot\nabla_p\hat{f}_k\bigr)$
	\Else
		\State $\hat{\bfS}_{\text{temp}} = \hat{\bfS}_{k}$
		\For{$h\in\{0,\dots,n_{\tint}-1\}$}
			\State $t_a = t_k + \frac{h}{n_{\tint}}(t_{k+1}-t_k)$
			\State $t_b = t_k + \frac{h+1}{n_{\tint}}(t_{k+1}-t_k)$
			\State $\Delta t = t_b - t_a$
			\State $\hat{\bfx}_a = \hat{\bfx}_k + \frac{h}{n_{\tint}}(\hat{\bfx}_{k+1}-\hat{\bfx}_k)$
			\State $\hat{\bfx}_b = \hat{\bfx}_k + \frac{h+1}{n_{\tint}}(\hat{\bfx}_{k+1}-\hat{\bfx}_k)$
			\State $\nabla_x\hat{f}_a = \nabla_xf(t_a,\hat{\bfx}_a,\bfp)$
			\State $\nabla_x\hat{f}_b = \nabla_pf(t_b,\hat{\bfx}_b,\bfp)$
			\State $\nabla_p\hat{f}_a = \nabla_pf(t_a,\hat{\bfx}_a,\bfp)$
			\State $\nabla_p\hat{f}_b = \nabla_pf(t_b,\hat{\bfx}_b,\bfp)$
			\State $\hat{\pbsI}_{1} = \frac{\Delta t}{2}\cdot\bigl(  \nabla_x\hat{f}_a + \nabla_x\hat{f}_{b}\bigr)$
			\State $\hat{\pbsI}_{2} = \frac{\Delta t^2}{4}\cdot\bigl(\nabla_x\hat{f}_{b}\cdot(  \nabla_x\hat{f}_a + \nabla_x\hat{f}_{b}\bigr)\bigr)$
			\State $\hat{\Phi}(t_b{;}t_a)=I+ \hat{\pbsI}_{1}+\hat{\pbsI}_{2}$
			\State $\hat{\Phi}(t_{a}{;}t_{b})=I- \hat{\pbsI}_{1}+\hat{\pbsI}_{2}$
			\State $\hat{\bfS}_{\text{temp}} = \hat{\Phi}(t_{b}{;}t_a)\cdot\bigl(\hat{\bfS}_{\text{temp}}+\frac{\Delta t}{2}\cdot\bigl(\nabla_p\hat{f}_a + \hat{\Phi}(t_a{;}t_b)\cdot \nabla_p\hat{f}_b\bigr)\bigr)$
		\EndFor
		\State $\hat{\bfS}_{k+1} = \hat{\bfS}_{\text{temp}}$
	\EndIf
 \EndFor
 \end{algorithmic}
\end{algorithm}

While the PBS formula~\eqref{eq:Approx} must be implemented on small time steps when the problem is stiff, the exponential formula~\eqref{solutionAtEquilibrium} is stable regardless of the size of the time intervals $[t_k,t_{k+1}]$. Thus, another way of handling the problem of stiffness consists in using the exponential formula \eqref{solutionAtEquilibrium} at each time step, without refining the time intervals. This corresponds to treating the Jacobians \textit{as if} they were constant on each time interval $[t_k,t_{k+1}]$, and therefore, by Corollary~\ref{cor:Sequi}, we can express the (approximate) solution for the sensitivity matrix using the exponential formula \eqref{solutionAtEquilibrium}. The corresponding method is the Exp algorithm~\cite{Kramer2016}, whose pseudo-code is displayed in Algorithm~\ref{alg:ExpAlg}.

\begin{algorithm}
 \caption{Exponential (Exp) algorithm for the sensitivity matrix $\bfS$ \label{alg:ExpAlg}}
 \begin{algorithmic}
 \Require $\hat{\bfx}_k\approx \bfx(t_k),\quad k=1,\dots,K$
\Ensure $\hat{\bfS}_1,\dots,\hat{\bfS}_K$
\State $\nabla_x\hat{f}_k = \nabla_xf(t_k,\hat{\bfx}_k,\bfp),$ \quad $\nabla_p\hat{f}_k = \nabla_pf(t_k,\hat{\bfx}_k,\bfp)$
\State $\hat{\bfS}_0=0\in\mathbb{R}^{n_x\times n_p}$
 \For{$k\in\{0,\dots,K-1\}$}
 	\State $\Delta t_k = t_{k+1}-t_k$
	\State $\hat{\bfS}_{k+1}=e^{\Delta t_k\cdot\nabla_x\hat{f}_k}\cdot\bigl(\hat{\bfS}_k+\bigl(I-e^{-\Delta t_k\cdot\nabla_x\hat{f}_k}\bigr)\cdot\nabla_x\hat{f}_k^{-1}\cdot\nabla_p\hat{f}_k\bigr)$
\EndFor
 \end{algorithmic}
\end{algorithm}

Note that in the case of severe stiffness, the PBSR algorithm can trigger the use of the exponential formula \eqref{solutionAtEquilibrium} in most of the time intervals; consequently, in such cases the PBSR algorithm is equivalent to the Exp algorithm.

\subsection{Computational cost and relative accuracy}

The PBSR and Exp algorithms are, along with the FS method, the methods we implement in our numerical experiments (see Section~\ref{sec:numerical}). To complement this numerical study, it is of interest to analyze the computational cost of the three algorithms, and compare them in terms of accuracy.

The FS method is the most expensive of the three algorithms. In particular, the larger the dimension $n_p$ of the parameter space is, the less efficient the FS method becomes when compared to the other two. The reason is that the FS method is based on solving an ODE system of dimension $n_x\times (n_p+1)$, because it couples the ODE for the state $\bfx$ (of dimension $n_x$) with the ODE for the sensitivity matrix $\bfS$ (of dimension $n_x\times n_p$). In contrast, the PBSR and Exp methods only require solving the (smaller) $n_x$-dimensional ODE system for $\bfx$ (which in general requires a coarser time grid compared to the ODE system in the FS method), as well as performing some additional operations on each of the time intervals $[t_k,t_{k+1}]$ that are returned by the ODE solver. 
The cost of these additional operations along with solving the $n_x$-dimensional ODE system does not outweigh the cost of solving the $n_x\times (n_p+1)$-dimensional ODE system in the FS method; the larger $n_p$ is, the more expensive FS becomes, when compared to the PBSR and Exp methods.

To better understand, and compare, the cost of the PBSR and Exp methods, we now outline the operations used by the two methods over each ODE time interval $[t_k,t_{k+1}]$; we omit cheaper computations, such as scalar multiplication and matrix summation.
\vspace{0.1cm}

\noindent    \textbf{Exp algorithm.} In this method the exponential formula \eqref{solutionAtEquilibrium} is applied on each time interval. The (expensive) operations associated with \eqref{solutionAtEquilibrium} are:
    \begin{itemize}
        \item[-] $1$ matrix-matrix multiplication (dimension $n_x\times n_x$),
        \item[-]  $1$ matrix exponentiation (dimension $n_x\times n_x$),
        \item[-] solving $1$ linear system $\nabla_xf \cdot v= \nabla_pf$, where the dimension of $\nabla_xf$ is $n_x\times n_x$, and the dimension of $\nabla_pf$ is $n_x\times n_p$ (1 QR decomposition + $n_p$ forward substitution sequences; or 1 LU decomposition + $n_p$  substitution sequences).
    \end{itemize}
\vspace{0.1cm}

\noindent   \textbf{PBSR algorithm.} The cost of the PBSR method for each time interval depends on whether or not the current time interval is refined into finer sub-intervals:
\begin{itemize}
    \item[(i)] if the refinement is applied, the (expensive) operations are $3\cdot n_\tint$ matrix-matrix multiplications (dimension $n_x\times n_x$), because the time interval is refined into $n_\tint$ sub-intervals and the PBS formula \eqref{eq:Approx} consists of 3 matrix-matrix multiplications;  
    \item[(ii)] if the exponential formula \eqref{solutionAtEquilibrium} is applied (and thus no refinement), the cost is the same as for the Exp algorithm (described above). 
\end{itemize}
The outline of the (expensive) operations used for each method suggests that, in general, because of the refinement, the PBSR algorithm requires a larger number of operations, and therefore we expect the Exp algorithm to be less computationally expensive. %more efficient.
This will also be confirmed in the numerical experiments in Section~\ref{sec:numerical}. 

The advantage of the PBSR over the Exp algorithm is the higher accuracy it provides. The reasons for this higher accuracy are (i) the refinement performed in the PBSR algorithm, and (ii) the higher accuracy of the PBS formula~\ref{eq:Approx} (used in the PBSR algorithm) over the exponential formula~\ref{solutionAtEquilibrium}: when the system is not at equilibrium and $\Delta t_k$ is sufficiently small, the PBS formula \eqref{eq:Approx} is more accurate than \eqref{solutionAtEquilibrium}, as it takes into account the changes in $\nabla_x f$ and $\nabla_p f$ in the interval $[t_k,t_{k+1}]$, whereas only the values of the Jacobians at time $t_k$ enter into the exponential formula \eqref{solutionAtEquilibrium}. 

Note that when the system reaches a steady state, the Jacobians $\nabla_xf$ and $\nabla_pf$ are (approximately) constant, and the error in the sensitivity matrix approximated with the Exp algorithm is of the same order as the PBSR approximation. This behaviour can also be observed in the numerical example in Section~\ref{sec:numerical} where we approximate the sensitivity matrix in the model for the CaMKII signalling pathway, which reaches convergence in the time span considered (see Figure~\ref{fig:CaMKII}). If one is interested in the approximation of the sensitivities \textit{only} at steady state, then the Exp algorithm is to be preferred, as it is cheaper and provides similar approximation to the PBSR algorithm (at convergence). However, in our applications (see Section~\ref{subsec:motivation}) we need to compute the Fisher information~\eqref{eq:F}, which requires the sensitivities at all time-points $\{t_l\}$ at which the state of the system is measured. These time-points are typically in the transient phase. Thus for our applications, and in general for cases where a good approximation of the sensitivity matrix is required in the transient phase, the PBSR algorithm is more suitable than the Exp algorithm, despite the slight increase in computational time.

We end this section with a brief discussion of another possible modification of the algorithms presented so far: an algorithm that performs the same refinement as the PBSR and applies the exponential formula \eqref{solutionAtEquilibrium} on each sub-interval. Although a natural idea to consider, we argue that it is not worthwhile to implement such an algorithm. In fact, because the PBS formula~\eqref{eq:Approx} is more accurate than the exponential formula~\eqref{solutionAtEquilibrium}, such a method becomes less accurate than the PBSR algorithm. Moreover, the three (expensive) operations in the exponential formula \eqref{solutionAtEquilibrium} (1 matrix-matrix multiplication, 1 matrix exponentiation, solving 1 linear system) are much more expensive than those in the PBS formula \eqref{eq:Approx} (3 matrix-matrix multiplications), as can be seen from Table~\ref{tab:costs}, where we compare the cost for the different operations. We do this in two situations corresponding to two of the models used in Section \ref{sec:numerical}. The conclusion is that  combining refinement of the time intervals with (only) the exponential formula comes with a higher computational cost, but no improvement in accuracy when compared to the PBSR.

\begin{table}\centering
  \caption{Relative numerical costs (calculation time ratio) between different
    operations applied to matrices of sizes $11\times 11$ and $21\times 21$ corresponding to Jacobians $\nabla_xf$ arising from the two models from systems biology described in Section~\ref{sec:numerical} (PKA and CaMKII, respectively) listed in the first
    column. Operations: double precision general matrix-matrix
    multiplication (\texttt{dgemm}) (used as reference for each of the two models), matrix exponentiation via the \emph{scale and square} algorithm (\texttt{expm}), solve
    linear system in place (\texttt{svx}), LU decomposition
    (\texttt{LU}). These numerical results are obtained from a C
    implementation using the GNU Scientific Library. The standard error is indicated in parentheses (concise error notation). %(the best error notation).
    \label{tab:costs}}
  \begin{tabular}{cccccc}
    Model &Matrix dimension & \texttt{dgemm}&\texttt{expm}&\texttt{svx}&\texttt{LU}\\
    \midrule 
    PKA &11 $\times$ 11 & 1 & \num{5.96 +- 0.31}  & \num{4.79 +- 0.27} & \num{0.74 +- 0.03} \\
    CaMKII &21 $\times$ 21 &1& \num{8.46 +- 0.99} & \num{3.22 +- 0.79} & \num{0.48 +- 0.07} \\
  \end{tabular}
 % \vspace*{2ex}
\end{table}

\section{Error Analysis of the Approximations of $\bfS$}
\label{sec:errEst}

In this section we present the result of an error analysis for the approximation of the sensitivity matrix $\bfS$ obtained by applying the PBS formula~\eqref{eq:Approx} to the iteratively defined ODE systems~\eqref{sensitivityODE}, as $\max_k\Delta t_k\to 0$ (or equivalently, as $K\to\infty$). 

The errors involved are represented as vectors or matrices that correspond to the difference between an exact term and its approximation. In order to provide error estimates, and conclude on the order of convergence, we will consider the norm $\|\cdot\|$ of the errors. Since we are working in finite-dimensional vector spaces, all norms are equivalent and the exact choice is irrelevant to the analysis. Recall also from Section~\ref{sec:problem} that we have assumed enough regularity of the problem so that  existence and uniqueness of a solution of \eqref{eq:ODE} is guaranteed.

To simplify the notation, we define $\pbsJ_k$ as
\begin{equation*}
    \pbsJ_k:=\int_{t_k}^{t_{k+1}}\Phi(t_k{;}\tau)\cdot \nabla_pf(\tau,\bfx,\bfp)d\tau,
\end{equation*}
and the corresponding approximation
\begin{equation*}
\hat{\pbsJ}_k:=\frac{\Delta t_k}{2}\cdot(\nabla_p\hat{f}_k+\hat{\Phi}(t_{k}{;}t_{k+1})\cdot\nabla_p\hat{f}_{k+1}).
\end{equation*}

Having established the notation, we can use Theorem~\ref{thm:generalODESol} to formulate the exact solution of the ODE system \eqref{sensitivityODE} for the sensitivity matrix at time step $t_{k+1}$ as
\begin{equation}
    \label{exactSensitivity}
    \bfS(t_{k+1}) = \Phi(t_{k+1}{;}t_k)\cdot\left(\bfS(t_k)+\pbsJ_k\right),
\end{equation}
while the approximate solution is defined as
\begin{equation}
\label{approximateSensitivity}
\hat{\bfS}_{k+1} = \hat{\Phi}(t_{k+1}{;}t_k)\cdot\left(\hat{\bfS}_k+\hat{\pbsJ}_k\right).
\end{equation}
Figure~\ref{fig:errorPropagation} illustrates the terms and the corresponding errors in the approximation \eqref{approximateSensitivity} in a graph. Each node in the graph represents a new source of error and the directed edges show the propagation of error from one node (an approximation) to another, as a result of the approximation being used instead of the exact quantity.

\begin{figure}
\centering
 \includegraphics[width=\textwidth]{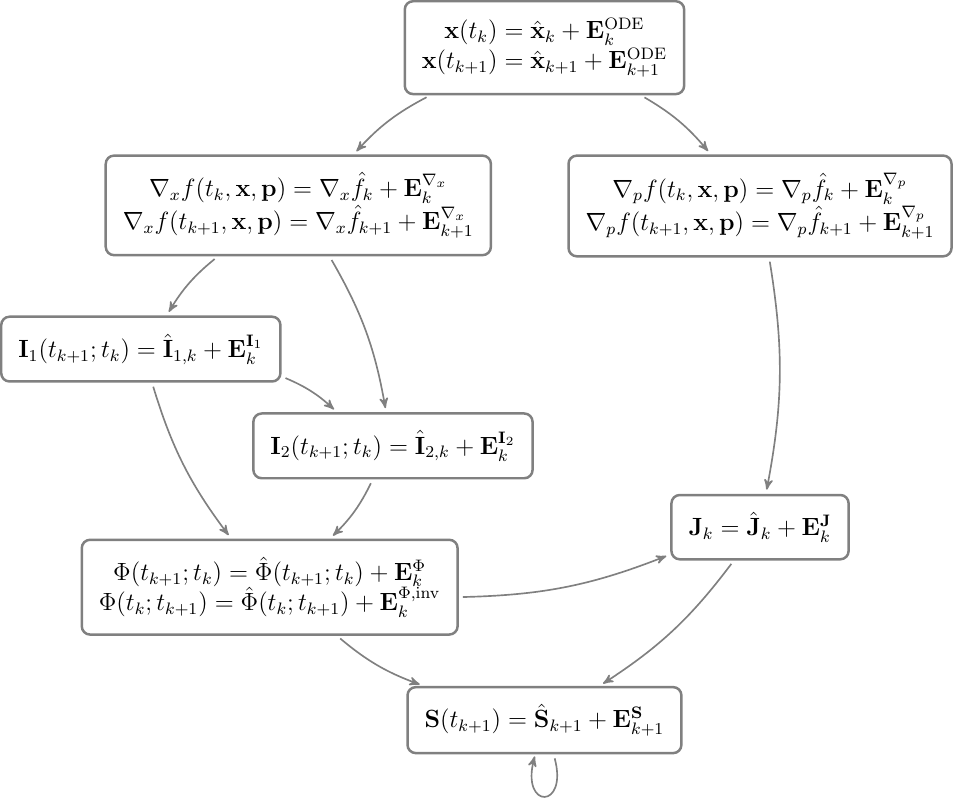}
\caption{Error propagation graph.}
\label{fig:errorPropagation}
\end{figure}

The main result of this section, Theorem~\ref{thm:error}, is a characterisation of the error in the approximation of $\bfS$, at time step $k+1$, in terms of the largest time step $\Delta t _{\tmax}$ used by the underlying ODE solver. Before we can state this result, we list the assumptions we make on the functions involved. The first assumption concerns the numerical solution $\hat{\bfx}_k$, at each time step $t_k$, of the ODE \eqref{eq:ODE},
\begin{equation*}
    \err^{\tODE}_{k} = \bfx(t_k)-\hat{\bfx}_k,
\end{equation*}
 which is defined as an $n_x-$dimensional error vector. We make the following assumption on the global error of the underlying numerical method.

\begin{assumption}
\label{as:errorODE}
The ODE solution of \eqref{eq:ODE} has a global error of at least second order, i.e. for all $k=0,\dots,K-1$, \ $\left\|\err^{\tODE}_{k}\right\|=\mathcal{O}(\Delta t_{\tmax}^2)$.
\end{assumption}
Next, for partial derivatives of $f$ with respect to $\bfx$ and $\bfp$ we use the notation,
\begin{align*}
    \left(H_{xx}\right)^{jh}_l &:=\frac{\partial^2 f_l}{\partial x_h\partial x_j}, \\
    \left(H_{xp}\right)^{ih}_l &:=\frac{\partial^2 f_l}{\partial x_h\partial p_i}.
\end{align*}
\begin{assumption}
\label{as:boundedH}
    The third order tensors $H_{xx}(t,\bfx,\bfp)$ and $H_{xp} (t,\bfx,\bfp)$ of second partial derivatives with respect to $\bfx$ and $\bfp$ are both bounded with respect to the supremum norm $\| \cdot \|_{[0,T]}$.
\end{assumption}
The fourth assumption concerns boundedness of certain maps with respect to $\| \cdot \| _{[0,T]}$.
\begin{assumption}
\label{as:time}
The following maps are bounded with respect to $\| \cdot \| _{[0,T]}$:
\begin{enumerate}[label=(\ref{as:time}.\arabic*)]
    \item\label{as:time1} $t \mapsto \nabla _x f(t,\bfx,\bfp)$,
    \item\label{as:time1p} $t \mapsto \nabla _p f(t,\bfx,\bfp)$,
    \item\label{as:time2} $t\mapsto\frac{d^2}{dt^2}\nabla_xf(t,\bfx,\bfp)$,
    \item\label{as:time3} $t\mapsto\frac{d^2}{dt^2}\left(\nabla_xf(t,\bfx,\bfp)\pbsI_1(t{;}t_k)\right)$,
    \item\label{as:time4} $t\mapsto\frac{d^2}{dt^2}\left(\Phi(t_k{;}t)\cdot\nabla_pf(t,\bfx,\bfp)\right)$.
\end{enumerate}
\end{assumption}
Assumption~\ref{as:time1} is used in the study of the remainder term of the Peano-Baker series, after truncation at $n=2$, and Assumption~\ref{as:time1p} in the proof of the error in the terms $\mathbf{J}_k$. The remaining assumptions~\ref{as:time2}–\ref{as:time4} are used for arguments involving the trapezoidal rule, which explains why the second derivative with respect to time appears.

The final assumption concerns the sensitivity matrix $\bfS$ and the integrands in the $\pbsJ _k$s.
\begin{assumption}
\label{as:matrices}
There exist finite, positive constants $C_\bfS$ and $C_{\pbsJ}$, such that $C_\bfS \geq \| \bfS (t) \|$ for all $t \in [0,T]$, and $C_{\pbsJ}:=\max_{k=0,\dots,K-1}\|\Phi(t_k;t)\cdot\nabla_pf(t,\bfx,\bfp)\|_{[0,T]}$.
\end{assumption}

The assumptions are stated not with the aim of full generality, but rather to provide enough regularity for the overall results to not get obscured by technical details. It is clear that the assumptions can be made both more explicit and less restrictive if aimed at specific examples. As an example, the assumptions of Lemma~\ref{lemma:errI1} are satisfied if, for example, $f\in\calC^3(\bR,\bR^{n_x},\bR^{n_p})$. Similarly, if the vector field $f$ in the ODE \eqref{eq:ODE} is not time-dependent, more explicit error estimates can be obtained. We leave such refinements of the results for future work considering more specific formulations of the underlying IVP \eqref{eq:ODE}.

We are now ready to state the main theorem on the error of the proposed approximation; the proof is carried out in Section \ref{sec:proofErrEst}.
\begin{theorem}
\label{thm:error}
Suppose Assumptions~\ref{as:errorODE}–\ref{as:matrices} hold. Given a discretization $0=t_0<t_1<\dots<t_{K-1}<t_K=T$ of the time interval $[0,T]$, let $\Delta t_{\tmax}:=\max_{h=0,\dots,K-1}\Delta t_h$ and $\Delta t_{\tmin}:=\min_{h=0,\dots,K-1}\Delta t_h$, and assume that there exists a finite constant $\Xi$ such that $\frac{\Delta t_{\tmax}}{\Delta t_{\tmin}}\le \Xi$, as $\Delta t_{\tmax}\to 0$. Then, the error $\err_{k+1}^{\bfS}$ in the approximation of the sensitivity matrix $\bfS$ at time step $t_{k+1}$ obtained by applying the PBS formula \eqref{eq:Approx} to the ODE systems~\eqref{sensitivityODE} satisfies
\begin{equation}
    \label{errSOrder}
    \left\|\err_{k+1}^{\bfS}\right\|=\mathcal{O}(\Delta t_{\tmax}^2),
\end{equation}
as $\Delta t_{\tmax}\to 0$.
\end{theorem}

Theorem \ref{thm:error} covers the error associated with the PBS algorithm. It turns out that, if the map $t\mapsto \frac{d}{d t} f(t,\bfx,\bfp)$ is bounded with respect to $\| \cdot \| _{[0,T]}$, the result holds also when a refinement of the time intervals $\Delta t_{k}$ is used, and the PBS formula \eqref{eq:Approx} is applied on the sub-intervals, as in the PBSR algorithm (Algorithm~\ref{alg:PBSR}). This is formalised in the following corollary, the proof of which is also provided in Section \ref{sec:proofErrEst}.

\begin{corollary}
    \label{cor:errorPBSR}
Suppose Assumptions~\ref{as:errorODE}–\ref{as:matrices} hold and the map $t\mapsto \frac{d}{d t} f(t,\bfx,\bfp)$ is bounded with respect to $\| \cdot \| _{[0,T]}$. Given a discretization $0=t_0<t_1<\dots<t_{K-1}<t_K=T$ of the time interval $[0,T]$, let $\Delta t_{\tmax}:=\max_{h=0,\dots,K-1}\Delta t_h$ and $\Delta t_{\tmin}:=\min_{h=0,\dots,K-1}\Delta t_h$, and assume that there exists a finite constant $\Xi$ such that $\frac{\Delta t_{\tmax}}{\Delta t_{\tmin}}\le \Xi$, as $\Delta t_{\tmax}\to 0$. Consider any refinement of the time intervals $[t_h,t_{h+1}]$ and linearly interpolate the numerical solution $\hat\bfx$ of the ODE system~\eqref{eq:ODE} on each sub-interval. Then, the error $\err_{k+1}^{\bfS}$ in the approximation of the sensitivity matrix $\bfS$ at time step $t_{k+1}$ obtained by applying the PBS formula \eqref{eq:Approx} to the ODE systems~\eqref{sensitivityODE} defined on the refined time grid satisfies
\begin{equation}
    \label{errSOrder}
    \left\|\err_{k+1}^{\bfS}\right\|=\mathcal{O}(\Delta t_{\tmax}^2),
\end{equation}
as $\Delta t_{\tmax}\to 0$.
\end{corollary}

\section{Proof of the Error Estimates for the Approximations of $\bfS$}
\label{sec:proofErrEst}
In the previous section we stated Theorem~\ref{thm:error} and Corollary~\ref{cor:errorPBSR} about the error estimates for the approximations of $\bfS$ using the PBS formula \eqref{eq:Approx} on the ODE systems \eqref{sensitivityODE} (without and with refinement of the time intervals, respectively). Here we provide the associated proofs, starting from Theorem~\ref{thm:error}, for which we will use a series of lemmas, corresponding roughly to the directed edges in Figure~\ref{fig:errorPropagation}. The first source of error comes from the numerical solution $\hat{\bfx}_k$, at each time step $t_k$, of the ODE \eqref{eq:ODE}, represented by the uppermost node in Figure~\ref{fig:errorPropagation}. Under Assumption~\ref{as:errorODE}, the error $\err^{\tODE}_{k}$ associated with $\hat \bfx _k$ is $\mathcal{O}(\Delta t^2_\tmax)$.

Next, we consider the approximations of the Jacobians, defined in in \eqref{Jac_x} and \eqref{Jac_p}. As shown by the two directed edges going from the uppermost node in Figure~\ref{fig:errorPropagation}, the error $\err_k^{\tODE}$ propagates to both these terms: the errors that arise in the Jacobians are due to the evaluation of the exact Jacobians $\nabla_x f$ and $\nabla_p f$ in the approximated solution $\hat{\bfx}_k$ instead of the exact counterpart $\bfx(t_k)$:
\begin{align*}
    \err_k^{\nabla_x} = \nabla_{x}f(t_k,\bfx,\bfp)-\nabla_x\hat{f}_k,\\
    \err_k^{\nabla_p} = \nabla_{p}f(t_k,\bfx,\bfp)-\nabla_p\hat{f}_k.
\end{align*}
The following result is standard and included for completeness.

\begin{lemma}
\label{lemma:Jacobians}
    Under Assumptions~\ref{as:errorODE} and \ref{as:boundedH}, the errors $\err_k ^{\nabla _x}$ and $\err _k ^{\nabla _p}$ are both of second order in $\Delta t_\tmax$, i.e. 
    \begin{align*}
        \left\|\err_k ^{\nabla _x}\right\|=\mathcal{O}(\Delta t_{\tmax}^2) \qquad \text{and}\qquad \left\|\err_k ^{\nabla _p}\right\|=\mathcal{O}(\Delta t_{\tmax}^2),
    \end{align*}
    as $\Delta t_{\tmax}\to0$.
\end{lemma}
\begin{proof}
The errors $\err_k ^{\nabla _x}$ and $\err _k ^{\nabla _p}$ are matrices of dimension $n_x \times n_x$ and $n_x \times n_p$, respectively. For an arbitrary $h\in\{1,\dots,n_x\}$, let $\left(\err_{k}^{\tODE}\right)_h$ denote the $h$th component of the $n_x-$dimensional error vector $\err_k^{\tODE}$. The entries of $\err_k ^{\nabla _x}$ and $\err _k ^{\nabla _p}$ are then given by
\begin{align*}
    \left(\err_k^{\nabla_x}\right)_{l}^{~j} &=\frac{\partial f_{l}(t_k,\bfx,\bfp)}{\partial x_j} - \frac{\partial f_{l}(\hat{\bfx}_k)}{\partial x_j}\\
    &=\frac{\partial^2f_{l}(t_k,\bfx,\bfp)}{\partial x_h\partial x_j}\cdot \left(\bfx(t_k)-\hat{\bfx}_{k}\right)_h+\mathcal{O}\left(\left\|\err_{k}^{\tODE}\right\|^2\right)\\
    &=\left(H_{xx}(t,\bfx,\bfp)\right)^{~jh}_l\left(\err_{k}^{\tODE}\right)_h+\mathcal{O}\left(\left\|\err_{k}^{\tODE}\right\|^2\right),
\end{align*}
and
\begin{align*}
    \left(\err_k^{\nabla_p}\right)_{l}^{~i} &= \frac{\partial f_{l}(t_k,\bfx,\bfp)}{\partial p_i} - \frac{\partial f_{l}(\hat{\bfx}_k)}{\partial p_i}\\
    &=\frac{\partial^2f_{l}(t_k,\bfx,\bfp)}{\partial x_h\partial p_i}\cdot \left(\bfx(t_k)-\hat{\bfx}_{k}\right)_h+\mathcal{O}\left(\left\|\err_{k}^{\tODE}\right\|^2\right)\\
    &=\left(H_{xp}(t,\bfx,\bfp)\right)^{~ih}_l\left(\err_{k}^{\tODE}\right)_h+\mathcal{O}\left(\left\|\err_{k}^{\tODE}\right\|^2\right).
\end{align*}
Given these forms for the entries, we obtain the following bounds,
\begin{align*}
    \left\|\err_k^{\nabla_x}\right\| &\le C_1 \left\|H_{xx}(t_k,\bfx,\bfp)\right\|\left\|\err_k^{\tODE}\right\|+\mathcal{O}\left(\left\|\err_k^{\tODE}\right\|^2\right),\\
    \left\|\err_k^{\nabla_p}\right\| &\le C_2 \left\|H_{xp}(t_k,\bfx,\bfp)\right\|\left\|\err_k^{\tODE}\right\|+\mathcal{O}\left(\left\|\err_k^{\tODE}\right\|^2\right),
\end{align*}
where the constants $C_1>0$ and $C_2>0$ depend on the specific dimensions $n_x$ and $n_p$ considered and on the choice of norms.

Under Assumption~\ref{as:boundedH} on $H_{xx}(t,\bfx,\bfp)$ and $H_{xp} (t,\bfx,\bfp)$, there exist finite constants $C_{H_{xx}}:=\left\|H_{xx}\right\|_{[0,T]}$ and $C_{H_{xp}}:=\left\|H_{xp}\right\|_{[0,T]}$. By the definition of the supremum norm, we obtain uniform (with respect to time) bounds for the errors in the Jacobians:
\begin{align*}
    \left\|\err_k^{\nabla_x}\right\| &\le C_1C_{H_{xx}}\left\|\err_k^{\tODE}\right\|+\mathcal{O}\left(\left\|\err_k^{\tODE}\right\|^2\right),\\
    \left\|\err_k^{\nabla_p}\right\| &\le  C_2C_{H_{xp}}\left\|\err_k^{\tODE}\right\|+\mathcal{O}\left(\left\|\err_k^{\tODE}\right\|^2\right).
\end{align*}
This shows that the norm of either error is of the same order as the error in the ODE solution, which by Assumption~\ref{as:errorODE} is $\mathcal{O}(\Delta t_\tmax^2)$ as $\Delta t_\tmax\to0$.
\end{proof}

Following the error propagation in Figure~\ref{fig:errorPropagation}, the approximations of $\nabla_xf(t_k,\bfx,\bfp)$ as $\nabla_x\hat{f}_k$, and its counterpart at time step $k+1$, are used to approximate the integrals $\pbsI_1(t_{k+1}{;}t_k)$ and $\pbsI_2(t_{k+1}{;}t_k)$, respectively, with $\hat{\pbsI}_{1,k}$ and $\hat{\pbsI}_{2,k}$, as defined in \eqref{approxIntegrals}; note that the approximation $\hat{\pbsI}_{1,k}$ is used in $\hat{\pbsI}_{2,k}$ (see Figure~\ref{fig:errorPropagation}). The errors that arise are a result of both the application of the trapezoidal rule for numerical integration and the errors in the Jacobians, as described by Lemma~\ref{lemma:Jacobians}.

\begin{lemma}
\label{lemma:errI1}
Suppose Assumptions~\ref{as:errorODE},~\ref{as:boundedH} and~\ref{as:time2} hold. Then the error 
\begin{align*}
    \err_k^{\trap,\pbsI_1} = \pbsI_1(t_{k+1};t_k) -\frac{\Delta t_k}{2}\cdot\left(\nabla_xf(t_k,\bfx,\bfp) +\nabla_xf(t_{k+1},\bfx,\bfp)\right)
\end{align*}
due to the approximation of $\pbsI_1(t_{k+1}{;}t_k)$ by the trapezoidal rule\footnote{indicated by the wedge symbol $\wedge$} is
\begin{equation}
    \label{errTrapOrder}
    \left\|\err_k^{\trap,\pbsI_1}\right\|=\mathcal{O}(\Delta t_k^3),
\end{equation}
and the error $\err_k^{\pbsI_{1}}$, from the approximation of $\pbsI_1(t_{k+1};t_k)$ by $\hat{\pbsI}_{1,k}$, satisfies
\begin{equation}
    \label{errI1Order}
    \left\|\err_k^{\pbsI_{1}}\right\|=\mathcal{O}(\Delta t_\tmax^3),
\end{equation}
as $\Delta t_\tmax\to0$.
\end{lemma}
\begin{proof}
It can be proved \cite[Sec. 5.2.1]{SauerTim2006Na} that the error of the trapezoidal rule applied to the integral of a function $g\in C^2(A,\bR)$, with $A\subset\bR$, on the interval $[a,b]\subset A$ is
\begin{equation}
    \label{trapezoidalRuleError}
    \int_a^bg(x)dx-\frac{b-a}{2}\cdot(g(a)+g(b))=-\frac{(b-a)^3}{12}\cdot g''(\xi),
\end{equation}
for some $\xi\in(a,b)$. Using the latter, the components of $\err_k^{\trap,\pbsI_1}$ are
\begin{equation}
    \label{errTrapIntermediate}
    \left(\err_k^{\trap,\pbsI_1}\right)_{l}^{~j}=-\frac{\Delta t_k^3}{12}\cdot\left( \frac{d^2}{dt^2}\nabla_xf(t,\bfx,\bfp)\right)_l^{~j}\Bigg\vert_{t=\xi},
\end{equation}
where in general $\xi\in(t_k,t_{k+1})$ is different for each pair of indices $l,j$. From the properties of matrix norms, we have that for every $t\in[t_k,t_{k+1}]$ and $l,j=1,\dots,n_x$,
\begin{equation*}
\left\| \left( \frac{d^2}{dt^2}\nabla_xf(t,\bfx,\bfp) \right)_l^{~j} \right\| \le C \left\| \frac{d^2}{dt^2}\nabla_xf(t,\bfx,\bfp)\right\|.
\end{equation*}
for a constant $C>0$ that only depends on the choice of norm $\|\cdot\|$.

Using Assumption~\ref{as:time2} for $\frac{d^2}{dt^2}\nabla_xf(t,\bfx,\bfp)$, there exist a constant $C' <\infty$ such that  $C'=\left\|\frac{d^2}{dt^2}\nabla_xf(t,\bfx,\bfp)\right\|_{[0,T]}$. It follows that, for every $l,j=1,\dots,n_x$,
\begin{equation*}
\left\|\left(\err_k^{\trap,\pbsI_1}\right)_{l}^{~j}\right\|\le\frac{CC'}{12}\Delta t_k^3.
\end{equation*}

Using the fact that all matrix norms are equivalent, there exists a constant $C''>0$, which depends on the choice of norm, such that
\begin{equation*}
\left\|\err_k^{\trap,\pbsI_1}\right\|\le C'' \max_{l,j=1,\dots,n_x}\left\|\left(\err_k^{\trap,\pbsI_1}\right)_{l}^{~j}\right\|.
\end{equation*}
Combined with the previous inequality this yields the following upper bound on the error from the trapezoidal rule,
\begin{equation*}
\left\|\err_k^{\trap,\pbsI_1}\right\|\le \frac{CC'C''}{12}\Delta t_k^3.
\end{equation*}
This proves \eqref{errTrapOrder}.

Having established an order of convergence for the error due to the trapezoidal rule, we now expand the error arising from the approximation of $\pbsI _{1} (t_{k+1}{;}t_k)$:
\begin{align*}
    \err_k^{\pbsI_{1}} &= \pbsI_1(t_{k+1}{;}t_k) - \hat{\pbsI}_{1,k} \\
    &=  \pbsI_1(t_{k+1}{;}t_k) - \frac{\Delta t_k}{2}\cdot (\nabla_x\hat{f}_k + \nabla_x\hat{f}_{k+1})\nonumber\\
    &=\pbsI_1(t_{k+1}{;}t_k) -\frac{\Delta t_k}{2}\cdot\left(\nabla_xf(t_k,\bfx,\bfp) +\nabla_xf(t_{k+1},\bfx,\bfp)\right)\\
    & \quad +\frac{\Delta t_k}{2}\cdot\left(\nabla_xf(t_k,\bfx,\bfp) -\nabla_x\hat{f}_k+\nabla_xf(t_{k+1},\bfx,\bfp) - \nabla_x\hat{f}_{k+1}\right)\nonumber\\
    &=\err_k^{\trap,\pbsI_1} + \frac{\Delta t_k}{2}\cdot(\err_k^{\nabla_x}+\err_{k+1}^{\nabla_x})\label{errI1}.
\end{align*}
Together with Lemma~\ref{lemma:Jacobians}, this yields the upper bound

\begin{equation*}
    \left\|\err_k^{\pbsI_1}\right\|\le \left\|\err_k^{\trap,\pbsI_1}\right\|+\mathcal{O}(\Delta t_\tmax^3),
\end{equation*}
which proves the claimed order of convergence \eqref{errI1Order}.
\end{proof}

With these estimates for the errors in $\nabla_x\hat{f}_k$ and $\hat{\pbsI}_{1,k}$, we now turn to the approximation of $\pbsI_{2,k}$. The first part of this lemma, concerning the error due to an application of the trapezoidal rule, is analogous to the first part of Lemma~\ref{lemma:errI1}.

\begin{lemma}
\label{lemma:errI2}
Suppose Assumptions~\ref{as:errorODE}, \ref{as:boundedH}, \ref{as:time1}-\ref{as:time3} hold. Then, the error 
\begin{align*}
    \err_k^{\trap,\pbsI_2}=\pbsI_2(t_{k+1};t_k) - \frac{\Delta t_k}{2}\cdot\nabla_xf(t_{k+1},\bfx,\bfp)\cdot\pbsI_1(t_{k+1};t_k)
\end{align*}
due to the approximation of $\pbsI_2(t_{k+1};t_k)$ by the trapezoidal rule is
\begin{equation}
    \label{errTrapOrder2}
    \left\|\err_k^{\trap,\pbsI_2}\right\|=\mathcal{O}(\Delta t_k^3),
\end{equation}
and the error $\err_k^{\pbsI_{2}}$, from the approximation of $\pbsI_2(t_{k+1}{;}t_k)$ by $\hat{\pbsI}_{2,k}$, is
\begin{equation}
    \label{errI2Order}
    \left\|\err_k^{\pbsI_{2}}\right\|=\mathcal{O}(\Delta t_\tmax^3),
\end{equation}
as $\Delta t_\tmax\to0$.
\end{lemma}
\begin{proof}
The proof of \eqref{errTrapOrder2} is analogous to the proof of Lemma~\ref{lemma:errI1}: by replacing $\err_k^{\trap,\pbsI_1}$ with $\err_k^{\trap,\pbsI_2}$, and $\frac{d^2}{dt^2}\nabla_xf(t,\bfx,\bfp)$ with $\frac{d^2}{dt^2}\left(\nabla_xf(t,\bfx,\bfp)\pbsI_1(t,t_k)\right)$, and defining
\begin{equation*}
\tilde{C}':=\left\|\frac{d^2}{dt^2}\left(\nabla_xf(t,\bfx,\bfp)\pbsI_1(t,t_k)\right)\right\|_{[0,T]},
\end{equation*} we obtain
\begin{equation*}
\left\|\err_k^{\trap,\pbsI_2}\right\|\le \frac{C\tilde{C}'C''}{12}\Delta t_k^3.
\end{equation*}
This proves the order of convergence \eqref{errTrapOrder2}.

To show \eqref{errI2Order}, we note that this error can be expressed as
\begin{align}
    \err_k^{\pbsI_2} &= \pbsI_2(t_{k+1}{;}t_k)-\hat{\pbsI}_{2,k} \nonumber\\
    &=\pbsI_2(t_{k+1}{;}t_k) - \frac{\Delta t_k}{2}\cdot\nabla_xf(t_{k+1},\bfx,\bfp)\cdot\pbsI_1(t_{k+1}{;}t_k) \nonumber \\
    &\quad + \frac{\Delta t_k}{2}\cdot\left(\nabla_xf(t_{k+1},\bfx,\bfp)\cdot\pbsI_1(t_{k+1}{;}t_k)
    -\nabla_x\hat{f}_{k+1}\cdot\hat{\pbsI}_{1,k}\right) \nonumber\\
    &=\err_k^{\trap,\pbsI_2}+ \frac{\Delta t_k}{2}\cdot\left(\nabla_xf(t_{k+1},\bfx,\bfp)\cdot\pbsI_1(t_{k+1}{;}t_k) - \nabla_xf(t_{k+1},\bfx,\bfp) \cdot\hat{\pbsI}_{1,k} \right. \nonumber \\
    & \qquad \qquad \qquad \qquad \qquad  \left. +\nabla_xf(t_{k+1},\bfx,\bfp) \cdot\hat{\pbsI}_{1,k} -\nabla_x\hat{f}_{k+1}\cdot\hat{\pbsI}_{1,k}\right)\nonumber\\
    &=\err_k^{\trap,\pbsI_2}+\frac{\Delta t_k}{2}\cdot\left(\nabla_xf(t_{k+1},\bfx,\bfp) \cdot \err_k^{\pbsI_1}+\err_{k+1}^{\nabla_x}\cdot\hat{\pbsI}_{1,k}\right)\label{errI2}.
\end{align}
Applying the norm operator to \eqref{errI2} and using the triangle inequality gives
\begin{equation*}
  \left\|\err_k^{\pbsI_2}\right\|\le \left\|\err_k^{\trap,\pbsI_2}\right\|+\frac{\Delta t_k}{2}\cdot\left(\left\|\nabla_xf(t_{k+1},\bfx,\bfp)\right\|\left\|\err_k^{\pbsI_1}\right\|+\left\|\err_{k+1}^{\nabla_x}\right\|\left\|\hat{\pbsI}_{1,k}\right\|\right).
\end{equation*}
 From Assumption~\ref{as:time1}, $\left\|\nabla_xf(t_{k+1},\bfx,\bfp)\right\|=\mathcal{O}(1)$ and, from the definition of $\hat{\pbsI}_{1,k}$ (see \eqref{approxIntegrals}), we have $\left\|\hat{\pbsI}_{1,k}\right\|=\mathcal{O}(\Delta t_k)$. From the upper bound in the last displayed formula, combined with %\eqref{errTrapOrder2} and 
 Lemmas~\ref{lemma:Jacobians}-\ref{lemma:errI1}, we obtain
\begin{equation*}
\left\|\err_k^{\pbsI_2}\right\|\le \left\|\err_k^{\trap,\pbsI_2}\right\|+\mathcal{O}(\Delta t_\tmax^4),
\end{equation*}
which, together with \eqref{errTrapOrder2}, proves the order of convergence \eqref{errI2Order}.
\end{proof}

Next, we move to the approximation of $\Phi(t_{k+1}{;}t_k)$. As illustrated in Figure~\ref{fig:errorPropagation}, this approximation depends directly on the approximations of $\pbsI_{1}(t_{k+1}{;}t_k)$ and $\pbsI_{2}(t_{k+1}{;}t_k)$, and thus Lemmas~\ref{lemma:errI1}-\ref{lemma:errI2} will be used to obtain the order of convergence of the associated error.

\begin{lemma}
\label{lemma:errPhiOrder}
Suppose that Assumptions~\ref{as:errorODE}, \ref{as:boundedH} and \ref{as:time1}-\ref{as:time3} hold. Then, the error $\err_k^{\Phi}$ in the approximation of $\Phi(t_{k+1}{;}t_k)$ by $\hat{\Phi}(t_{k+1}{;}t_k)$ is
\begin{equation*}
\left\|\err_k^{\Phi}\right\|=\mathcal{O}(\Delta t_\tmax^3),
\end{equation*}
as $\Delta t_\tmax\to0$.
\end{lemma}
\begin{proof}
Recalling the definition \eqref{PhiApproximation} of the approximation $\hat{\Phi}(t_{k+1}{;}t_k)$, we can express the associated $\err_k^{\Phi}$ as
\begin{align}
\err^{\Phi}_k &=\Phi(t_{k+1}{;}t_k)-\idI_{n_x} - \hat{\pbsI}_{1,k} - \hat{\pbsI}_{2,k} \nonumber\\
&=\Phi(t_{k+1}{;}t_k) - \sum_{n=0}^{2} \pbsI_n(t_{k+1}{;}t_k)  \nonumber \\
&\quad + \idI_{n_x} + \pbsI_1(t_{k+1}{;}t_k) + \pbsI_2(t_{k+1}{;}t_k) - \idI_{n_x} - \hat{\pbsI}_{1,k} - \hat{\pbsI}_{2,k} \nonumber\\
&=\sum_{n=3}^\infty\pbsI_n(t_{k+1}{;}t_k) + \err^{\pbsI_1}_k + \err^{\pbsI_2}_k\label{errPhi},
\end{align}

Lemmas~\ref{lemma:errI1} and \ref{lemma:errI2} describe the behaviour of $\err_k^{\pbsI_1}$ and $\err_k^{\pbsI_2}$ in terms of $\Delta t_k$. To study the term $\sum_{n=3}^\infty\pbsI_n(t_{k+1}{;}t_k)$, we recall from the definition \eqref{recursiveTerm} of the summands $\pbsI_n(t_{k+1}{;}t_k)$,
\begin{equation*}
\pbsI_{n}(t{;}s) = \int_{s}^t \nabla_xf(\tau,\bfx,\bfp)\pbsI_{n-1}(\tau{;}s)d\tau, \ \ n \geq 1.
\end{equation*}
Since $\pbsI_0(t,s)=\idI_{n_x}$, each such term can be expressed as
\begin{equation*}
\pbsI_n(t_{k+1}{;}t_k) = \int_{t_k}^{t_{k+1}}\int_{t_k}^{\tau_1}\cdots\int_{t_k}^{\tau_{n-1}}\nabla_xf(\tau_{n},\bfx,\bfp)\cdots \nabla_xf(\tau_1,\bfx,\bfp) d\tau_{n}\cdots d\tau_1.
\end{equation*}
Using this identity we can obtain an upper bound on the norm of $\sum_{n=3}^\infty\pbsI_n(t_{k+1}{;}t_k)$, the part of the sum that is removed in the truncation term:
\begin{equation}
    \label{majorationNormPBS}
    \begin{split}
    \left\|  \sum_{n=3}^\infty\pbsI_n(t_{k+1}{;}t_k)\right\| &\le \sum_{n=3}^\infty\left\|\pbsI_n(t_{k+1}{;}t_k) \right\| \\
    &=\sum_{n=3}^\infty\left\|\int_{t_k}^{t_{k+1}}\int_{t_k}^{\tau_1}\cdots\int_{t_k}^{\tau_{n-1}}\nabla_xf(\tau_{n},\bfx,\bfp)\cdots \nabla_xf(\tau_1,\bfx,\bfp) d\tau_{n}\cdots d\tau_1 \right\| \\
    &\le \sum_{n=3}^\infty\int_{t_k}^{t_{k+1}}\int_{t_k}^{\tau_1}\cdots\int_{t_k}^{\tau_{n-1}}\left\|\nabla_xf(\tau_{n},\bfx,\bfp)\cdots \nabla_xf(\tau_1,\bfx,\bfp)\right\| d\tau_{n}\cdots d\tau_1 \\
    &\le \sum_{n=3}^\infty\int_{t_k}^{t_{k+1}}\int_{t_k}^{\tau_1}\cdots\int_{t_k}^{\tau_{n-1}}\left\|\nabla_xf(\tau_{n},\bfx,\bfp)\right\|\cdots \left\|\nabla_xf(\tau_1,\bfx,\bfp)\right\| d\tau_{n}\cdots d\tau_1 \\
    &=\sum_{n=3}^\infty\frac{1}{n!} \left(\int_{t_k}^{t_{k+1}}\left\|\nabla_xf(\tau,\bfx,\bfp)\right\|d\tau\right)^n\le\sum_{n=3}^\infty\frac{1}{n!} (C_{\nabla_x})^n(\Delta t_k)^n,
\end{split}
\end{equation}
with $C_{\nabla_x}:=\left\|\nabla_xf(t,\bfx,\bfp)\right\|_{[0,T]}$, which is finite by Assumption~\ref{as:time1}. The last equality is due to the fact that the multiple integrals
\begin{equation*}
\int_{t_k}^{t_{k+1}}\int_{t_k}^{\tau_1}\cdots\int_{t_k}^{\tau_{n-1}}\left\|\nabla_xf(\tau_{n},\bfx,\bfp)\right\|\cdots \left\|\nabla_xf(\tau_1,\bfx,\bfp)\right\| d\tau_{n}\cdots d\tau_1,
\end{equation*}
can be seen as the summands of the Peano-Baker series for the one-dimensional ODE $\dot{z}(t)=\left\|\nabla_xf(t,\bfx,\bfp)\right\|\cdot z(t)$; since the terms $\left\|\nabla_xf(\tau_i,\bfx,\bfp)\right\|$ commute (being scalar functions), such multiple integrals are shown in \cite{PeanoBaker} to be equal to the simpler terms $\frac{1}{n!} \left(\int_{t_k}^{t_{k+1}}\left\|\nabla_xf(\tau,\bfx,\bfp)\right\|d\tau\right)^n$.

If we assume $\Delta t_k<1$ (in fact, we consider $\Delta t_{\tmax}\to 0$), then $\Delta t_k^3\ge \Delta t_k^n$ for every $n\ge 3$, and we obtain the upper bound
 \begin{equation}
     \label{seriesMajoration}
     \left\|  \sum_{n=3}^\infty\pbsI_n(t_{k+1}{;}t_k)\right\|\le (\Delta t_k)^3\sum_{n=3}^\infty\frac{1}{n!} (C_{\nabla_x})^n\le \Delta t_k^3\sum_{n\ge 0}\frac{1}{n!} (C_{\nabla_x})^n=\Delta t_k^3 e^{C_{\nabla_x}}.
 \end{equation}

Taking the norm of \eqref{errPhi}, and using the latter upper bound for $\left\| \sum_{n=3}^\infty\pbsI_n(t_{k+1}{;}t_k) \right\|$,  we obtain
\begin{equation*}
\left\|\err^{\Phi}_k\right\|\le \Delta t_k^3 e^{C_{\nabla_x}}+\left\|\err_k^{\pbsI_1}\right\|+\left\|\err_k^{\pbsI_2}\right\|,
\end{equation*}
where all three terms at the right-hand side are $\mathcal{O}(\Delta t_\tmax^3)$. This concludes the proof.
\end{proof}

Before we proceed with analysing the approximation of $\pbsJ _k$, which is the last term to consider before we move on to the approximation of $\bfS$ (the final node in Figure~\ref{fig:errorPropagation}), we discuss the choice of truncating the Peano-Baker series at $n=2$. Introducing additional terms in the approximation (i.e., truncating after a larger number of summands) would lead to a higher power of $\Delta t_k$ in the upper bound \eqref{seriesMajoration}, which in turn would imply faster convergence. However, we rely on approximations of the summands in the Peano-Baker series rather than on the exact terms $\pbsI_n$, and these approximations retain an error of order $\mathcal{O}(\Delta t_\tmax^3)$ (see Lemmas~\ref{lemma:errI1} and \ref{lemma:errI2}). Therefore, although including additional terms in the series would suggest a higher order of convergence, this would be cancelled by the $\mathcal{O}(\Delta t_\tmax^3)$ appearing in $\hat{\pbsI}_{1,k}$ and $\hat{\pbsI}_{2,k}$.

We could also opt to truncate the Peano-Baker series at $n=1$ instead of $n=2$, i.e.\ retaining only $\idI_{n_x}$ and $\pbsI_{1}(t_{k+1}{;}t_k)$. In this case, the opposite situation would arise: we would lower the power of $\Delta t_k$ in \eqref{seriesMajoration} to $\Delta t_k^2$, and the error $\err_k^{\Phi}$ would be $\mathcal{O}(\Delta t_\tmax^2)$. Thus, we would not benefit from the third order convergence of $\hat{\pbsI}_{1,k}$. The conclusion is that if the integrals in $\pbsI_1(t_{k+1}{;}t_k)$ and $\pbsI_2(t_{k+1}{;}t_k)$ are approximated with the trapezoidal rule (which produces $\mathcal{O}(\Delta t_k^3)$ errors), then it is optimal to truncate the Peano-Baker series at $n=2$; optimal here means obtaining the highest possible order of convergence with as few summands as possible.

We now move to the analysis of the approximation error associated with $\pbsJ _k$, the final error to consider before proving Theorem~\ref{thm:error}. Recalling the definition,
\begin{equation*}
    \pbsJ_k:=\int_{t_k}^{t_{k+1}}\Phi(t_k{;}\tau)\cdot \nabla_pf(\tau,\bfx,\bfp)d\tau,
\end{equation*}
we note that an application of the trapezoidal rule will lead to the state-transition matrix $\Phi(t_k{;}t_{k+1})$—the inverse of $\Phi(t_{k+1}{;}t_k)$ (see Proposition~\ref{prop:Phi})— appearing. However, for greater efficiency, we compute $\Phi(t_k{;}t_{k+1})$ by the Peano-Baker series, instead of computing the inverse $\Phi(t_{k+1}{;}t_k)^{-1}$. In particular, we observe that $\Phi(t_k{;}t_{k+1})$ can be obtained from $\Phi(t_{k+1}{;}t_k)$ by interchanging the limits of integration in each term $\pbsI_n$ (see \eqref{PBS} and \eqref{recursiveTerm}); interchanging the limits of integration, does not change the norm of an integral. Similarly, the approximations of the $\pbsI_n$s by the trapezoidal rule are the same (up to the sign) for both $\Phi(t_{k+1}{;}t_k)$ and $\Phi(t_k{;}t_{k+1})$; thus, they have the same norm. By replicating the arguments we applied to $\Phi(t_{k+1}{;}t_k)$ throughout Lemmas~\ref{lemma:errI1}-\ref{lemma:errPhiOrder}, also on its inverse $\Phi(t_k{;}t_{k+1})$, we obtain the same bound for the error term $\err_k^{\Phi,\tinv}$: as $\Delta t_\tmax\to0$,
\begin{equation*}
\left\|\err_k^{\Phi,\tinv}\right\|=\left\|\Phi(t_k{;}t_{k+1})-\hat{\Phi}(t_k;t_{k+1})\right\|=\mathcal{O}(\Delta t_\tmax^3).
\end{equation*}

\begin{lemma}
\label{lemma:errJOrder}
Suppose that Assumptions~\ref{as:errorODE}, \ref{as:boundedH}, \ref{as:time} hold. Then, the error $\err_k^{\trap,\pbsJ}$ associated with the approximation of the integral $\pbsJ_k$ by the trapezoidal rule is
\begin{equation}
    \label{errTrapJOrder}
    \left\|\err_k^{\trap,\pbsJ}\right\|=\mathcal{O}(\Delta t_k^3),
\end{equation}
and the whole error in the approximation of $\pbsJ_k$ with $\hat{\pbsJ}_k$ is
\begin{equation}
    \label{errJOrder}
    \left\|\err_k^{\pbsJ}\right\|=\mathcal{O}(\Delta t_\tmax^3),
\end{equation}
as $\Delta t_\tmax\to0$.

\end{lemma}
\begin{proof}
As in the proofs of Lemmas~\ref{lemma:errI1} and \ref{lemma:errI2}, we can invoke the error of the trapezoidal rule \eqref{trapezoidalRuleError} and obtain
\begin{align*}
    \left(\err_k^{\trap,\pbsJ}\right)_i^j=-\frac{\Delta t_k^3}{12}\cdot\left(\frac{d^2}{dt^2}\left(\Phi(t_k{;}t)\cdot\nabla_pf(t,\bfx,\bfp)\right)\right)_i^j\Bigg\vert_{t=t_{\xi}},
\end{align*}
for some $\xi\in(t_k{;}t_{k+1})$. The proof of \eqref{errTrapJOrder} is now analogous to Lemmas \ref{lemma:errI1} and \ref{lemma:errI2} and we omit the details.

Moving to the error $\err_k^\pbsJ$, based on approximating the integral $\pbsJ_k$ with the trapezoidal rule, we first expand the error similar to what was done for $\err_k^{\pbsI_1}$ and $\err_k^{\pbsI_2}$:
\begin{align}
    \err_k^{\pbsJ} &= \int_{t_k}^{t_{k+1}}\Phi(t_k{;}\tau)\cdot \nabla_pf(\tau,\bfx,\bfp)d\tau - \frac{\Delta t_k}{2}\cdot(\nabla_p\hat{f}_k+\hat{\Phi}(t_{k}{;}t_{k+1})\cdot\nabla_p\hat{f}_{k+1}) \nonumber\\
    &= \int_{t_k}^{t_{k+1}}\Phi(t_k{;}\tau)\cdot \nabla_pf(\tau,\bfx,\bfp)d\tau \\ 
    &\quad -\frac{\Delta t_k}{2}\cdot(\nabla_pf(t_k,\bfx,\bfp)+\Phi(t_{k}{;}t_{k+1})\cdot\nabla_pf(t_{k+1},\bfx,\bfp) \nonumber\\
    & \quad +\frac{\Delta t_k}{2}\cdot(\nabla_pf(t_k,\bfx,\bfp)+\Phi(t_{k}{;}t_{k+1})\cdot\nabla_pf(t_{k+1},\bfx,\bfp) \nonumber \\
    & \quad - \frac{\Delta t_k}{2}\cdot(\nabla_p\hat{f}_k+\hat{\Phi}(t_{k}{;}t_{k+1})\cdot\nabla_p\hat{f}_{k+1}) \nonumber\\
    &=\err_k^{\trap,\pbsJ}+\frac{\Delta t_k}{2}\cdot \left(\err_k^{\nabla_p}+\Phi(t_k{;}t_{k+1})\cdot\nabla_pf(t_{k+1},\bfx,\bfp)-\hat{\Phi}(t_k{;}t_{k+1})\cdot\nabla_p\hat{f}_{k+1}\right)\nonumber\\
    &=\err_k^{\trap,\pbsJ}+\frac{\Delta t_k}{2}\cdot \bigl(\err_k^{\nabla_p}+\Phi(t_k{;}t_{k+1})\cdot\nabla_pf(t_{k+1},\bfx,\bfp)-\Phi(t_k{;}t_{k+1})\cdot\nabla_p\hat{f}_{k+1}\nonumber\\
    & \quad +\Phi(t_k{;}t_{k+1})\cdot\nabla_p\hat{f}_{k+1}-\hat{\Phi}(t_k{;}t_{k+1})\cdot\nabla_p\hat{f}_{k+1}\bigr)\nonumber\\
    &=
    \err_k^{\trap,\pbsJ}+\frac{\Delta t_k}{2}\cdot \left(\err_k^{\nabla_p}+\Phi(t_k{;}t_{k+1})\cdot \err_{k+1}^{\nabla_p}+\err_k^{\Phi,\tinv}\cdot\nabla_p\hat{f}_{k+1}\right)\label{errJ}.
\end{align}
From the previous results and Assumption~\ref{as:time1p}, the terms in the parenthesis in \eqref{errJ} are $\mathcal{O}(\Delta t_\tmax^2)$, hence
\begin{equation*}
\left\|\err_k^{\pbsJ}\right\|\le \left\|\err_k^{\trap,\pbsJ}\right\|+\mathcal{O}(\Delta t_\tmax^3)=\mathcal{O}(\Delta t_\tmax^3),
\end{equation*}
from which the claim \eqref{errJOrder} follows .
\end{proof}

With Lemma~\ref{lemma:errJOrder}, we now have estimates for all errors that propagate—as illustrated in Figure~\ref{fig:errorPropagation}—into the approximation $\hat{\bfS}_{k+1}$ of the sensitivity matrix at time step $t_{k+1}$. We are now ready to prove Theorem~\ref{thm:error}, the characterisation of the error in the approximation of the sensitivity matrix $\hat{\bfS}_{k+1}$ at time step $t_{k+1}$.

\begin{proof}[Proof of Theorem~\ref{thm:error}]
To estimate the error $\err_{k+1}^\bfS$ in the approximation of $\bfS$, we start from the exact expression for $\bfS(t_{k+1})$, given in \eqref{exactSensitivity}:
\begin{equation*}
    \bfS(t_{k+1}) = \Phi(t_{k+1}{;}t_k)\cdot\left(\bfS(t_k)+\pbsJ_k\right).
\end{equation*}
The approximation naturally takes a similar recursive form, using the approximations $\hat \Phi (t_{k+1}{;}t_k)$ and $\hat \pbsJ _k$,
\begin{equation*}
    \hat \bfS _{k+1} = \hat \Phi (t_{k+1}{;}t_k) \cdot\left(\hat \bfS _k+\hat \pbsJ_k\right).
\end{equation*}
In order to analyse the associated error, first we replace every exact term in the definition of $\bfS (t_{k+1})$ with the sum of the corresponding  approximation and error; for example, we replace $\bfS(t_{k+1})$ with $\hat{\bfS}_{k+1}+\err_{k+1}^{\bfS}$. This leads to the following relation for the approximating terms and errors, implicitly defining the error $\err_{k+1}^\bfS$,
\begin{align*}
    \hat{\bfS}_{k+1}+\err_{k+1}^\bfS&=\left(\hat{\Phi}(t_{k+1}{;}t_k)+\err_k^{\Phi}\right)\cdot\left(\left(\hat{\bfS}_k+\err_k^{\bfS}\right)+\left(\hat{\pbsJ}_k+\err_k^\pbsJ\right)\right).
\end{align*}
By expanding the product, we identify the expression $\hat{\Phi}(t_{k+1{;}t_k})\cdot \hat{\bfS}_k+\hat{\pbsJ}_k$ on the right-hand side, which cancels $\hat{\bfS}_{k+1}$ on the left-hand side.

As a result, the error in the sensitivity matrix can be expressed as
\begin{align*}
    \err_{k+1}^\bfS&=\hat{\Phi}(t_{k+1}{;}t_k)\cdot\left(\err_k^\bfS+\err_k^{\pbsJ}\right)+\err_k^{\Phi}\cdot \left(\hat{\bfS}_k+\err_k^{\bfS}+\hat{\pbsJ}_k+\err_k^{\pbsJ}\right)\\
    &=\hat{\Phi}(t_{k+1}{;}t_k)\cdot \err_k^\bfS+\hat{\Phi}(t_{k+1}{;}t_k)\cdot \err_k^{\pbsJ}+\err_k^{\Phi}\cdot \left(\bfS(t_k)+\pbsJ_k\right),
\end{align*}
where in the last equality we re-introduced the exact terms for the sensitivity matrix at time step $t_k$ and the exact integral $\pbsJ_k$. This recursive expression can be expanded, using that $\err_0^\bfS=0$,
\begin{equation*}
\err_{k+1}^\bfS = \sum_{h=0}^k\left(\prod_{j=h+1}^k\hat{\Phi}(t_{j+1}{;}t_j)\cdot\left(\hat{\Phi}(t_{h+1}{;}t_h)\cdot \err_h^{\pbsJ}+\err_h^\Phi\cdot(\bfS(t_h)+\pbsJ_h)\right)\right).
\end{equation*}
Taking the norm of the error, we obtain the following bound
\begin{align}
    \label{inequalityErrS}
   & \left\|\err_{k+1}^\bfS\right\| \nonumber \\
   &\quad \le \sum_{h=0}^k\left(\prod_{j=h+1}^k\left\|\hat{\Phi}(t_{j+1}{;}t_j)\right\|\cdot\left(\left\|\hat{\Phi}(t_{h+1}{;}t_h)\right\|\cdot \left\|\err_h^{\pbsJ}\right\|+\left\|\err_h^\Phi\right\|\cdot(\left\|\bfS(t_h)\right\|+\left\|\pbsJ_h\right\|)\right)\right).
\end{align}

To understand the convergence rate of the error $\err _{k+1} ^\bfS$, it now suffices to consider the terms on the right-hand side of \eqref{inequalityErrS}.

We start by considering $\hat \Phi (t_{j+1}{;} t_j)$. From the definition  \eqref{PhiApproximation}, applying the norm operator we obtain
\begin{equation*}
\left\|\hat{\Phi}(t_{j+1}{;}t_j)\right\|\le 1 + \Delta t_j C_{\nabla_x} + \frac{\Delta t_j^2}{2}C_{\nabla_x}^2,
\end{equation*}
for every $j=0,\dots, k${;} here $C_{\nabla _x} = \left\| \nabla _x f (\bfx (\cdot)) \right\| _{[0,T]}$, which is finite by Assumption~\ref{as:time1}. This term is $\mathcal{O}(1)$ for $\Delta t_j\to 0$. For arguments used later in the proof, it is convenient to define $C_{\Delta t}:=C_{\nabla_x}\cdot\left(1+\frac{\Delta t_{\tmax}C_{\nabla_x}}{2}\right)$, which is not a constant, but depends on $\Delta t_{\tmax}$ and $C_{\Delta t}\to C_{\nabla_x}$ as $\Delta t_{\tmax}\to 0$. With this definition we can rewrite the upper bound as
\begin{equation*}
\left\|\hat{\Phi}(t_{j+1}{;}t_j)\right\|\le 1 +C_{\Delta t}\Delta t_j.
\end{equation*}

From Lemmas~\ref{lemma:errJOrder} and \ref{lemma:errPhiOrder} we know that the error terms $\err_h^\pbsJ$ and $\err_h^{\Phi}$ are $\mathcal{O}(\Delta t_\tmax^3)$.

By Assumption~\ref{as:matrices}, the exact sensitivity matrix at time step $t_h$ can be bounded as $\left\|\bfS(t_h)\right\|\le C_\bfS$, hence $\left\|\bfS(t_h)\right\|=\mathcal{O}(1)$, and $\left\|\Phi(t_h;t)\cdot\nabla_pf(t,\bfx,\bfp)\right\|_{[0,T]}\le C_{\pbsJ}$ for every $h=0,\dots,k$. It follows that the integral term $\pbsJ_h$ can be bounded as
\begin{equation*}
\left\|\pbsJ_h\right\| = \left\|\int_{t_h}^{t_{h+1}}\Phi(t_h{;}\tau)\cdot\nabla_pf(\tau,\bfx,\bfp)d\tau\right\|\le C_{\pbsJ}\Delta t_h,
\end{equation*}
and we conclude that $\left\|\pbsJ_h\right\|=\mathcal{O}(\Delta t_h)$.
As a consequence, the term
\begin{equation*}
\left\|\hat{\Phi}(t_{h+1}{;}t_h)\right\|\cdot \left\|\err_h^{\pbsJ}\right\|+\left\|\err_h^\Phi\right\|\cdot(\left\|\bfS(t_h)\right\|+\left\|\pbsJ_h\right\|),
\end{equation*}
in \eqref{inequalityErrS} is $\mathcal{O}(\Delta t_\tmax^3)$ and we can define a new constant $0<C<\infty$, such that
\begin{align*}
\left\|\err_{k+1}^\bfS\right\|& \le\sum_{h=0}^k\left(\prod_{j=h+1}^k\left\|\hat{\Phi}(t_{j+1}{;}t_j)\right\|\cdot C\Delta t_\tmax^3\right) \\
&\le\sum_{h=0}^k\left(\prod_{j=h+1}^k\left(1 + C_{\Delta t}\Delta t_j\right)\cdot C\Delta t_\tmax^3\right).
\end{align*}
Moreover, we can bound $\Delta t_j$ by $\Delta t_{\tmax}$, which gives an upper bound for $\err_{k+1}^\bfS$ in terms of $\Delta t_{\tmax}$,
\begin{equation*}
\left\|\err_{k+1}^\bfS\right\|\le\sum_{h=0}^k\left(1 +C_{\Delta t} \Delta t_{\tmax}\right)^{k-h}\cdot C\Delta t_{\tmax}^3.
\end{equation*}
The sum $\sum_{h=0}^k\left(1 + C_{\Delta t}\Delta t_{\tmax}\right)^{k-h}$ can be rewritten as $\sum_{h=0}^k\left(1 + C_{\Delta t}\Delta t_{\tmax}\right)^{h}$, which admits the closed formula
\begin{align}
    \sum_{h=0}^k\left(1 + C_{\Delta t}\Delta t_{\tmax}\right)^{h} &= \frac{1-\left(1 + C_{\Delta t}\Delta t_{\max}\right)^{k+1}}{1-\left(1 +C_{\Delta t} \Delta t_{\max}\right)} \nonumber\\
    &=\frac{\left(1 + C_{\Delta t}\Delta t_{\max}\right)^{k+1}-1}{C_{\Delta t} \Delta t_{\max}}. \label{geomSum}
\end{align}
First, we consider the term $\left(1 + C_{\Delta t}\Delta t_{\max}\right)^{k+1}$ in the numerator. Since the exponent $k$ runs over $k=0,\dots,K-1$, and the term inside the parenthesis is non-negative,
\begin{equation}
\label{inequalityExp}
\left(1 + C_{\Delta t}\Delta t_{\max}\right)^{k+1}\le \left(1 + C_{\Delta t}\Delta t_{\max}\right)^K.
\end{equation}

Second, the total number of time steps $K$ can be bounded from above,
\begin{align*}
    K\le \frac{T}{\Delta t_{\tmin}}=\frac{T}{\Delta t_{\tmin}}\frac{\Delta t_\tmax}{\Delta t_\tmax} = \Xi\frac{T}{\Delta t_\tmax}.
\end{align*}

Inserting this in \eqref{inequalityExp} yields
\begin{equation*}
\left(1 + C_{\Delta t}\Delta t_{\max}\right)^{K}\le \left(1 + C_{\Delta t}\Delta t_{\max}\right)^{ \Xi \frac{T}{\Delta t_{\tmax}}}.
\end{equation*}

To finish the proof, we use that the function $g:(0,\infty)\to\bR$, $g(x)=(1+x)^{\frac{1}{x}}$ is monotonically decreasing and that $\lim_{x\to0^+}g(x)=e$. Since $C_{\Delta t}\Delta t_{\tmax}$ is increasing in $\Delta t_{\tmax}$, it follows that
\begin{equation*}
(1+C_{\Delta t}\Delta t_{\tmax})^{\frac{1}{C_{\Delta t}\Delta t_{\tmax}}}<e,
\end{equation*}
hence
\begin{equation*}\left(1 + C_{\Delta t}\Delta t_{\max}\right)^{ \Xi \frac{T}{\Delta t_{\tmax}}}<e^{\Xi C_{\Delta t}T}.
\end{equation*}
The previous results show that \eqref{geomSum} is bounded from above by
\begin{equation*}
\frac{e^{\Xi C_{\Delta t}T}-1}{C_{\Delta t}\Delta t_{\tmax}}.
\end{equation*}
Lastly, inserting this into the upper bound for $\err _{k+1} ^\bfS$ yields
\begin{equation*}
\left\|\err_{k+1}^\bfS\right\|\le \frac{e^{\Xi C_{\Delta t}T}-1}{C_{\Delta t}\Delta t_{\tmax}}\cdot C\Delta t_{\tmax}^3=\frac{C(e^{\Xi C_{\Delta t}T}-1)}{C_{\Delta t}}\Delta t_{\tmax}^2.
\end{equation*}
Since $C_{\Delta t}=\mathcal{O}(1)$, this proves the claimed bound \eqref{errSOrder}.

\end{proof}

\begin{remark}
    Collecting all the error estimates determined in the proofs of Lemmas~\ref{lemma:Jacobians}-\ref{lemma:errJOrder} and the proof of Theorem~\ref{thm:error}, we obtain that the constant inside the $\mathcal{O}$ in \eqref{errSOrder} in Theorem \ref{thm:error} is 
    \begin{align*}
        \frac{\left(e^{\Xi C_{\Delta t}T}-1\right)}{C_{\Delta t}}\cdot C,
    \end{align*}
    with $\Xi, C_{\Delta t}, T$ and $C$ defined above.

    Finally, Corollary~\eqref{cor:errorPBSR} is a direct consequence of the following lemma, which shows that the error in the linearly interpolated numerical ODE solution $\hat\bfx$ is $\mathcal{O}(\Delta t_{\tmax}^2)$. This is the same order that we assume for $\err^{\text{ODE}}_k$ in Assumption~\ref{as:errorODE}. As a consequence, Lemmas~\ref{lemma:Jacobians}-\ref{lemma:errJOrder} continue to hold, and therefore the same proof as Theorem~\ref{thm:error} can be applied also when a refinement of the time intervals is performed and the ODE solution $\hat\bfx$ is linearly interpolated.
    
\begin{lemma}
\label{lem:interpError}
Assume that the map $t\mapsto \frac{d}{dt}f(t,\bfx,\bfp)$ is bounded with respect to $\| \cdot \| _{[0,T]}$. For $i\in\{1,\dots,n_{\text{int}}\}$ let $t_{k,i}=t_k+\frac{i-1}{n_{\text{int}}}\Delta t_{k}$ and
\begin{align*}
    \hat\bfx_{k,i}=\hat\bfx_k+\frac{i-1}{n_{\text{int}}}(\hat\bfx_{k+1}-\hat\bfx_k).
\end{align*}
Let 
\begin{align*}
    \err_{k,i}^{\text{ODE}}=\bfx(t_{k,i})-\hat\bfx_{k,i}
\end{align*}
be the error of approximating the exact ODE solution $\bfx$ at time $t_{k,i}$ with the interpolation $\hat\bfx_{k,i}$ of the ODE numerical solution, $\{\hat\bfx_k\}_{k=0}^K$. Under Assumption~\ref{as:errorODE},
\begin{align*}
    \|\err_{k,i}^{\text{ODE}}\|=\mathcal{O}(\Delta t_\tmax^2).
\end{align*}
\end{lemma}
\begin{proof}
Let 
\begin{align*}
    \bfx^{\text{interp}}_{k,i} = \bfx(t_k)+\frac{i-1}{n_{\text{int}}}(\bfx(t_{k+1})-\bfx(t_k))
\end{align*}
be the interpolation of the exact ODE solution $\bfx$ at time $t_{k,i}$. Because $t\mapsto \frac{d}{dt}f(t,\bfx,\bfp)$, i.e. $t\mapsto \frac{d^2}{dt^2}\bfx(t)$, is bounded with respect to $\| \cdot \| _{[0,T]}$, the linear interpolation error satisfies
\begin{align*}
   \|\bfx(t_{k,i}) -  \bfx^{\text{interp}}_{k,i}\|=\mathcal{O}(\Delta t_k^2),
\end{align*} 
as showed in \cite[Chap.~5]{ScarboroughJamesBlaine1950Nma}. Observe that
    \begin{align*}
        \bfx^{\text{interp}}_{k,i}-\hat\bfx_{k,i}&=\bfx(t_k)+\frac{i-1}{n_{\text{int}}}(\bfx(t_{k+1})-\bfx(t_k)) - \hat\bfx_k-\frac{i-1}{n_{\text{int}}}(\hat\bfx_{k+1}-\hat\bfx_k)\\
        &=\err^{\text{ODE}}_k + \frac{i-1}{n_{\text{int}}}\left(\err^{\text{ODE}}_{k+1}-\err^{\text{ODE}}_{k}\right).
    \end{align*}
    Consequently,
    \begin{align*}
        \| \bfx^{\text{interp}}_{k,i}-\hat\bfx_{k,i}\|\le \|\err^{\text{ODE}}_k\|+ \frac{i-1}{n_{\text{int}}}\left(\|\err^{\text{ODE}}_{k+1}\|+\|\err^{\text{ODE}}_{k}\|\right)=\mathcal{O}(\Delta t_\tmax^2),
    \end{align*}
    by Assumption~\ref{as:errorODE}.
We conclude that 
\begin{align*}
     \|\err_{k,i}^{\text{ODE}}\|&=\|\bfx(t_{k,i})-\hat\bfx_{k,i}\|\le \|\bfx(t_{k,i})-\bfx^{\text{interp}}_{k,i}\|+\|\bfx^{\text{interp}}_{k,i}-\hat\bfx_{k,i}\|,
\end{align*}
where both error on the right hand side of the last inequality are of order $\mathcal{O}(\Delta t_\tmax^2)$, hence the result.
\end{proof}

\end{remark}

\section{Numerical results}
\label{sec:numerical}

In this section we complement the theoretical analysis in Sections~\ref{sec:genSolution}-\ref{sec:proofErrEst} with numerical experiments, illustrating the performance of the proposed sensitivity approximation methods in a number of examples. First, in Section~\ref{subsec:implementation} we describe some implementation details of the PBSR and Exp algorithms, first presented in Section~\ref{sec:stability}. Next, in Section \ref{sec:experiments} we outline how the accuracy of the different methods is evaluated and show numerical results for four examples: two biological models, referred to as PKA and CaMKII, a random linear system, and a model with a limit cycle (the Chua's circuit model).
The GitHub repository associated with this paper\footnote{\url{https://github.com/federicamilinanni/JuliaSensitivityApproximation}} presents more implementation details; another repository\footnote{\url{https://github.com/a-kramer/CSensApprox}} contains an alternative implementation in C of the methods described in this manuscript and sets up more examples.

\subsection{Implementation details}
\label{subsec:implementation}
To test the accuracy and run time of the sensitivity approximation methods proposed in this paper (described in Section~\ref{sec:stability}), we have implemented the methods in the Julia language. For this purpose, we used the Julia packages \texttt{DifferentialEquations.jl} and \texttt{Sundials.jl} to solve the $n_x-$dimensional ODE system for the state variable $\bfx$. In particular, we used the solver \texttt{CVODE\_BDF} \, from \texttt{Sundials.jl}, for which we set absolute and relative tolerances of 1e-6 and 1e-5, respectively. The ODE solver uses an adaptive time stepping algorithm, and therefore returns the approximate ODE solution $\{\hat\bfx_k\}$ on a non-uniform time grid $\{t_k\}$, where time steps $[t_k,t_{k+1}]$ can be arbitrarily large (within the set time span). 

As discussed in Section~\ref{sec:stability}, the size $\Delta t_k$ of the time intervals $[t_k,t_{k+1}]$ does not give rise to stability concerns in the case of the Exp algorithm (Algorithm~\ref{alg:ExpAlg}), as this method relies only on the exponential formula~\eqref{solutionAtEquilibrium}, which is stable for all choices of $\Delta t_k$. For the PBSR algorithm (Algorithm~ \ref{alg:PBSR}), in the case of stiff problems, the time intervals $[t_k,t_{k+1}]$ can be too large for the PBS formula~\eqref{eq:Approx}, causing stability issues. Therefore, the time intervals are refined into $n_\tint$ sub-intervals $\{[t_{k,i},t_{k,i+1}]\}_{i=1}^{n_\tint}$ of length $\Delta t_k/n_\tint$; note that we suppress the $n_\tint$s dependence on $k$ in the notation. To determine a good choice for the number of sub-intervals $n_{\tint}$, we consider some of the error terms in the PBS sensitivity approximation, analysed in Section~\ref{sec:proofErrEst}. In particular, a potentially problematic term is the error due to the truncation of the Peano-Baker series, i.e., the omission of $\sum_{n=3}^\infty \pbsI_n(t_{k,i+1};t_{k,i})$, on the sub-interval $[t_{k,i},t_{k,i+1}]$ (see \eqref{errPhi}). With majorations similar to \eqref{majorationNormPBS}, we obtain the upper bound
\begin{equation*}
  \left\|  \sum_{n=3}^\infty\pbsI(t_{k,i+1};t_{k,i})\right\| \le \sum_{n=3}^\infty\frac{1}{n!} \left(\frac{\Delta t_{k}}{n_\tint} C\right)^n=e^{\frac{\Delta t_{k}}{n_\tint} C}-1-\left(\frac{\Delta t_{k}}{n_\tint} C\right)-\frac{1}{2}{\left(\frac{\Delta t_{k}}{n_\tint} C\right)^2},
\end{equation*}
where $C=\|\nabla_xf\|_{[t_{k,i},t_{k,i+1}]}$. One way to ensure that this value is sufficiently small is to require that $\frac{\Delta t_k}{n_{\tint}}C \lesssim 0.1$. Based on this, the choice for the number of sub-intervals becomes $n_{\tint} = \lceil 10\Delta t_k \|\nabla_xf(\hat{\bfx}_k)\| \rceil$.

Refining the time grid increases the running time of the algorithm, in particular when the time span of the whole simulation is of the order of hundreds or thousands time units, as in the examples considered in the following subsections. To make the algorithm more efficient, on time intervals $[t_k,t_{k+1}]$ where the proposed $n_\tint$ is deemed too large, we apply the exponential formula~\eqref{solutionAtEquilibrium} instead of refining and using the PBS formula~\eqref{eq:Approx}. As outlined in Section~\ref{sec:stability}, this means that the exponential formula is used in two situations:
for time intervals $[t_k,t_{k+1}]$ for which (i) $n_\tint$ is ``too large'', or (ii) $\nabla_xf$ and $\nabla_pf$ are approximately constant. In our implementation of the PBSR algorithm we consider $n_\tint$ ``too large'' when it exceeds the value $100$, and we treat the Jacobians $\nabla_xf$ and $\nabla_pf$ as constant when 
\begin{align*}
    \frac{\|\nabla_x\hat f_{k+1}-\nabla_x\hat f_k\|}{\|\nabla_x\hat f_k\|}< 10^{-4}\quad\text{and}\quad\frac{\|\nabla_p\hat f_{k+1}-\nabla_p\hat f_{k}\|}{\|\nabla_p\hat f_k\|}< 10^{-4},
\end{align*}
respectively. 
Combining the two criteria, the \texttt{if}-statement in Algorithm~\ref{alg:PBSR} that triggers the use of the exponential formula~\eqref{solutionAtEquilibrium} is implemented as:
\begin{equation}
    \label{eq:ifStat}
    \begin{split}
        \texttt{if}&\quad n_{\tint}>100\quad \texttt{OR}\\
    &\frac{\|\nabla_x\hat f_{k+1}-\nabla_x\hat f_k\|}{\|\nabla_x\hat f_k\|}< 10^{-4} \quad\texttt{AND}\quad \frac{\|\nabla_p\hat f_{k+1}-\nabla_p\hat f_{k}\|}{\|\nabla_p\hat f_k\|}< 10^{-4}.
    \end{split}
\end{equation}

The exponential formula~\eqref{solutionAtEquilibrium} requires an implementation of matrix exponentiation. Each programming language offers a built in\footnote{Alternatively, a linear algebra package.}
  function for this, typically named \texttt{expm}
  (we use \texttt{Base.exp()} in Julia). Such methods often rely on the scaling and squaring method; Julia specifically cites \cite{Higham2009}. An unstable/bad
  implementation of \texttt{expm} could lead to large numerical
  inaccuracies, however this is not something we have observed in practice.

  Lastly, in the numerical examples we also implement the Forward Sensitivity method to approximate the sensitivity matrix. For this purpose we use the Julia package \texttt{DiffEqSensitivity.jl}.

\subsection{Simulation experiments}
\label{sec:experiments}

In the numerical experiments that follow, we evaluate the accuracy and speed of the Exp and PBSR algorithms (Algorithms~\ref{alg:PBSR} and \ref{alg:ExpAlg}). The accuracy is measured in two different ways: one is to compare the two algorithms with the FS method, which is considered to have a high accuracy, and one is to consider minor perturbations, with respect to the parameter vector, of the ODE solution.

As our first performance measure, after solving the ODE system~\eqref{eq:ODE} for $\bfx$, we consider the time points $t_k, k=0,\dots,K$, returned by the ODE solver, and for each time point $t_k$ compute approximations of the sensitivity matrix $\bfS$ using the PBSR and Exp algorithms: $\hat{\bfS}^{\text{PBSR}}_k$ (PBSR), and $\hat{\bfS}^{\text{Exp}}_k$ (Exp). Besides, we apply the FS method to obtain the FS approximation of the sensitivity matrix. The time steps used by the FS method are in general different from those returned by the ODE solver mentioned above; therefore, we linearly interpolate the FS sensitivity matrix approximation in the same time points $\{t_k\}$ used for the PBSR and Exp algorithms, thus obtaining the approximation $\{\hat \bfS_k^{\text{FS}}\}$.

Next, for each time point $t_k$, we compute the scalar relative error of the PBSR ($\rerr^{\text{PBSR}}_k$) and Exp ($\rerr^{\text{Exp}}_k$) algorithms with respect to the FS method:
\begin{equation*}
  \rerr^{\text{PBSR}}_k = \frac{\|\hat{\bfS}^{\text{PBSR}}_k-\hat{\bfS}^{\text{FS}}_k\|}{\|\hat{\bfS}^{\text{FS}}_k\|}, \quad\quad \rerr^{\text{Exp}}_k = \frac{\|\hat{\bfS}^{\text{Exp}}_k-\hat{\bfS}^{\text{FS}}_k\|}{\|\hat{\bfS}^{\text{FS}}_k\|}\,.
\end{equation*}

The method just outlined treats $\hat{\bfS}^{\text{FS}}_k$, obtained using the FS method, as the true sensitivity matrix at time $t_k$. However, the FS method also yields an approximation, and thus $\hat{\bfS}^{\text{FS}}_k$ comes with an error. We therefore introduce an alternative, in a sense more objective, way to evaluate the accuracy of the proposed methods. The idea is to use the sensitivity matrix to approximate the solution of the underlying ODE model after making a small perturbation $\delta_{\bfp}$ in the parameter vector $\bfp$. Consider making such a perturbation by either adding or subtracting (an appropriate) $\delta _\bfp$ from $\bfp$: from a Taylor expansion, we have,
\begin{align}
  \bfx(t,\bfp+\delta_{\bfp}) &= \bfx(t,\bfp) + \bfS(t,\bfx,\bfp) \cdot \delta_{\bfp} + \frac{1}{2}\delta_{\bfp}^{\tT}\cdot\mathbb{S}(t,\bfx,\bfp)\cdot\delta_{\bfp} + \mathcal{O}(\delta_{\bfp}^3)\,, \label{eq:approx1}\\
  \bfx(t,\bfp-\delta_{\bfp}) &= \bfx(t,\bfp) - \bfS(t,\bfx,\bfp) \cdot \delta_{\bfp} + \frac{1}{2}\delta_{\bfp}^{\tT}\cdot\mathbb{S}(t,\bfx,\bfp)\cdot\delta_{\bfp} + \mathcal{O}(\delta_{\bfp}^3)\,, \label{eq:approx2}
\end{align} 
where $\mathbb{S}$ is the second order sensitivity tensor (i.e. the Hessian of $\bfx(t,\bfp)$, with respect to $\bfp$). Subtracting \eqref{eq:approx2} from \eqref{eq:approx1} leads to a natural criterion to judge how well the sensitivity works as a linear approximation of the solution with respect to parameter changes,
\begin{equation}
  \label{eq:criterion}
    \bfx(t,\bfp+\delta_{\bfp}) - \bfx(t,\bfp-\delta_{\bfp}) = 2 \bfS(t,\bfx,\bfp) \cdot \delta_{\bfp}  + \mathcal{O}(\delta_{\bfp}^3) \,.
\end{equation}
For an accurate approximation $\hat \bfS$ of $\bfS$, the difference between the left-hand side and the right-hand side of \eqref{eq:criterion}, with $\bfS$ replaced by $\hat \bfS$, should be small and this therefore works as a performance measure. In line with this, for an approximation $\hat \bfS$, we define, for a given $\delta_{\bfp}$, the following (objective) error estimate:
\begin{equation}
  \label{eq:obj-rel-error}
  \ObjRelErr_{\delta_{\bfp}}=\frac{\|\bfx(t,\bfp+\delta_{\bfp}) - \bfx(t,\bfp-\delta_{\bfp}) - 2 \hat \bfS(t,\bfx,\bfp) \cdot \delta_{\bfp}\|}{\|\epsilon + \bfx(t,\bfp+\delta_{\bfp}) - \bfx(t,\bfp-\delta_{\bfp})\|} \,,
\end{equation}
where $\epsilon$ is a small quantity, used to regularize the fraction in the cases where the denominator in \eqref{eq:obj-rel-error} is vanishingly small\footnote{Initially, at $t_0$, the sensitivity is exactly $0$ in all components, thus $\bfx(t_0,\bfp+\delta_{\bfp})=\bfx(t_0,\bfp-\delta_{\bfp})$.}. In practice, we pick several random $\delta_{\bfp}$ values and average the resulting errors $\ObjRelErr_{\delta_{\bfp}}$:

  \begin{equation}
    \label{eq:averageRelError}
    \ObjRelErr = \frac{1}{N}\sum_{i=1}^N \ObjRelErr_{\delta_\bfp^i}\,,\quad\delta_{\bfp}^i = \mathbf{h}^i \odot \bfp\,,
  \end{equation}
  where $\odot$ denotes point-wise multiplication, or Hadamard product, between the vectors $\mathbf{h}^i$ and $\bfp$ (both of dimension $n_p$), and for each $i=1,\dots,N$, $\mathbf{h}^i$ is a random vector with components
  \begin{equation*}
      h^i_j\sim\mathcal{U}(10^{-5},10^{-4}),\quad j=1,\dots,n_p.
  \end{equation*}
With this definition, the perturbations $\delta_\bfp^i$ are small compared to $\bfp$, in the sense that each component $\delta_{\bfp,j}$ satisfies $ \vert \delta _{\bfp,j} \vert /p_j\ll 1 $, $j=1,\dots,n_p$.
  
    In each of the examples in the following subsections, the error estimate $\ObjRelErr$ is computed for all the three methods (PBSR, Exp and FS) at each time $t_k$, denoted by $\ObjRelErr_k^{\text{PBSR}}$, $\ObjRelErr_k^{\text{Exp}}$ and $\ObjRelErr_k^{\text{FS}}$. Note that there is some arbitrariness in the chosen distribution for $\delta_\bfp$.
    
  A benefit of this alternative, seemingly more objective, performance measure is that it provides an estimate also of the error in the FS approximation of the sensitivity matrix, and removes the reliance on any other numerical approximation in evaluating the accuracy of the PBSR and the Exp algorithms.

\subsubsection*{Models for molecular signaling pathways: PKA and CaMKII models}
The first numerical experiments highlight the original motivation of considering models from systems biology (see Section~\ref{sec:intro}): in this section we apply the PBSR Algorithm~\ref{alg:PBSR} and the Exp Algorithm~\ref{alg:ExpAlg} to two models that describe molecular signaling pathways within neurons involved in learning and memory. In particular, these mechanisms are involved in the strengthening or weakening of neuron synapses, referred to as long term potentiation (LTP) and long term depression (LTD), respectively.

A crucial role in signaling pathways is played by protein kinases and phosphatases, thanks to their ability to, respectively, phosphorylate and dephosphorylate substrate proteins. In the two models that we consider, the phosphorylating role is performed by the cAMP-dependent protein kinase A (PKA) and the Ca$^{2+}$/calmodulin-dependent protein kinase II (CaMKII), which give the two models their names: PKA and CaMKII—for more details see references \cite{Church2021.03.14.435320, NAIR2014277, Eriksson19}.

In the two ODE models, the state vectors $\bfx(t)$ represent the concentrations of the different forms of the modelled proteins and ions at time $t$; the dimension of the state space is $n_x=11$ in the PKA model, and $n_x=21$ in the CaMKII model. In both models the parameter vector $\bfp$ corresponds to the kinetic constants that characterize the reactions in the underlying pathway. The dimension of the parameter space is $n_p=35$ in the PKA, and $n_p=59$ in the CaMKII model.

In Figures~\ref{fig:PKA} and \ref{fig:CaMKII} we show the relative errors $\rerr^{\text{PBSR}}$ (solid line) and $\rerr^{\text{Exp}}$ (dashed line) against the time step for the PKA and CaMKII model, respectively. The small gray circles and red triangles superimposed on the solid line (PBSR algorithm) refer to the time steps at which the \texttt{if}-statement is satisfied and the exponential formula~\eqref{solutionAtEquilibrium} is used: gray circles correspond to time steps where the use of the exponential formula~\eqref{solutionAtEquilibrium} is triggered because the system is approximately at equilibrium, while red triangles indicate time steps where \eqref{solutionAtEquilibrium} is used because the estimated value of $n_\tint$ is too large. We observe that the results obtained by the PBSR algorithm are $1$ to $2$ order of magnitude more accurate than the Exp algorithm at the time steps where the PBSR formula~\eqref{eq:Approx} is used (when FS is considered as providing the true sensitivity matrix).
\begin{figure}[h]
    \centering
    \includegraphics[width=\linewidth]{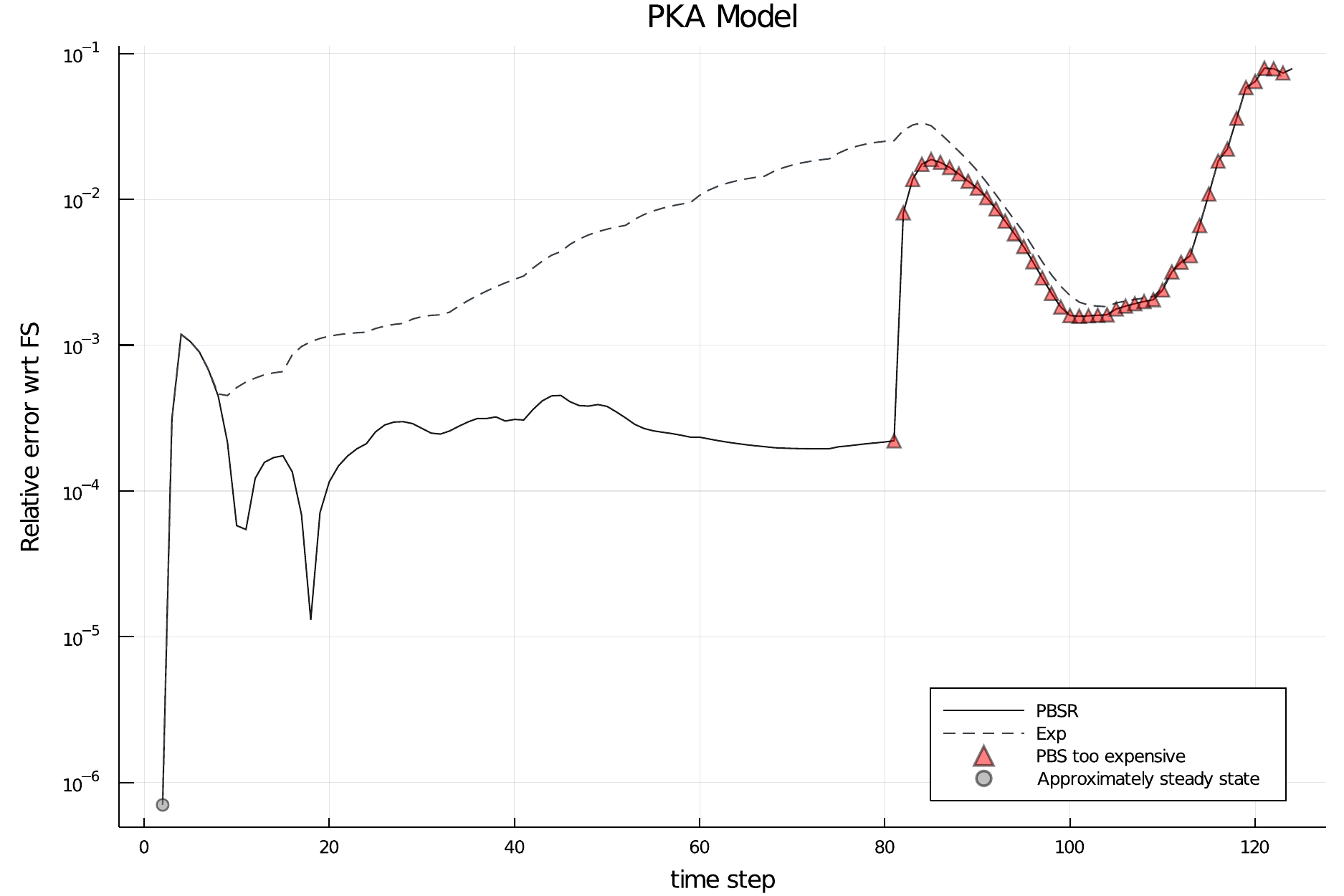}
    \caption{Relative error of the sensitivity matrix, against time step, for the PBSR (solid line) and the Exp (dashed line) algorithms for the PKA model. The simulations are run over the time span $[0,600]$. The gray circles superimposed on the solid line indicate the time steps at which within the PBSR algorithm  the exponential formula~\eqref{solutionAtEquilibrium} is used instead of the PBS formula~\eqref{eq:Approx} because the system is approximately at equilibrium; the small red triangles show when the exponential formula~\eqref{solutionAtEquilibrium} is used instead of the PBS formula~\eqref{eq:Approx} because the number of sub-intervals $n_{int}$ in the PBSR refinement is too large.}
    \label{fig:PKA}
\end{figure}

\begin{figure}[h]
    \centering
    \includegraphics[width=\linewidth]{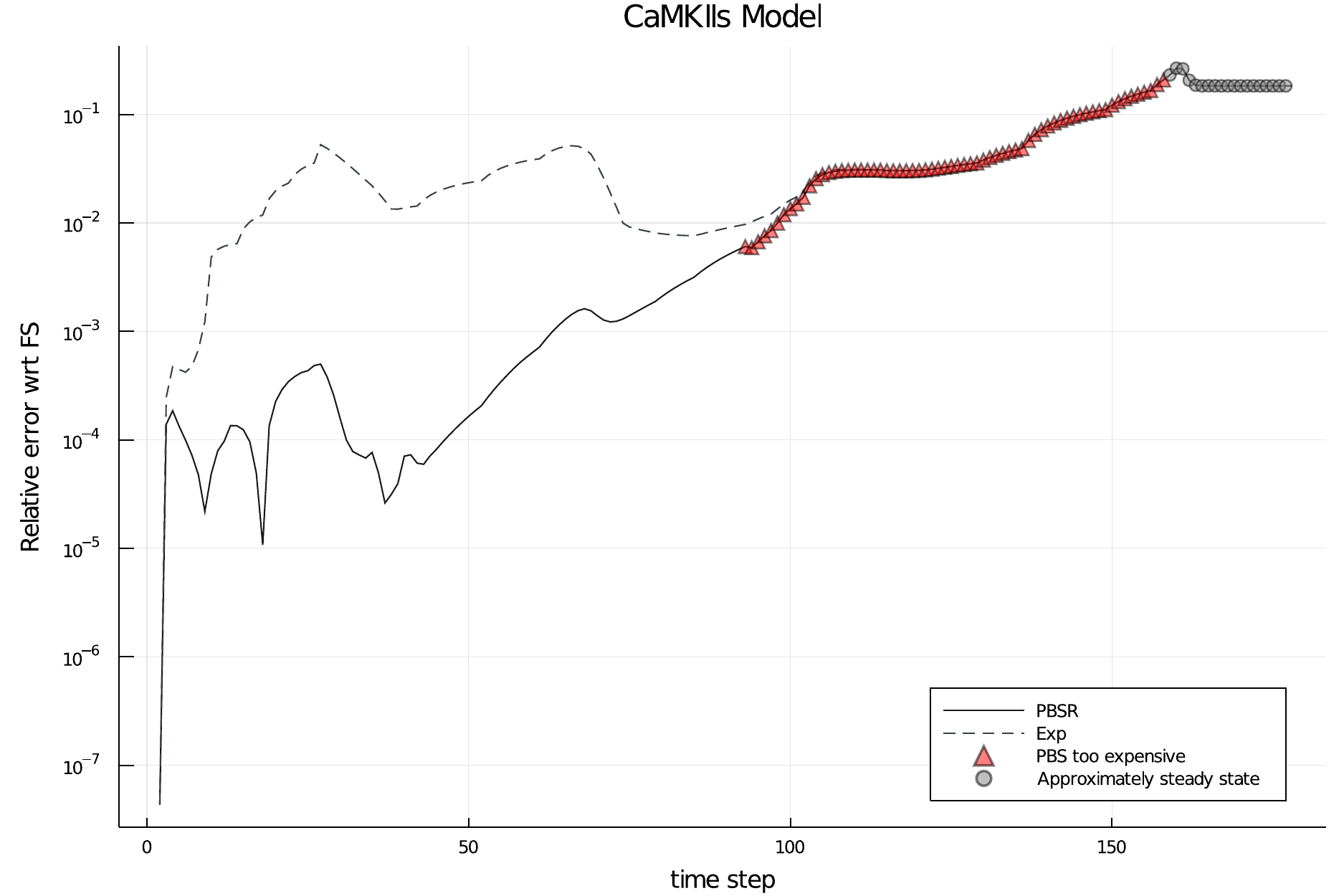}
    \caption{Relative error of the sensitivity matrix, against time step, for the PBSR (solid line) and the Exp (dashed line) algorithms for the CaMKII model. The simulations are run over the time span $[0,600]$. The gray circles superimposed on the solid line indicate the time steps at which within the PBSR algorithm the exponential formula~\eqref{solutionAtEquilibrium} is used instead of the PBS formula~\eqref{eq:Approx} because the system is approximately at equilibrium; the small red triangles show when the exponential formula~\eqref{solutionAtEquilibrium} is used instead of the PBS formula~\eqref{eq:Approx} because the number of sub-intervals $n_{int}$ in the PBSR refinement is too large.}
    \label{fig:CaMKII}
\end{figure}

Figures~\ref{fig:ObjRelErrPKA} and~\ref{fig:ObjRelErrCaMKIIs} show the objective errors $\ObjRelErr$, defined in~\eqref{eq:averageRelError}, with $N=100$. 
The blue and red shaded areas indicate simulation time points where the use of the exponential formula~\eqref{solutionAtEquilibrium} is triggered in the PBSR algorithm: blue corresponds to simulation time points where the system is approximately in equilibrium, and red indicates time points where the estimated number of sub-intervals $n_\tint$ for the refinement is too large. Note that in Figure~\ref{fig:PKA}, the PKA system is simulated in the time span $[0,600]$, during which it does not reach convergence (final time points are labelled with red triangles), whereas in Figure~\ref{fig:ObjRelErrPKA} we simulate the system over the time span $[0,60000]$, and the system then reaches equilibrium (final time points are depicted in a shaded blue area). 

We observe that in the first instants of the simulations, the errors $\ObjRelErr^{\text{PBSR}}$ and $\ObjRelErr^{\text{FS}}$ are about one order of magnitude lower than $\ObjRelErr^{\text{Exp}}$; after the initial phase, the errors in the three methods result to be close to each other, with $\ObjRelErr^{\text{Exp}}$ being slightly higher in the model of larger dimension (CaMKII), between $10^{-2}$ and $10^{-1}$. At the time points where the PBSR algorithm implements the exponential formula~\eqref{solutionAtEquilibrium} (red and blue regions in the Figures) the error $\ObjRelErr^{\text{PBSR}}$ is comparable with $\ObjRelErr^{\text{Exp}}$. Interestingly, we observe that in the larger model (CaMKII) the errors $\ObjRelErr^{\text{PBSR}}$ and $\ObjRelErr^{\text{Exp}}$ in the PBSR and Exp approximations decay to $10^{-5}$ when the system converges to equilibrium, whereas the error $\ObjRelErr^{\text{FS}}$ in the FS approximation remains of an order of $10^{-1}$. This explains why in Figure~\ref{fig:CaMKII} the errors $E^\text{PBSR}$ and $E^\text{Exp}$ in the PBSR and Exp approximations are large in the final time steps, even though the PBSR and Exp approximation have a sufficiently high accuracy: in fact, $E^\text{PBSR}$ and $E^\text{Exp}$ correspond to the relative errors of the PBSR and Exp sensitivity approximations with respect to the FS approximation, which in this case is much less accurate.

\begin{figure}[h]
  \centering
  \includegraphics[width=\textwidth]{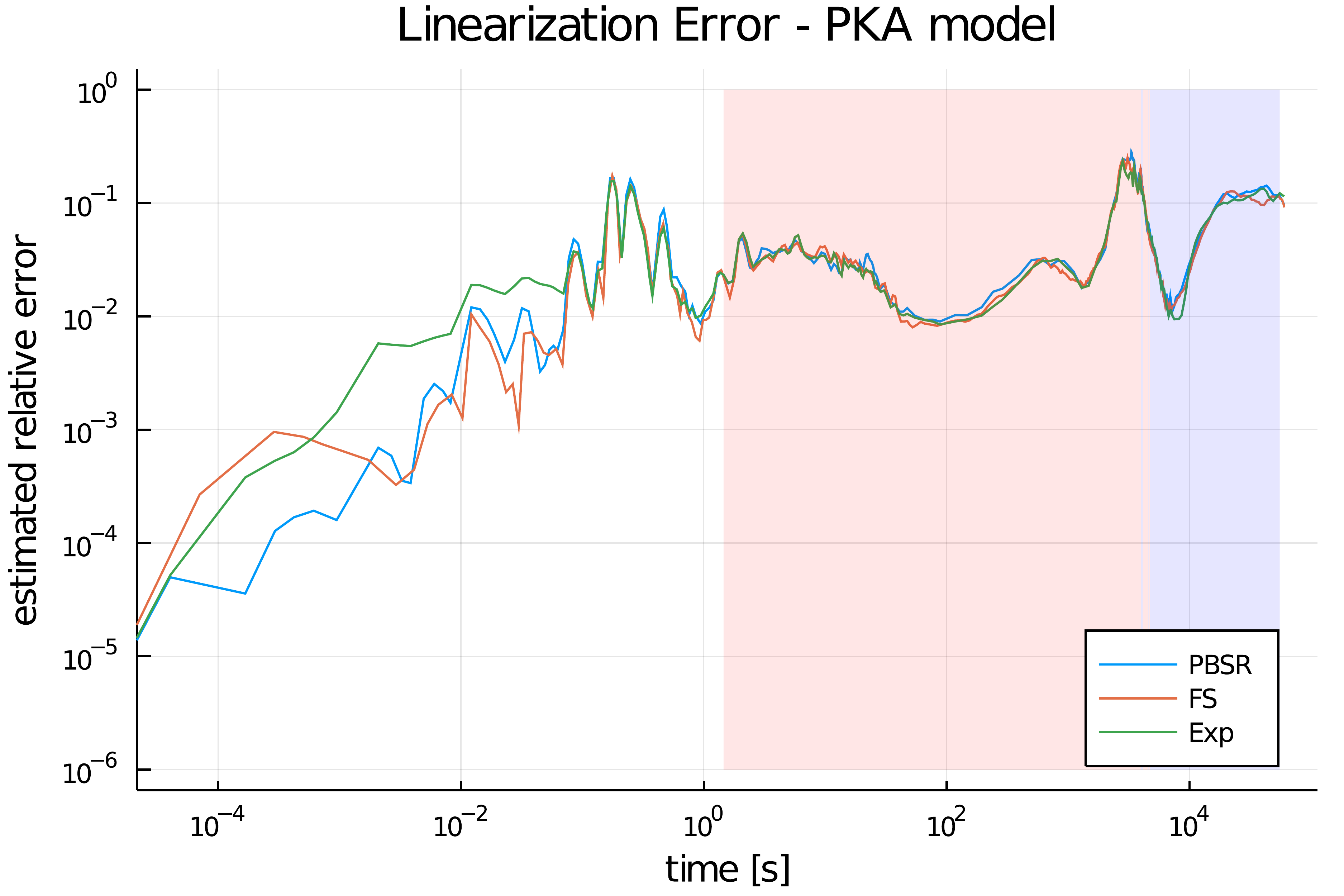}
  \caption{The estimated relative error $\ObjRelErr$ against time for the PKA
    model. The simulations are run over the time span $[0,60000]$. The blue and red shaded areas both indicate that within
    the PBSR method the exponential formula~\eqref{solutionAtEquilibrium} is used:
    \textcolor{blue}{\emph{blue}} because the system is close to a
    (stable) steady state, \textcolor{red}{\emph{red}} because
    refinement is numerically prohibitive. The error estimate $\ObjRelErr$ uses small parameter perturbations $\delta_{\bfp}$ to test how well the sensitivity matrix predicts the changes in the solution (averaged over $N=100$ small, random parameter perturbations).}
  \label{fig:ObjRelErrPKA}
\end{figure}

\begin{figure}
  \centering
  \includegraphics[width=\textwidth]{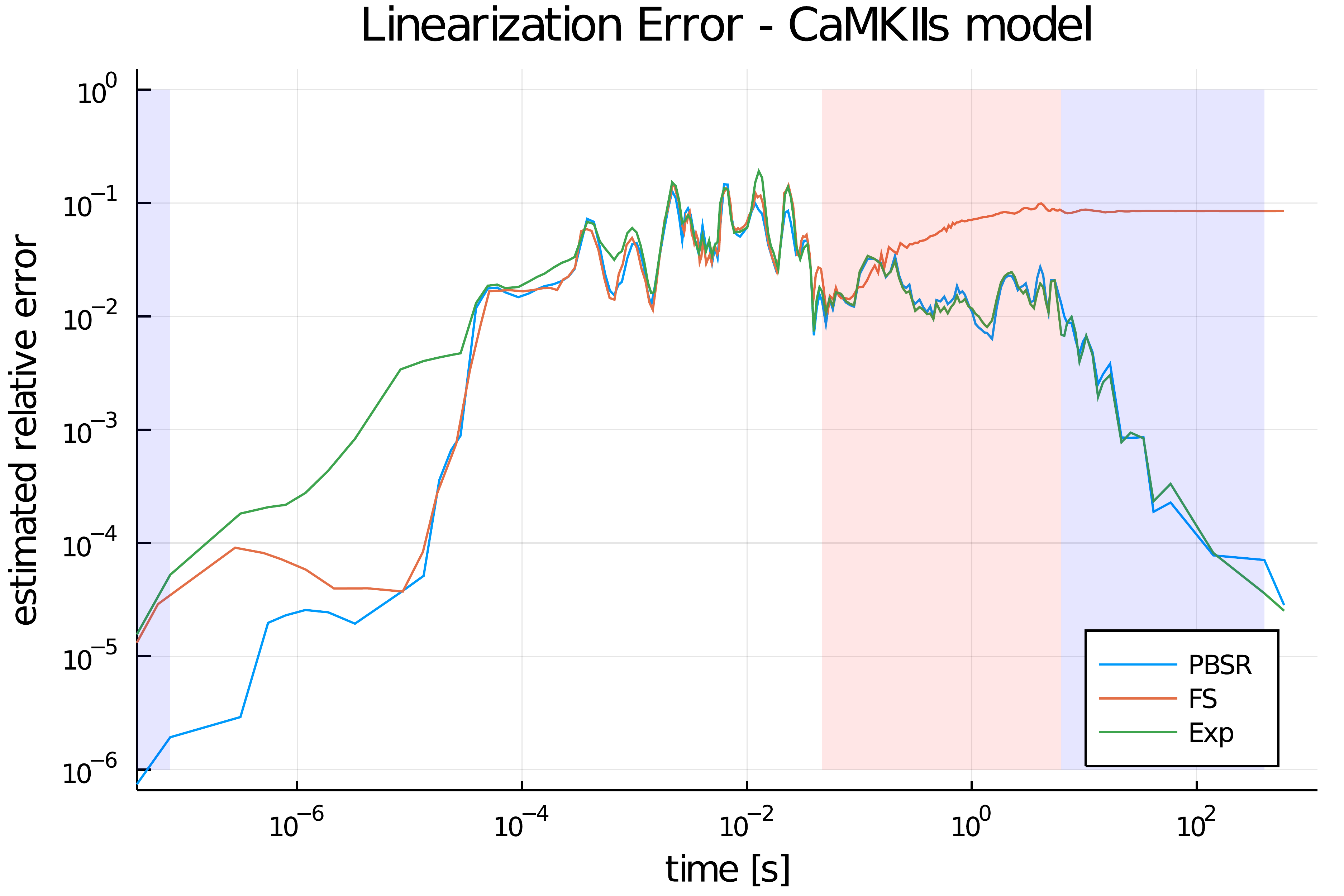}
  \caption{The estimated relative error $\ObjRelErr$ against time for the CaMKMIIs
    model. The simulations are run over the time span $[0,600]$. The blue and red shaded areas
both indicate that within the PBSR method the exponential formula~\eqref{solutionAtEquilibrium} is used: \textcolor{blue}{\emph{blue}} because the system is close to a (stable) steady state, \textcolor{red}{\emph{red}} because refinement is numerically
    prohibitive. The error estimate $\ObjRelErr$ uses small parameter perturbations $\delta_\bfp$ to test how well the sensitivity matrix predicts the changes in the solution (averaged over $N=100$ small, random parameter perturbations).}
  \label{fig:ObjRelErrCaMKIIs}
\end{figure}

As a final test, we compared the computational time for the different algorithms in the PKA and CaMKII models by computing the average runtime (averaged over 100 iterations) for the three methods; the results are presented in Table~\ref{tab:timingPKACaMKII}. As expected, the PBSR shows an increased computational cost compared to the Exp algorithm, whereas compared to the FS method, the PBSR algorithm is faster. The latter becomes more pronounced when the dimension of the system is rather high (CaMKII model vs.\ PKA model).

\begin{table}[ht]
\begin{center}
\begin{minipage}{174pt}
\caption{Average runtime (averaged over $100$ iterations) of FS, PBSR and Exp algorithms applied to the PKA and CaMKII models}
\label{tab:timingPKACaMKII}
\begin{tabular}{@{}llll@{}}%{ c | c c c }
\toprule
  & PKA & CaMKII \\
  \midrule
 FS & $0.056\,s$ & $1.518\,s$ \\
 PBSR & $0.047\,s$ & $0.132\,s$ \\
 Exp & $0.008\,s$ & $0.030\,s$\\
 \botrule
\end{tabular}
\end{minipage}
\end{center}
\end{table}

\subsubsection*{Random linear system}
In order to numerically investigate the impact of an increase in the dimensions $n_x$ and $n_p$ of the underlying ODE system on the algorithms' computational cost, we considered a random linear system for which we can choose an arbitrary dimensions $n_x$; for convenience, we take this to be equal to the parameter space dimension $n_p$ and denote both with $n$.

The model is obtained by first generating a random matrix $\bfB$ of
dimension $n \times n$ and a random vector $\bfp$ of length $n$, both
with values uniformly distributed in the interval $[0,1]$. Next, we
define $\bfA:=-\bfB^\tT\bfB$, and let $\bfp^2$ be the element-wise square
of $\bfp$, and $\bfu$ an $n$-dimensional vector of ones. The associated (random) linear system is defined as
\begin{equation*}
  \dot{\bfx}(t)=\bfA\bfx+\bfp^2+\bfu\,,
\end{equation*}
and we use this system to test how the dimension of the system affects the runtime of the three algorithms (PBSR, Exp and FS).

In this example, the Jacobians $\nabla _x f$ and $\nabla _p f$ are, by construction, constant with respect to the state $\bfx$ and time--$\nabla _x f \equiv \bfA$ and $\nabla _p f = 2\bfp$--and the \texttt{if}-statement \eqref{eq:ifStat} in Algorithm~\ref{alg:PBSR} is always true. Thus, in the PBSR algorithm, there would never be a refinement of the time intervals $[t_k;t_{k+1}]$ and the exponential formula~\eqref{solutionAtEquilibrium} would always be applied, leading to the same output as the Exp algorithm. Therefore, to test the approximation based on the Peano-Baker series, we remove the \texttt{if}-statement and always apply the PBS formula~\eqref{eq:Approx}, including the refinement of the time grid, to approximate the state-transition matrix and compute $\bfS$.

We performed tests of the three algorithms on this type of random linear system of dimensions $n=5,10,\dots,95,100$. In Figure~\ref{fig:TimevsSize} we show the average runtime (over $10$ iterations) of the FS algorithm (solid line), the PBSR algorithm (dashed line) and Exp (dotted line).
To determine how the runtime scales with the dimension $n$ of the system, we perform regressions on the form $\text{runtime} = a\cdot\text{dimension}^b$, using runtime data for each of the three algorithms. The results suggest a runtime of $\mathcal{O}(n^{2.1})$ for the Exp algorithm, $\mathcal{O}(n^{3.7})$ for the PBSR and $\mathcal{O}(n^{4.2})$ for the FS, indicating that the PBSR algorithm gives an improvement of approximately $n ^{0.5}$ compared to the FS method for this particular system.

\begin{figure}
    \centering
    \includegraphics[width=\linewidth]{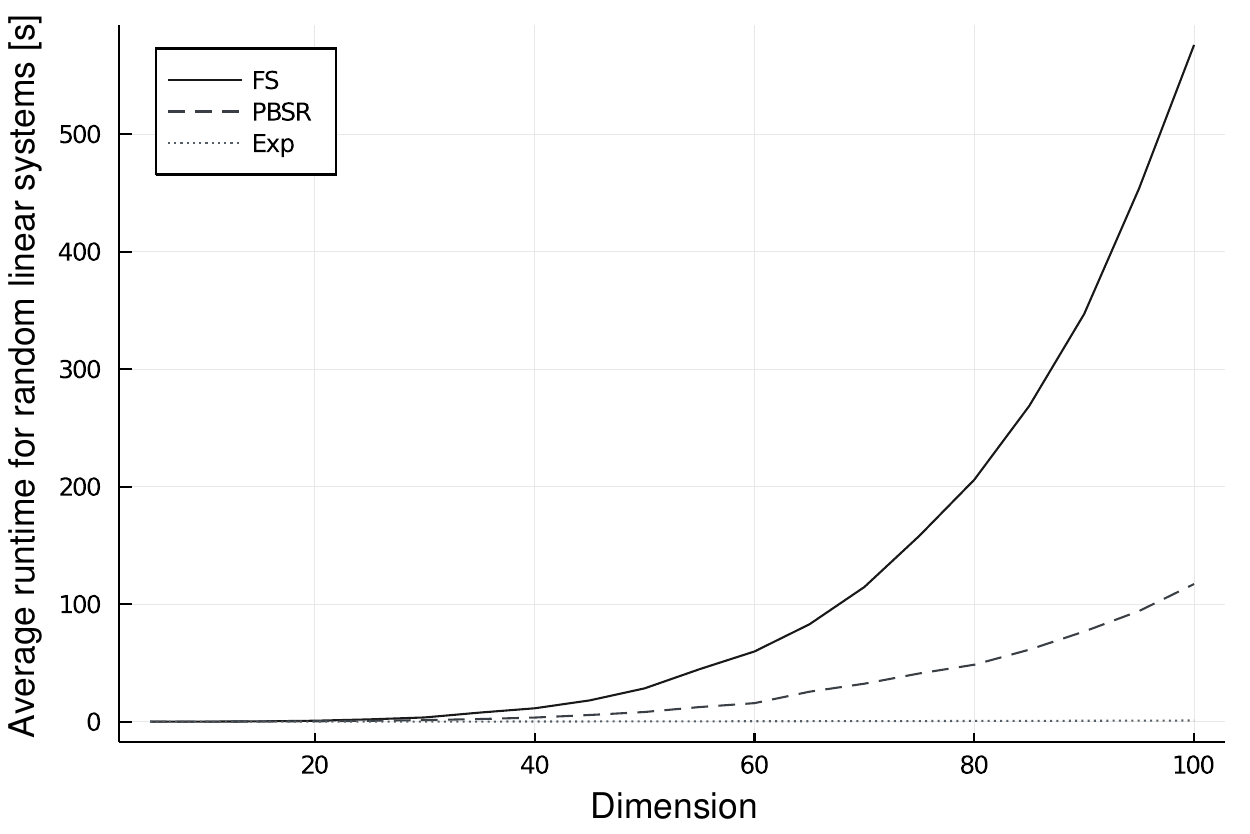}
    \caption{Average running time (over 10 iterations) of the FS (solid line), PBSR (dashed line) and Exp (dotted line) algorithms applied to random linear systems of dimension $5,10,\dots,100$}
    \label{fig:TimevsSize}
\end{figure}

\subsubsection*{Example with a limit cycle: Chua's circuit model.}
As our last example, we 
consider the following three-dimensional dynamical system,
\begin{equation*}
    \begin{cases}
    \dot{x}_1 = p_1(x_2 - x_1 - f(x)),\\
    \dot{x}_2 = x_1 - x_2 + x_3,\\
    \dot{x}_3 = -p_2x_2,
    \end{cases}
\end{equation*}
with $f(x) = -8/7x_1+4/63 x_1^3$, parameters $p_1 = 7$, $p_2 = 15$, and initial condition $\bfx_0 = (0,0,-0.1)^\tT$.
This system models Chua's circuit, an electrical circuit consisting of two capacitors and an inductor, and the choice of parameters and initial condition causes the system to converge to a limit cycle.

In this example, the Jacobians $\nabla_xf$ and $\nabla_pf$ exhibit high variability within time steps $[t_k,t_{k+1}]$. Because the PBS formula~\eqref{eq:Approx} can capture such variations, the PBSR algorithm (Algorithm~\ref{alg:PBSR}) provides an approximation of the sensitivity matrix that is roughly two orders of magnitude more accurate than the Exp algorithm (Algorithm~\ref{alg:ExpAlg}); this is according to both types of error estimates, comparing $\rerr^{\text{PBSR}}_k$ and $\rerr^{\text{Exp}}_k$ (see Figure~\ref{fig:Chua}), and $\ObjRelErr^{\text{PBSR}}_k$ and $\ObjRelErr^{\text{Exp}}_k$ (see Figure~\ref{fig:ObjRelErrChua}), respectively. Moreover, in Figure~\ref{fig:ObjRelErrChua} we note that aside from the initial phase, the FS method has an error that is essentially equivalent to that of the PBSR.
%, within the limit cycle. The FS method is more precise in the transient phase, although at higher computational cost.

\begin{figure}
    \centering
    \includegraphics[width=\linewidth]{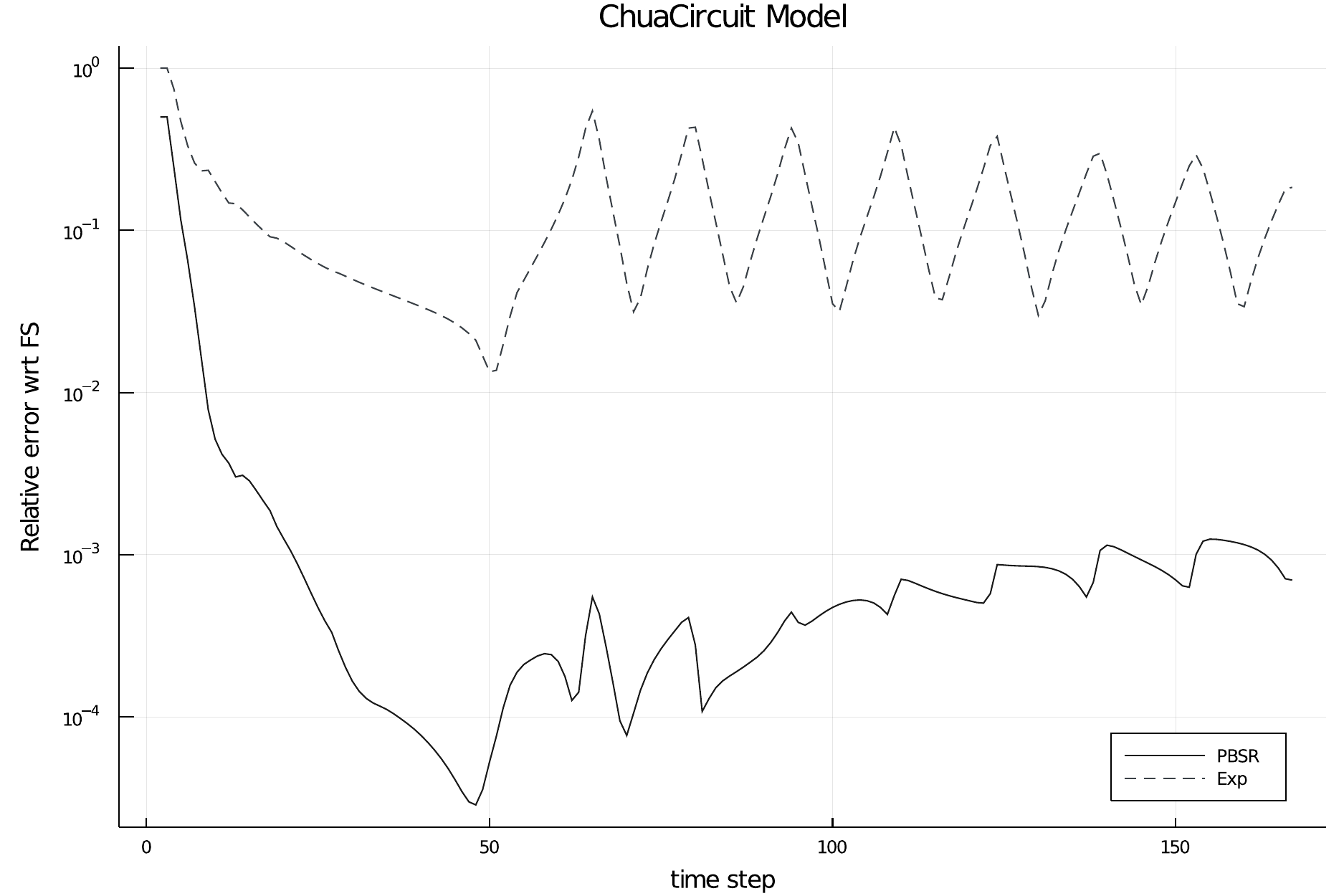}
    \caption{Relative error of the sensitivity matrix, against time step, for the PBSR (solid line) and the Exp (dashed line) algorithms for the Chua's system. The simulations are run over the time span $[0,10]$. 
    }
    \label{fig:Chua}
\end{figure}

\begin{figure}
  \centering
  \includegraphics[width=\textwidth]{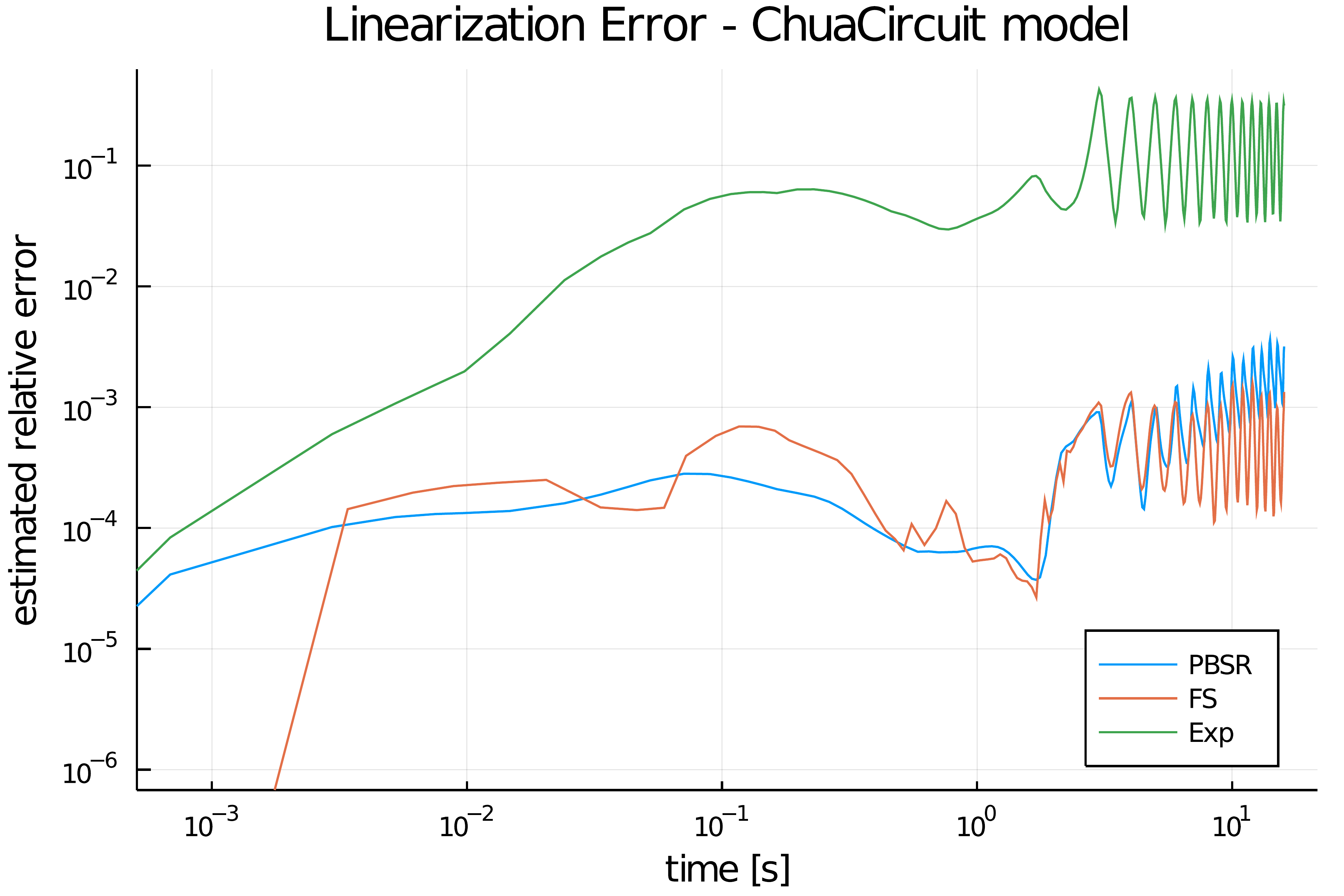}
  \caption{The estimated relative error $\ObjRelErr$ for the Chua's
    Circuit model. The error estimate $\ObjRelErr$ uses small
    parameter perturbations $\delta_{\bfp}$ to test how well the sensitivity
    matrix predicts the changes in the solution (averaged over $N=100$
    small, random parameter perturbations).}
  \label{fig:ObjRelErrChua}
\end{figure}

\section{Conclusion and Future Work}
\label{sec:conclusion}
In this paper we have addressed the problem of computing the sensitivity matrix of parameter-dependent ODE models in high-dimensional settings, where the forward and adjoint methods become too slow for practical purposes. This situation arises in, e.g., uncertainty quantification using Bayesian methods, where there can be a need to compute the sensitivity matrix at a large number of time steps and parameter values, coupled with a high-dimensional parameter space.

We developed a new numerical method based on the Peano-Baker series from control theory, which appears in the exact solution of $\bfS$ (Theorem \ref{thm:generalODESol}). By truncating the series and applying the trapezoidal rule for the integrals involved in the exact formula, we constructed an approximation to $\bfS$ that is suitable for numerical computation, referred to here as the \textit{PBS formula}~\eqref{eq:Approx}. In addition to the general representation of $\bfS$, given in Theorem \ref{thm:generalODESol}, in Corollary~\ref{cor:Sequi} we derived a simplified expression for the solution $\bfS$ in the setting of constant coefficients and forcing term in the ODE for $\bfS$. This led to a second formula to approximate the sensitivity matrix $\bfS$—the \textit{exponential formula}~\eqref{solutionAtEquilibrium}—which is exact when the system is at equilibrium, and a good approximation when the vector field of the ODE system has almost-constant Jacobians $\nabla _x f$ and $\nabla_pf$. Combining these two formulas, we defined the \textit{Peano-Baker series (PBS) algorithm} in Section~\ref{sec:genSolution}. 

As discussed in Section~\ref{sec:stability}, the PBS algorithm may suffer from stability issues, in particular for stiff problems. To overcome this, we define two additional algorithms: the \textit{PBS with refinement (PBSR) algorithm}, where we implement a refinement of the time grid, along with what can be viewed as a stiffness detection mechanism, and the \textit{Exponential (Exp) algorithm}, where only the exponential formula is used (and no refinement is performed). For applications, such as the type of uncertainty quantification and MCMC sampling that originally motivated this work (see Section \ref{subsec:motivation}), these are the two algorithms we propose to use. %for implementation. 

A rigorous error analysis, carried out in Sections~\ref{sec:errEst}-\ref{sec:proofErrEst}, showed that, under standard regularity assumptions, the proposed algorithm, based on the Peano-Baker series (with or without refinement of the time grid), admits a global error of order $\mathcal{O}(\Delta t_{\tmax} ^2)$. The analysis also showed that the proposed method is optimal in the sense of at what term the Peano-Baker series is truncated.

The theoretical error analysis was complemented by several numerical experiments in Section~\ref{sec:numerical}. We compared the performance of the different methods to that of the forward sensitivity method for two ODE models from systems biology, a random linear system and a system modelling Chua's circuit. The results showed that both our algorithms produced accurate approximations of the sensitivity matrix with significant speed-up, that seemed to increase rapidly with the dimension of the problem. In the dynamical system modeling Chua's circuit, the limit trajectory of the ODE was a limit cycle and thus the Jacobians never become (close to) constant. In this example, the PBSR algorithm produced approximations that had an accuracy of two orders of magnitude better than that of the Exp algorithm, motivating the use of the PBSR algorithm for accuracy in general problems. Further motivation for the PBSR comes from the applications in neuroscience that we considered, where ODE models are commonly characterised by time-dependent inputs (e.g. so-called Ca-spike trains). In such systems, the framework is similar to the Chua's circuit example, where the Jacobians are time-variant. Therefore, the PBSR algorithm is expected to be more precise, and thus more useful, than the Exp algorithm.

The methods presented in this paper are expected to be useful in the context of MCMC methods, particularly for problems arising in, e.g., systems biology: the speed-up provided by the PBSR and Exp algorithms, compared to forward sensitivity analysis, has the potential to drastically increase the efficiency of MCMC methods. As a first test of this, we equipped an implementation (in C, using the CVODES solver) of the simplified manifold Metropolis-adjusted Langevin algorithm (\textsc{smmala}, see, e.g., \cite{GiroCal}) with the Exp algorithm to compute the Fisher Information and posterior gradients, and compared it to \textsc{smmala} with conventional forward sensitivity analysis. The (real) data used for MCMC was obtained at or near the steady-state of the system, so the use of the Exp algorithm was suitable. The near-steady-state sensitivity approximation enabled sampling approximately $100$ times faster from the posterior distribution, compared to using CVODES' forward sensitivity analysis\footnote{With large samples sizes, a speed-up of that magnitude is difficult to test thoroughly, due to the time required by the slower sampler.}. Similar tests for the PBSR algorithm will be part of future work on the integration of the proposed methods in the MCMC setting.

In addition to applications to MCMC sampling, future work includes further investigation of the impact of different implementation aspects of the PBSR algorithm—e.g.\ the criteria (based on the Jacobians $\nabla_x f$ and $\nabla_pf$) to switch from PBSR to exponential formula, and the refinement of the time grid of the ODE solver. Additional comparisons with existing methods and a better understanding of the methods performance, particularly in large-scale systems, is another important direction.

\backmatter

\bmhead{Supplementary information}

\bmhead{Acknowledgments}
We sincerely thank the referees for their thorough and insightful feedback regarding the first version of the paper. Their comments and questions have lead to significant improvements, both in terms of presentation and content. The list of changes inspired by their reports is too long to include here, however we highlight the discussion about stability issues throughout the paper, and the error analysis covering the case with refined time intervals (see Corollary \ref{cor:errorPBSR}); both additions are direct results of the feedback from the referees. %

The research was supported by Swedish e-Science Research Centre (SeRC) and Science for Life Laboratory. The research of AK and OE was supported by EU/Horizon 2020 no. 945539 (HBP SGA3). The research of AK was also supported by the EU/Horizon 2020 no. 101147319 (EBRAINS 2.0 Project), European Union's Research and Innovation Program Horizon Europe no 101137289 (the Virtual Brain
Twin Project) and the Swedish Research Council (VR-M-2020-01652). The research of PN was supported in part by the Swedish Research Council (VR-2018-07050, VR-2023-03484) and the Wallenberg AI, Autonomous Systems and Software Program (WASP) funded by the Knut and Alice Wallenberg Foundation.

The computations were enabled by resources provided by the Swedish National Infrastructure for Computing (NAISS) at the PDC Center for High Performance Computing, KTH Royal Institute of Technology, partially funded by the Swedish Research Council through grant agreement no. 2022-06725 and by LUNARC, The Centre for Scientific and Technical Computing at Lund University.

\bmhead{Authors contribution}
AK, FM, and PN conceived the study. OE and AK contributed with the application. AK developed the exponential algorithm and FM the PBS algorithm. FM performed, and PN supervised the theoretical analysis. AK and FM performed the numerical experiments. FM, AK and PN wrote the manuscript with contributions from OE. All authors approved the final manuscript.

\section*{Declarations}
\bmhead{Statements and Declarations} %[from the guidelines]

 The authors declare no conflicts of interest. The funding is disclosed in the Acknowledgment section. This numerical study uses free software, the code and numerical results have been made available through publicly accessible repositories.

 \textbf{Competing Interests:} None
% all other points from the list below seem to be covered by the submission rules, as a given:
% \begin{itemize}
% \item Funding
% \item Conflict of interest/Competing interests (check journal-specific guidelines for which heading to use)
% \item Ethics approval
% \item Consent to participate
% \item Consent for publication
% \item Availability of data and materials
% \item Code availability
% \item Authors' contributions
% \end{itemize}

%%===========================================================================================%%
%% If you are submitting to one of the Nature Portfolio journals, using the eJP submission   %%
%% system, please include the references within the manuscript file itself. You may do this  %%
%% by copying the reference list from your .bbl file, paste it into the main manuscript .tex %%
%% file, and delete the associated \verb+\bibliography+ commands.                            %%
%%===========================================================================================%%

\bibliography{sn-bibliography}
%\printbibliography% common bib file
%% if required, the content of .bbl file can be included here once bbl is generated
%%\input sn-article.bbl

%% Default %%
%%\input sn-sample-bib.tex%

\end{document}